\documentclass[a4paper,12pt]{article}
\usepackage{tikz}
\usetikzlibrary{matrix,arrows,decorations.markings}
\usepackage{amsmath,amsfonts,amssymb,amsthm,url,textcomp}
\usepackage{csquotes}
\usepackage{a4wide}
\usepackage[numbers]{natbib}
\usepackage{fancyhdr}
\usepackage{anyfontsize}
\usepackage{hyperref}
\usepackage{hhline}
\usepackage{leftidx}
\usepackage{enumitem}
\usepackage{mdframed}
\usepackage[ruled,vlined]{algorithm2e}
\usepackage{makecell}
\usepackage{diagbox}

\allowdisplaybreaks

\makeatletter
\let\@fnsymbol\@arabic
\makeatother

\newtheorem{theorem}{Theorem}\numberwithin{theorem}{section}
\newtheorem*{proposition*}{Proposition}
\newtheorem{lemma}[theorem]{Lemma}

\newtheorem{proposition}[theorem]{Proposition}

\newtheorem{theoremm}{Theorem}\numberwithin{theoremm}{subsection}
\newtheorem{lemmma}[theoremm]{Lemma}
\newtheorem{corrollary}[theoremm]{Corollary}
\newtheorem{propposition}[theoremm]{Proposition}

\numberwithin{theoremmm}{subsubsection}

\theoremstyle{definition}
\newtheorem*{definition*}{Definition}
\newtheorem{definition}[theorem]{Definition}
\newtheorem{deffinition}[theoremm]{Definition}

\newtheorem{nottation}[theoremm]{Notation}

\newtheorem{example}[theorem]{Example}
\newtheorem*{example*}{Example}

\newtheorem{exxample}[theoremm]{Example}

\newcommand{\Aut}{\operatorname{Aut}}

\newcommand{\lcm}{\operatorname{lcm}}

\newcommand{\ord}{\operatorname{ord}}
\newcommand{\Sym}{\operatorname{Sym}}
\newcommand{\Hol}{\operatorname{Hol}}

\newcommand{\id}{\operatorname{id}}

\newcommand{\Mod}[1]{\ (\textup{mod}\ #1)}

\newcommand{\IN}{\mathbb{N}}

\newcommand{\IF}{\mathbb{F}}

\newcommand{\IZ}{\mathbb{Z}}

\newcommand{\IP}{\mathbb{P}}

\newcommand{\CT}{\operatorname{CT}}

\newcommand{\omicron}{\operatorname{o}}

\newcommand{\lex}{\mathrm{lex}}

\newcommand{\rad}{\operatorname{rad}}

\newcommand{\imp}{\mathrm{imp}}
\newcommand{\fcpc}{\operatorname{fcpc}}

\newcommand{\CP}{\operatorname{CP}}
\newcommand{\GCP}{\operatorname{GCP}}
\newcommand{\FOCP}{\operatorname{FOCP}}
\newcommand{\aord}{\operatorname{aord}}
\newcommand{\reg}{\mathrm{reg}}
\newcommand{\pow}{\operatorname{pow}}
\newcommand{\IQ}{\mathbb{Q}}
\newcommand{\CI}{\operatorname{CI}}
\newcommand{\fcp}{\operatorname{fcp}}
\newcommand{\CC}{\operatorname{CC}}

\begin{document}

\title{Generalized cyclotomic mappings: Switching between polynomial, cyclotomic, and wreath product form}

\author{Alexander Bors\textsuperscript{1} \and Qiang Wang\thanks{School of Mathematics and Statistics, Carleton University, 1125 Colonel By Drive, Ottawa ON K1S 5B6, Canada. \newline First author's e-mail: \href{mailto:alexanderbors@cunet.carleton.ca}{alexanderbors@cunet.carleton.ca} \newline Second author's e-mail: \href{mailto:wang@math.carleton.ca}{wang@math.carleton.ca} \newline The authors are supported by the Natural Sciences and Engineering Research Council of Canada (RGPIN-2017-06410). \newline 2020 \emph{Mathematics Subject Classification}: Primary: 11T22. Secondary: 11A07, 11T06, 20E22. \newline \emph{Keywords and phrases}: Finite fields, Cyclotomy, Cyclotomic mappings, Permutation polynomials, Wreath product, Cycle structure, Involution.}}

\date{\today}

\maketitle

\abstract{This paper is concerned with so-called index $d$ generalized cyclotomic mappings of a finite field $\IF_q$, which are functions $\IF_q\rightarrow\IF_q$ that agree with a suitable monomial function $x\mapsto ax^r$ on each coset of the index $d$ subgroup of $\IF_q^{\ast}$. We discuss two important rewriting procedures in the context of generalized cyclotomic mappings and present applications thereof that concern index $d$ generalized cyclotomic permutations of $\IF_q$ and pertain to cycle structures, the classification of $(q-1)$-cycles and involutions, as well as inversion.}

\section{Introduction}\label{sec1}

\subsection{Background and main results}\label{subsec1P1}

Let $q$ be a prime power, let $\omega$ be a primitive root of $\IF_q$, let $d$ be a positive integer with $d\mid q-1$, and let $C$ be the (unique) index $d$ subgroup of $\IF_q^{\ast}$. Note that the cosets of $C$ in $\IF_q^{\ast}$ are of the form $C_i:=\omega^iC$ for $i=0,1,\ldots,d-1$. Index $d$ cyclotomic and generalized cyclotomic mappings of $\IF_q$ are an interesting yet still relatively well-controlled generalization of monomial mappings (functions $\IF_q\rightarrow\IF_q$ of the form $x\mapsto ax^r$ for fixed $a\in\IF_q$ and $r\in\IN$). They are studied in varying degrees of generality. A \emph{generalized cyclotomic mapping of $\IF_q$ of index $d$} is a function $\IF_q\rightarrow\IF_q$ of the form
\begin{equation}\label{cyclotomicFormEq}
f_{\omega}(\vec{a},\vec{r}):x\mapsto\begin{cases}0, & \text{if }x=0, \\ a_ix^{r_i}, & \text{if }x\in C_i,i\in\{0,1,\ldots,d-1\}\end{cases}
\end{equation}
for fixed $\vec{a}=(a_0,a_1,\ldots,a_{d-1})\in\IF_q^d$ and $\vec{r}=(r_0,r_1,\ldots,r_{d-1})\in\{1,\ldots,\frac{q-1}{d}\}^d$, and we call this representation of the function an \emph{$\omega$-cyclotomic form} of it. If $\vec{a'}=(a'_0,\ldots,a'_{d-1})\in\IF_q^d$ and $\vec{r'}=(r'_0,\ldots,r'_{d-1})\in\{1,2\ldots,\frac{q-1}{d}\}^d$ are other choices of such tuples, then $f_{\omega}(\vec{a},\vec{r})=f_{\omega}(\vec{a'},\vec{r'})$ if and only if $\vec{a}=\vec{a'}$ and $r_i=r'_i$ for all $i\in\{0,\ldots,d-1\}$ such that $a_i=a'_i\not=0$.

In case the entries of $\vec{r}$ are all equal, say to an integer $r$, the function $f_{\omega}(\vec{a},\vec{r})$ is called an \emph{$r$-th order cyclotomic mapping of $\IF_q$ of index $d$}. We note that the notions of a generalized cyclotomic mapping of $\IF_q$ of index $d$ and, for each fixed $r$, of an $r$-th order cyclotomic mapping of $\IF_q$ of index $d$, are independent of the choice of $\omega$. For more details, see \cite{Wan07a}.

In this paper, we are concerned with generalized cyclotomic mappings that are permutations of $\IF_q$, or \emph{generalized cyclotomic permutations of $\IF_q$} for short. Our two main results, given below as Theorems \ref{wreathIsoTheo} and \ref{cycloPolyTheo}, provide ways to rewrite the cyclotomic form (\ref{cyclotomicFormEq}) of such a permutation into two other important forms.

For the first rewriting method, which is the subject of Theorem \ref{wreathIsoTheo}, we observe that the generalized cyclotomic permutations of $\IF_q$ of index $d$ form a permutation group on $\IF_q$. As $0$ is a fixed point of every generalized cyclotomic permutation of $\IF_q$, only the behavior on $\IF_q^{\ast}$ is of interest, and so we will consider the permutation group on $\IF_q^{\ast}$ that consists of all restrictions to $\IF_q^{\ast}$ of generalized cyclotomic permutations of $\IF_q$ -- this permutation group will be denoted by $\GCP(d,q)$. The rewriting method of Theorem \ref{wreathIsoTheo} is a permutation group isomorphism between $\GCP(d,q)$ and a certain imprimitive permutational wreath product, described below. Since wreath products are well-understood, this allows one to study certain aspects of generalized cyclotomic permutations, such as cycle structure and conjugacy classes, using well-known results (dating back to a 1937 paper of P{\'o}lya, \cite{Pol37a}), and we discuss such applications in Subsections \ref{subsec5P1} to \ref{subsec5P2}. In this context, we note that various authors have studied the cycle structures of elements of certain classes of permutation polynomials over finite fields, see for example \cite{Ahm69a,CMT08a,LM91a,SSP12a}.

For the second rewriting method, recall that every function $f:\IF_q\rightarrow\IF_q$ admits a unique polynomial $P\in\IF_q[T]$ of degree at most $q-1$ such that $f(x)=P(x)$ for all $x\in\IF_q$. We call $P$ the \emph{polynomial form of $f$}. Polynomial forms of first-order cyclotomic mappings were determined by Niederreiter and Winterhof in \cite[Theorem 1]{NW05a}. The second author exhibited the polynomial forms of cyclotomic mappings of any fixed order in \cite[Lemma 1]{Wan07a}, and the permutation polynomials among those forms (i.e., those polynomials that correspond to cyclotomic permutations of a fixed order) have been intensely studied (see \cite{AGW09a,AW07a,PL01a,WL91a,Wan18a,Zie09a}, among others). In a more recent paper, the second author gave a formula for the polynomial form of an arbitrary generalized cyclotomic mapping of $\IF_q$, see \cite[Formula (3)]{Wan13a}. This formula is stated as one half of our second main result, Theorem \ref{cycloPolyTheo}, whereas the other half is new and provides an effective method to decide if a given polynomial of degree at most $q-1$ over $\IF_q$ is the polynomial form of an index $d$ generalized cyclotomic mapping of $\IF_q$ and, if so, output a cyclotomic form of it. As an application of this and Zheng-Yu-Zhang-Pei's inversion formula from \cite[Theorem 3.3]{ZYZP16a}, one can give a simple algorithm to pass from the polynomial form of a generalized index $d$ cyclotomic permutation $f$ of $\IF_q$ to the polynomial form of $f^{-1}$, see Subsection \ref{subsec5P3}.

Before stating the first main theorem, Theorem \ref{wreathIsoTheo}, we discuss some group-theoretic preliminaries. As is customary in group theory, we will work with right group actions in this paper. This means that in the symmetric group $\Sym(\Omega)$, where $\Omega$ is an arbitrary set, the product $\sigma\psi$ for $\sigma,\psi\in\Sym(\Omega)$ is the permutation that consists in applying first $\sigma$, then $\psi$ (i.e., $\sigma\psi=\psi\circ\sigma$ if $\circ$ is the usual composition of functions). Accordingly, it is more natural to use exponent notation for function values ($x^f$ instead of $f(x)$), since $\omega^{\sigma\psi}=(\omega^{\sigma})^{\psi}$ for $\omega\in\Omega$ and $\sigma,\psi\in\Sym(\Omega)$, as opposed to $(\sigma\psi)(\omega)=\psi(\sigma(\omega))$. We will still use the notation $f(x)$ in cases where we feel that it increases readability. Usually, for a positive integer $d$, the notation $\Sym(d)$ denotes the symmetric group on the set $\{1,2,\ldots,d\}$, but in this paper, it will be more convenient to define it as the symmetric group on $\{0,1,\ldots,d-1\}$.

We remind the reader of the following definition of an isomorphism of permutation groups (see Theorem \ref{wreathIsoTheo} for an occurrence of this concept in our paper):

\begin{deffinition}\label{permGroupIsoDef}
Let $G\leq\Sym(\Omega)$ and $H\leq\Sym(\Lambda)$ be permutation groups. An \emph{isomorphism of permutation groups $G\rightarrow H$} is a pair $(\beta,\iota)$ where
\begin{enumerate}
\item $\beta$ is a bijection $\Omega\rightarrow\Lambda$ and
\item $\iota$ is an abstract group isomorphism $G\rightarrow H$ such that for all $g\in G$ and all $\omega\in\Omega$, we have $\beta(\omega^g)=\beta(\omega)^{\iota(g)}$.
\end{enumerate}
\end{deffinition}

Wreath products are a fundamental concept in the theory of permutation groups. They were used by Wan and Lidl in \cite{WL91a} to study cyclotomic permutations, and in this paper, we will expand on their ideas. We will only need so-called \emph{imprimitive} wreath products, which are defined as follows:

\begin{deffinition}\label{impWreathDef}
Let $G$ be an abstract group, and let $P\leq\Sym(d)$ be a permutation group.
\begin{enumerate}
\item The \emph{(abstract) wreath product of $G$ and $P$}, written $G\wr P$, is the abstract group that can be defined as the external semidirect product $P\ltimes G^d$ where $P$ acts on $G^d$ by \enquote{coordinate permutations}. More explicitly, the elements of $G\wr P$ are ordered pairs of the form $(\sigma,\vec{g})$ with $\sigma\in P$ and $\vec{g}=(g_0,g_1,\ldots,g_{d-1})\in G^d$, and these elements are multiplied as follows:
\begin{align*}
&(\sigma,(g_0,g_1,\ldots,g_{d-1}))\cdot(\psi,(h_0,h_1,\ldots,h_{d-1})):= \\
&(\sigma\psi,(g_{\psi^{-1}(0)}h_0,g_{\psi^{-1}(1)}h_1,\ldots,g_{\psi^{-1}(d-1)}h_{d-1})).
\end{align*}
\item If $G\leq\Sym(\Omega)$ is a permutation group, then the abstract group $G\wr P$ is isomorphic to a certain permutation group on $\Omega\times\{0,1,\ldots,d-1\}$ that is called the \emph{imprimitive (permutational) wreath product of $G$ and $P$} and will be denoted by $G\wr_{\imp}P$. More explicitly, the function $G\wr P\rightarrow\Sym(\Omega\times\{0,1,\ldots,d-1\})$ that maps $(\sigma,(g_0,g_1,\ldots,g_{d-1}))\in G\wr P$ to the permutation
\[
(\omega,i)\mapsto(g_{\sigma(i)}(\omega),\sigma(i))
\]
is an isomorphism of abstract groups between $G\wr P$ and $G\wr_{\imp}P$.
\end{enumerate}
\end{deffinition}

As is customary, we will identify the abstract group $G\wr P$ with the permutation group $G\wr_{\imp}P$ via the isomorphism from Definition \ref{impWreathDef}(2). That is, we will write the elements of $G\wr_{\imp}P$ as pairs $(\sigma,\vec{g})$ but think of the associated permutations on $\Omega\times\{0,1,\ldots,d-1\}$.

The idea behind the definition of the imprimitive wreath product $G\wr P$ is that $\Omega\times\{0,1,\ldots,d-1\}$ is a disjoint union of $d$ copies of $\Omega$, indexed by the elements of $\{0,1,\ldots,d-1\}$, on which $G\wr_{\imp}P$ acts by compositions of
\begin{itemize}
\item permutations of the copies that match corresponding elements of $\Omega$ in different copies: $(\omega,i)^{(\sigma,(1_G,\ldots,1_G))}=(\omega,i^{\sigma})$; and
\item copy-wise applications of elements of $G$: $(\omega,i)^{(\id,(g_0,g_1,\ldots,g_{d-1}))}=(g_i(\omega),i)$.
\end{itemize}
The same behavior is exhibited by the elements of $\GCP(d,q)$ when viewing $\IF_q^{\ast}$ as a disjoint union of the $d$ cosets of the index $d$ subgroup $C$ of $\IF_q^{\ast}$ (and each coset being viewed as a copy of $C$). The role of $G$ is taken by permutations of the multiplicative group $C$ of the form $\lambda(r,a):c\mapsto ac^r$ where $a\in C$ and $r$ is an integer that is coprime to $|C|$ -- these permutations form a permutation group on $C$ which we denote by $\Hol(C)$ (see Section \ref{sec3} for more details on this group-theoretic notation). This observation leads to the main statement of Theorem \ref{wreathIsoTheo} below. The theorem also contains statements on index $d$ cyclotomic permutations of $\IF_q$ (of all possible orders) as well as the index $d$ first-order cyclotomic permutations of $\IF_q$, which form subgroups of the group of generalized index $d$ cyclotomic permutations of $\IF_q$. We will denote the corresponding subgroups of $\GCP(d,q)$ (consisting of their restrictions to $\IF_q^{\ast}$) by $\CP(d,q)$ and $\FOCP(d,q)$ respectively and are now ready to state Theorem \ref{wreathIsoTheo}:

\begin{theoremm}\label{wreathIsoTheo}(Switching between cyclotomic form and wreath product form)
Let $q$ be a prime power, let $d$ be a divisor of $q-1$, let $C$ be the index $d$ subgroup of $\IF_q^{\ast}$, and let $\omega$ be a primitive root of $\IF_q$. Then $\GCP(d,q)$ is isomorphic, as a permutation group, to $\Hol(C)\wr_{\imp}\Sym(d)\leq\Sym(C\times\{0,1,\ldots,d-1\})$. More precisely, consider
\begin{itemize}
\item the bijection $\beta_{\omega}:C\times\{0,1,\ldots,d-1\}\rightarrow\IF_q^{\ast}$ that maps $(c,i)$ to $c\omega^i$ for $c\in C$ and $i\in\{0,1,\ldots,d-1\}$; and
\item the function $\iota_{\omega}:\Hol(C)\wr_{\imp}\Sym(d)\rightarrow\GCP(d,q)$ that maps the wreath product element $(\psi,(\lambda(s_i,b_i))_{i=0,1,\ldots,d-1})$ to
\[
f_{\omega}((\omega^{\psi(i)-is_{\psi(i)}}b_{\psi(i)})_{i=0,1,\ldots,d-1},(s_{\psi(i)})_{i=0,1,\ldots,d-1})_{\mid\IF_q^{\ast}}.
\]
\end{itemize}
The pair $(\beta_{\omega},\iota_{\omega})$ is an isomorphism of permutation groups $\Hol(C)\wr_{\imp}\Sym(d)\rightarrow\GCP(d,q)$. The inverse group isomorphism of $\iota_{\omega}$ is given by the formula
\begin{align*}
&(f_{\omega}((a_i)_{i=0,1,\ldots,d-1},(r_i)_{i=0,1,\ldots,d-1}))_{\mid\IF_q^{\ast}} \\
&\mapsto \\
&(\psi,(\lambda(r_{\psi^{-1}(i)},\omega^{r_{\psi^{-1}(i)}\psi^{-1}(i)-i}a_{\psi^{-1}(i)}))_{i=0,1,\ldots,d-1}),
\end{align*}
where $\psi=\psi_{\omega}(f_{\omega}(\vec{a},\vec{r}))$ is the unique permutation of the set $\{0,1,\ldots,d-1\}$ such that for all $i=0,1,\ldots,d-1$, the image of the coset $C_i=\omega^iC$ under $f_{\omega}(\vec{a},\vec{r})$ is the coset $C_{\psi(i)}$. Moreover, the following hold:
\begin{enumerate}
\item The pre-image under $\iota_{\omega}$ of $\CP(d,q)$ is the subgroup of $\Hol(C)\wr_{\imp}\Sym(d)$ that consists of all elements of the form
\[
(\psi,(\lambda(s,b_i))_{i=0,1,\ldots,d-1})
\]
where $\psi\in\Sym(d)$, $s$ is an integer coprime to $|C|$, and $b_0,b_1,\ldots,b_{d-1}\in C$.
\item The pre-image under $\iota_{\omega}$ of $\FOCP(d,q)$ is the subgroup of $\Hol(C)\wr_{\imp}\Sym(d)$ that consists of all elements of the form
\[
(\psi,(\lambda(1,b_i))_{i=0,1,\ldots,d-1})
\]
where $\psi\in\Sym(d)$ and $b_0,b_1,\ldots,b_{d-1}\in C$. In other words, it is the subgroup $C_{\reg}\wr_{\imp}\Sym(d)$, where $C_{\reg}\leq\Hol(C)$ is the image of the injective group homomorphism $C\rightarrow\Hol(C),c\mapsto\lambda(1,c)$.
\end{enumerate}
\end{theoremm}

The following is our second main theorem:

\begin{theoremm}\label{cycloPolyTheo}(Switching between cyclotomic form and polynomial form)
Let $q$ be a prime power, let $d$ be a divisor of $q-1$, let $C$ be the index $d$ subgroup of $\IF_q^{\ast}$, let $\omega$ be a primitive root of $\IF_q$, and let $\zeta:=\omega^{(q-1)/d}$ (a primitive $d$-th root of unity in $\IF_q$). The following hold:
\begin{enumerate}
\item For all $\vec{a}=(a_0,\ldots,a_{d-1})\in\IF_q^d$ and all $\vec{r}=(r_0,\ldots,r_{d-1})\in\{1,2,\ldots,\frac{q-1}{d}\}^d$, the polynomial form of the index $d$ generalized cyclotomic mapping $f_{\omega}(\vec{a},\vec{r})$ of $\IF_q$ is
\begin{equation}\label{wangEq}
\frac{1}{d}\sum_{i,j=0}^{d-1}{\zeta^{-ij}a_iT^{j\cdot\frac{q-1}{d}+r_i}}.
\end{equation}
\item Algorithm \ref{algo1} below decides if a given polynomial $P\in\IF_q[T]$ with $\deg{P}\leq q-1$ is the polynomial form of an index $d$ generalized cyclotomic mapping $f$ of $\IF_q$ and, if so, outputs an $\omega$-cyclotomic form of $f$.
\end{enumerate}
\end{theoremm}

We note a few special terminologies and notations that are used in the formulation of Algorithm \ref{algo1}:
\begin{itemize}
\item The terminology \emph{term of $P$} denotes any polynomial of the form $cT^k$ where $c$ is the coefficient of $T^k$ in $P$ and $c\not=0$. For example, the polynomial $T^5-T^3+1\in\IF_3[T]$ has the three terms $T^5$, $-T^3$ and $1$.
\item The \emph{term degrees of $P$} are the polynomial degrees of the terms of $P$.
\item If $m$ is a positive integer and $n$ is any integer, then the \emph{remainder of $n$ upon division by $m$} is defined as the unique $r\in\{1,2,\ldots,m\}$ such that $n\equiv r\Mod{m}$. This differs from the standard definition of this concept, where the range for $r$ would be $\{0,1,\ldots,m-1\}$.
\item We enumerate the entries of a list or vector starting from $0$. For example, in the length $2$ column vector
\[
{3 \choose 5},
\]
we call $3$ the $0$-th entry, and $5$ the $1$-st entry.
\item For $x_0,x_1,\ldots,x_{d-1}\in\IF_q$, we denote by
\[
V(x_0,x_1,\ldots,x_{d-1})=\begin{pmatrix}1 & 1 & \cdots & 1 \\ x_0 & x_1 & \cdots & x_{d-1} \\ x_0^2 & x_1^2 & \cdots & x_{d-1}^2 \\ \vdots & \vdots & \vdots & \vdots \\ x_0^{d-1} & x_1^{d-1} & \cdots & x_{d-1}^{d-1}\end{pmatrix}
\]
the $(d\times d)$-Vandermonde matrix associated with $x_0,x_1,\ldots,x_{d-1}$.
\end{itemize}

\begin{algorithm}
\SetKwInOut{Input}{input}\SetKwInOut{Output}{output}

\Input{A polynomial $P\in\IF_q[T]$ with $\deg{P}\leq q-1$.}
\Output{The information whether $P$ is the polynomial form of a generalized index $d$ cyclotomic mapping $f$ of $\IF_q$ and, if so, an $\omega$-cyclotomic form of $f$.}
\BlankLine
\nl If the constant term in $P$ is nonzero, output \enquote{$P$ is not the polynomial form of any index $d$ generalized cyclotomic mapping of $\IF_q$.}, and halt.

\nl If $P=0$, let $\vec{a}$ resp.~$\vec{r}$ be the unique element of $\{0_{\IF_q}\}^d$ resp.~$\{1_{\IZ}\}^d$, output \enquote{$P$ is the polynomial form of the index $d$ generalized cyclotomic mapping $f_{\omega}(\vec{a},\vec{r})$ of $\IF_q$.}, and halt.

\nl If the number of terms of $P$ is larger than $d^2$, output \enquote{$P$ is not the polynomial form of any index $d$ generalized cyclotomic mapping of $\IF_q$.}, and halt.

\nl Let $\rho_0<\rho_1<\cdots<\rho_{k-1}$ denote the distinct remainders (in the non-standard sense defined above) of the integer divisions of the term degrees of $P$ by $\frac{q-1}{d}$. If $k>d$, output \enquote{$P$ is not the polynomial form of any index $d$ generalized cyclotomic mapping of $\IF_q$.}, and halt.

\nl For $\ell=0,1,\ldots,k-1$, let $\vec{v}_{\ell}$ denote the length $d$ column vector whose $j$-th entry is the coefficient of $T^{j\frac{q-1}{d}+\rho_{\ell}}$ in $P$ for $j=0,1,\ldots,d-1$.

\nl For $\ell=0,1,\ldots,k-1$, set $\vec{b}_{\ell}:=d\cdot V(1,\zeta^{-1},\zeta^{-2},\ldots,\zeta^{-(d-1)})^{-1}\cdot\vec{v}_{\ell}$. Write
\[
\vec{b}_{\ell}=\begin{pmatrix}b_0^{(\ell)} \\ b_1^{(\ell)} \\ \vdots \\ b_{d-1}^{(\ell)}\end{pmatrix}.
\]

\nl For $\ell=0,1,\ldots,k-1$, define
\[
S_{\ell}:=\begin{cases}\{i\in\{0,\ldots,d-1\}:b_i^{(0)}\not=0,\text{ or }b_i^{(\ell')}=0\text{ for all }\ell'=0,\ldots,k-1\}, & \text{if }\ell=0, \\ \{i\in\{0,\ldots,d-1\}:b_i^{(\ell)}\not=0\}, & \text{if }\ell>0.\end{cases}
\]

\nl If $(S_{\ell})_{\ell=0,1,\ldots,k-1}$ is not a partition of $\{0,1,\ldots,d-1\}$, output \enquote{$P$ is not the polynomial form of any index $d$ generalized cyclotomic mapping of $\IF_q$.}, and halt.

\nl For $i=0,1,\ldots,d-1$, let $\ell(i)$ be the unique index in $\{0,1,\ldots,k-1\}$ such that $i\in S_{\ell(i)}$, and set $a_i$ to be the $i$-th entry of $\vec{b}_{\ell(i)}$ and $r_i:=\rho_{\ell(i)}$.

\nl Set $\vec{a}:=(a_0,a_1,\ldots,a_{d-1})$ and $\vec{r}:=(r_0,r_1,\ldots,r_{d-1})$.

\nl Output \enquote{$P$ is the polynomial form of the index $d$ generalized cyclotomic mapping of $\IF_q$ with $\omega$-cyclotomic form $f_{\omega}(\vec{a},\vec{r})$.}.

\caption{Conversion from polynomial form to cyclotomic form.}\label{algo1}
\end{algorithm}

It is an interesting observation that for index $1$ (generalized) cyclotomic mappings of $\IF_q$, which are of the form $x\mapsto ax$ for some fixed $a\in\IF_q$ and all $x\in\IF_q$, composing them with functions of the form $x\mapsto x+b$ yields the class of affine functions $x\mapsto ax+b$ on $\IF_q$, which is also closed under function composition. This property is lost for $d>1$, but one may still consider functions $\IF_q\rightarrow\IF_q$ of the form
\begin{equation}\label{gengenEq}
x\mapsto\begin{cases}b, & \text{if }x=0, \\ a_ix^{r_i}+b, & \text{if }x\in C_i\text{ for some }i\in\{0,1,\ldots,d-1\}.\end{cases}
\end{equation}
The polynomial forms of such functions are of the form $Q=P+b$ where $P\in\IF_q[T]$ is the polynomial form of an index $d$ generalized cyclotomic mapping of $\IF_q$, and $b\in\IF_q$. Applying Algorithm \ref{algo1} to the non-constant part $P=Q-Q(0)$ of such a polynomial $Q$ allows one to check whether the function $f$ represented by $Q$ is of the form (\ref{gengenEq}) and, if so, find such a representation of $f$.

\subsection{More notation and terminology}\label{subsec1P2}

We denote by $\IN$ the set of natural numbers (including $0$), by $\IN^+$ the set of positive integers, and by $\IP$ the set of primes. The identity function on a set $X$ is denoted by $\id_X$. For a prime $p$ and a non-negative integer $n$, the \emph{$p$-adic valuation of $n$} (the largest power of $p$ that divides $n$, defined to be $\infty$ if $n=0$) is denoted by $\nu_p(n)$. If $k$ is another non-negative integer, we set $\nu_p^{(k)}(n):=\min\{\nu_p(n),k\}$. Moreover, if $n>0$, we set $n_p:=p^{\nu_p(n)}$, the so-called \emph{$p$-part of $n$}, and $n_{p'}:=\frac{n}{n_p}$, the \emph{$p'$-part of $n$}.

The \emph{order} of an element $g$ of a finite group $G$ (i.e., the smallest positive integer $o$ such that $g^o=1_G$) is denoted by $\ord(g)$. If $u$ is a unit in a ring $R$, then $u$ belongs to each of the groups $(R,+)$ and $(R^{\ast},\cdot)$. In that case, $\ord(u)$ will always denote the multiplicative order of $u$ (i.e., its order in $(R^{\ast},\cdot)$), and we write $\aord(u)$ for the additive order of $u$ (its order in $(R,+)$). In this paper, this notation is applied for $R=\IZ/m\IZ$, in Section \ref{sec3}.

For a given ring $R$, we can identify integers with elements of $R$, by identifying a positive integer $n$ with the sum of $n$ copies of the ring element $1$, a negative integer $n$ with a sum of $-n$ copies of the ring element $-1$, and the integer $0$ with the ring element $0$. That way, a function definition like $x\mapsto ax+b$ where $a$ and $b$ are integers makes sense over any ring.

\section{Proofs of the two main theorems}\label{sec2}

\subsection{Proof of Theorem \ref{wreathIsoTheo}}\label{subsec2P1}

Recall that an element $(\psi,(\lambda(s_0,b_0),\lambda(s_1,b_1),\ldots,\lambda(s_{d-1},b_{d-1})))$ of the imprimitive permutational wreath product $\Hol(C)\wr_{\imp}\Sym(d)$ maps a point $(c,i)$, where $c\in C$ and $i\in\{0,1,\ldots,d-1\}$, to
\[
(c^{\lambda(s_{\psi(i)},b_{\psi(i)})},\psi(i))=(b_{\psi(i)}c^{s_{\psi(i)}},\psi(i)).
\]
We need to check that the specified image of this wreath product element under $\iota_{\omega}$,
\[
f_{\omega}((\omega^{\psi(i)-is_{\psi(i)}}b_{\psi(i)})_{i=0,1,\ldots,d-1},(s_{\psi(i)})_{i=0,1,\ldots,d-1})_{\mid\IF_q^{\ast}},
\]
shows the same mapping behavior under the identification of points $(c,i)$ with nonzero field elements $c\omega^i$ via the bijection $\beta_{\omega}$. And indeed, noting that $c\omega^i$ lies in the $i$-th coset $C_i$, we conclude that
\begin{align*}
&(c\omega^i)^{f_{\omega}((\omega^{\psi(i)-is_{\psi(i)}}b_{\psi(i)})_{i=0,1,\ldots,d-1},(s_{\psi(i)})_{i=0,1,\ldots,d-1})}=(c\omega^i)^{x\mapsto\omega^{\psi(i)-is_{\psi(i)}}b_{\psi(i)}x^{s_{\psi(i)}}} \\
&=\omega^{\psi(i)-is_{\psi(i)}}b_{\psi(i)}\cdot(c\omega^i)^{s_{\psi(i)}}=\omega^{\psi(i)-is_{\psi(i)}}b_{\psi(i)}\cdot c^{s_{\psi(i)}}\omega^{is_{\psi(i)}}=\omega^{\psi(i)}b_{\psi(i)}c^{s_{\psi(i)}} \\
&=\beta_{\omega}^{-1}((b_{\psi(i)}c^{s_{\psi(i)}},\psi(i))).
\end{align*}
This shows that the diagram
\begin{center}
\begin{tikzpicture}
\matrix (m) [matrix of math nodes,row sep=3em,column sep=4em,minimum width=2em]
  {
     \IF_q^{\ast} & \IF_q^{\ast} \\
     C\times\{0,\ldots,d-1\} & C\times\{0,\ldots,d-1\} \\};
  \path[-stealth]
    (m-1-1) edge node[above,shift={(0pt,5pt)}] {$(\psi,(\lambda(s_i,b_i))_{i=0,\ldots,d-1})^{\iota_{\omega}}$} (m-1-2)
    (m-1-1) edge node [left] {$\beta_{\omega}$} (m-2-1)
		(m-1-2) edge node [right] {$\beta_{\omega}$} (m-2-2)
		(m-2-1) edge node [below,shift={(0pt,-5pt)}] {$(\psi,(\lambda(s_i,b_i))_{i=0,\ldots,d-1})$} (m-2-2);
\end{tikzpicture}
\end{center}
is commutative, whence $(\beta_{\omega},\iota_{\omega})$ is an isomorphism of permutation groups. We still need to check the formula for the inverse isomorphism as well as the claims on the two pre-images.

Note that $f_{\omega}(\vec{a},\vec{r})$ maps $c\omega^i$ to
\[
a_i(\omega^ic)^{r_i}=a_i\omega^{ir_i}c^{r_i}=\omega^{\psi(i)}\cdot a_i\omega^{ir_i-\psi(i)}c^{r_i},
\]
which corresponds to the pair
\[
(a_i\omega^{ir_i-\psi(i)}c^{r_i},\psi(i))
\]
under $\beta_{\omega}$. We need to check that the specified image of $f_{\omega}(\vec{a},\vec{r})$ under $\iota_{\omega}^{-1}$, the wreath product element
\[
(\psi,(\lambda(r_{\psi^{-1}(i)},\omega^{r_{\psi^{-1}(i)}\psi^{-1}(i)-i}a_{\psi^{-1}(i)}))_{i=0,1,\ldots,d-1}),
\]
shows the same mapping behavior on pairs $(c,i)$. And indeed,
\begin{align*}
&(c,i)^{(\psi,(\lambda(r_{\psi^{-1}(i)},\omega^{r_{\psi^{-1}(i)}\psi^{-1}(i)-i}a_{\psi^{-1}(i)}))_{i=0,1,\ldots,d-1})} = \\
&(c^{\lambda(r_{\psi^{-1}(\psi(i))},\omega^{r_{\psi^{-1}(\psi(i))}\psi^{-1}(\psi(i))-\psi(i)}a_{\psi^{-1}(\psi(i))})},\psi(i)) = \\
&(c^{\lambda(r_i,\omega^{r_ii-\psi(i)}a_i)},\psi(i))=(a_i\omega^{r_ii-\psi(i)}c^{r_i},\psi(i)),
\end{align*}
as required.

We now turn to the two claims on pre-images. Fix an integer $s$ that is coprime to $|C|=\frac{q-1}{d}$. Then for all $\psi\in\Sym(d)$ and all $b_0,b_1,\ldots,b_{d-1}\in C$, we have
\[
(\psi,(\lambda(s,b_i))_{i=0,1,\ldots,d-1})^{\iota_{\omega}}=f_{\omega}((\omega^{\psi(i)-is}b_{\psi(i)})_{i=0,1,\ldots,d-1},(s)_{i=0,1,\ldots,d-1})_{\mid\IF_q^{\ast}},
\]
which is certainly (the restriction to $\IF_q^{\ast}$ of) an $s$-th order cyclotomic mapping of $\IF_q$ of index $d$. Since we know that $\iota_{\omega}$ maps into $\GCP(d,q)$, it must actually be an $s$-th order cyclotomic permutation of $\IF_q$ of index $d$. Moreover, $\iota_{\omega}$ is injective, so we get $d!\cdot|C|^d$ distinct permutations that way, and this is just the number of $s$-th order cyclotomic permutations of $\IF_q$ of index $d$.

Hence, we conclude that for each fixed $s$, the isomorphism $\iota_{\omega}$ maps the wreath product elements $(\psi,(\lambda(s,b_i))_{i=0,1,\ldots,d-1})$ bijectively to the (restrictions to $\IF_q^{\ast}$ of) $s$-th order cyclotomic permutations of $\IF_q$. This implies both claims on pre-images and concludes the proof of Theorem \ref{wreathIsoTheo}.

\subsection{Proof of Theorem \ref{cycloPolyTheo}}\label{subsec2P2}

We explained right after Theorem \ref{cycloPolyTheo} how statement (1) can be obtained from the proof of \cite[Formula (3)]{Wan13a}. Hence, we will only prove statement (2) here.

Note that $P$ is the polynomial form of an index $d$ generalized cyclotomic mapping of $\IF_q$ if and only if it is of the form (\ref{wangEq}) for some $a_0,\ldots,a_{d-1}\in\IF_q$ and some $r_0,\ldots,r_{d-1}\in\{1,2,\ldots,\frac{q-1}{d}\}$. Such a polynomial always has vanishing constant term, and so if the constant term of $P$ is not $0$, it cannot be the polynomial form an index $d$ generalized cyclotomic mapping of $\IF_q$.

What we still need to show is that under the assumption $P(0)=0$, steps 2--11 of Algorithm \ref{algo1} verify that $P$ is of the form (\ref{wangEq}) and, if so, find $a_i\in\IF_q$ and $r_i\in\{1,\ldots,\frac{q-1}{d}\}$ such that $P$ is of the form (\ref{wangEq}); those $a_i$ and $r_i$ then are the parameters for the $\omega$-cyclotomic form $f_{\omega}(\vec{a},\vec{r})$ of the function $\IF_q\rightarrow\IF_q$ with polynomial form $P$.

The simple case $P=0$ is dealt with in step 2, so we assume henceforth that $P\not=0$. We will go through steps 3--11 of Algorithm \ref{algo1}, assuming that $P$ is of the form (\ref{wangEq}) for some $a_0,a_1,\ldots,a_{d-1}\in\IF_q$, not all zero, and some $r_0,r_1,\ldots,r_{d-1}\in\{1,2,\ldots,\frac{q-1}{d}\}$. Moreover, we may assume without loss of generality that
\begin{equation}\label{technicalAssumptionEq}
r_i=\min\{r_j: j\in\{0,\ldots,d-1\},a_j\not=0\}\text{ whenever }a_i=0.
\end{equation}

Before we start tracing the steps of the algorithm, we make some important definitions. We denote by $\rho'_0<\rho'_1<\cdots<\rho'_{k'-1}$ the distinct values of the $r_i$. For $\ell'=0,1,\ldots,k'-1$, we set
\[
S'_{\ell'}:=\{i\in\{0,1,\ldots,d-1\}:r_i=\rho'_i\}.
\]
Note that the sets $S'_{\ell'}$ form a partition of the set $\{0,1,\ldots,d-1\}$. For $i=0,1,\ldots,d-1$ and $\ell'=0,1,\ldots,k'-1$, we set
\[
{b'_i}^{(\ell')}:=\begin{cases}a_i, & \text{if }i\in S'_{\ell'}, \\ 0, & \text{otherwise}.\end{cases}
\]
By our assumption that $P$ is of the form (\ref{wangEq}), we have that for $\ell'=0,1,\ldots,k'-1$ and $j=0,1,\ldots,d-1$, the coefficient of $T^{j\cdot\frac{q-1}{d}+\rho_{\ell'}}$ in $P$ is
\begin{equation}\label{coeffEq}
\frac{1}{d}\sum_{i\in S'_{\ell'}}{\zeta^{-ij}a_i}=\frac{1}{d}\sum_{i=0}^{d-1}{\zeta^{-ij}{b'_i}^{(\ell')}}.
\end{equation}
For $\ell'=0,1,\ldots,k'-1$, let $\vec{v'}_{\ell'}$ be the length $d$ column vector over $\IF_q$ whose $j$-th entry, for $j=0,1,\ldots,d-1$, is the coefficient of $T^{j\cdot\frac{q-1}{d}+\rho_{\ell'}}$ in $P$.

By Formula (\ref{coeffEq}), we get that
\[
\vec{v'}_{\ell'}=\frac{1}{d}V(1,\zeta^{-1},\zeta^{-2},\ldots,\zeta^{-(d-1)})\vec{b'}_{\ell'},
\]
where $\vec{b'}_{\ell'}$ is the length $d$ column vector over $\IF_q$ whose $i$-th entry, for $i=0,1,\ldots,d-1$, is ${b'_i}^{(\ell')}$. Since $\vec{b'}_{\ell'}$ is not the zero vector (by our technical assumption (\ref{technicalAssumptionEq}), it contains at least one entry that is equal to one of the nonzero $a_i$) and the displayed Vandermonde matrix is invertible, we conclude that $\vec{v'}_{\ell'}$ is not the zero vector. This means that for each value $\rho'_{\ell'}$ that occurs among the $r_i$, there is a term in $P$ of degree $j\cdot\frac{q-1}{d}+\rho'_{\ell'}$, for a suitable $j\in\{0,1,\ldots,d-1\}$. Therefore, each $r_i$ occurs as the (non-standard) remainder of integer division by $\frac{q-1}{d}$ of a suitable term degree in $P$.

We now start to go through the steps of Algorithm \ref{algo1}:
\begin{itemize}
\item In step 3, note that if $P$ is of the assumed form, it has at most $d^2$ terms, so should the input have more terms than that, it cannot be the polynomial form of an index $d$ generalized cyclotomic mapping of $\IF_q$.
\item In step 4, by what we said above, each $\rho'_{\ell'}$ occurs as a value among $\rho_1,\ldots,\rho_k$. But also, by the form of $P$, it is clear that the value of each $\rho_{\ell}$ is one of the $r_i$, hence one of the $\rho_{\ell'}$. This shows that $k'=k$ and $\rho'_{\ell}=\rho_{\ell}$ for $\ell=0,1,\ldots,k-1$. Because there are at most $d$ distinct values of the $r_i$, and those distinct values are just the $\rho_{\ell}$, we conclude that the number $k$ of the $\rho_{\ell}$ cannot exceed $d$, unless $P$ is not of the form (\ref{wangEq}).
\item In step 5, since we now know that the $\rho'_{\ell'}$ and the $\rho_{\ell}$ are one and the same, we can infer that $\vec{v}_{\ell}=\vec{v'}_{\ell}$ for each $\ell=0,1,\ldots,k-1$.
\item In step 6, this yields that $\vec{b}_{\ell}=\vec{b'}_{\ell}$ for $\ell=0,1,\ldots,d-1$.
\item In step 7, by the definition of the set $S_{\ell}$ and our technical assumption (\ref{technicalAssumptionEq}), we conclude that $S_{\ell}=S'_{\ell}$ for $\ell=0,1,\ldots,d-1$.
\item In step 8, since $S_{\ell}=S'_{\ell}$ and the sets $S'_{\ell}$ form a partition of $\{0,1,\ldots,d-1\}$, we conclude that the sets $S_{\ell}$ must form a partition of $\{0,1,\ldots,d-1\}$ -- if they do not, then $P$ is not of the form (\ref{wangEq}).
\item In step 9, using the definition of the vectors $\vec{b'}_{\ell'}$, our technical assumption (\ref{technicalAssumptionEq}), and the fact that $\vec{b}_{\ell}=\vec{b'}_{\ell}$, we can read off the $a_i$ and $r_i$ as described, and the associated index $d$ generalized cyclotomic mapping of $\IF_q$ has the polynomial form $P$, whence the algorithm provides the correct output in step 11.
\end{itemize}

\section{Cycle types, cycle indices, and P{\'o}lya's Theorem}\label{secPol}

As an application, we will use Theorem \ref{wreathIsoTheo} to characterize the cycle types of index $d$ generalized cyclotomic permutations of $\IF_q$. More precisely, using Theorem \ref{wreathIsoTheo} together with results of P{\'o}lya, \cite[table at the bottom of p.~180]{Pol37a}, and Wei-Gao-Yang as well as Wei-Xu, \cite[Theorems 1.2, 1.3 and 2.4]{WX93a}, allows one to determine the cycle index of the permutation group $\GCP(d,q)$ (of restrictions to $\IF_q^{\ast}$ of index $d$ generalized cyclotomic permutations of $\IF_q$). In this section, we prepare the ground by reviewing the concepts of \emph{cycle type} and \emph{cycle index} as well as P{\'o}lya's Theorem. In the next section, we will discuss Wei-Gao-Yang's and Wei-Xu's results in detail.

Let $G\leq\Sym(\Omega)$ be a permutation group on a finite set $\Omega$ of size $n$. Intuitively, for a given $g\in G$, the \enquote{cycle type of $g$} should be the information how many cycles of each length $g$ has. Formally, this is commonly defined as the multivariate monomial
\[
\CT(g):=x_1^{k_1}x_2^{k_2}\cdots x_n^{k_n}\in\IQ[x_1,x_2,\ldots,x_n]
\]
where $k_i$ is the number of cycles of length $i$ of $g$, for $i=1,2,\ldots,n$. For example, the cycle type of the permutation $(0)(1)(2)(3,4)(5,6)(7,8,9)\in\Sym(10)$ is $x_1^3x_2^2x_3$. The \emph{cycle index of $G$}, which we will denote by $\CI(G)$, is defined as the average of all the cycle types in $G$:
\[
\CI(G):=\frac{1}{|G|}\sum_{g\in G}{\CT(g)}.
\]
For example, the symmetric group $\Sym(3)$ contains $1$ identity element, $3$ transpositions, and $2$ cycles of length $3$, leading to
\[
\CI(\Sym(3))=\frac{1}{6}(x_1^3+3x_1x_2+2x_3)=\frac{1}{6}x_1^3+\frac{1}{2}x_1x_2+\frac{1}{3}x_3.
\]
In general, the cycle index of a symmetric group is given by the following formula, which will be used in Subsection \ref{subsec5P1}:

\begin{proposition}\label{ciSymProp}
Let $d$ be a positive integer, and set
\[
\Lambda_d:=\{(\lambda_1,\lambda_2,\ldots,\lambda_d)\in\IN^d: \sum_{i=1}^d{i\lambda_i}=d\}.
\]
Then
\[
\CI(\Sym(d))=\frac{1}{d!}\sum_{(\lambda_1,\ldots,\lambda_d)\in\Lambda_d}{\frac{d!}{\prod_{i=1}^d{i^{\lambda_i}\lambda_i!}}x_1^{\lambda_1}x_2^{\lambda_2}\cdots x_d^{\lambda_d}}.
\]
\end{proposition}

Cycle indices are not only compact ways of displaying the information how many elements of a given cycle type a certain permutation group has -- they are a central tool in P{\'o}lya's celebrated enumeration theorem, used to count orbits of \enquote{colored configurations} under a permutation group, see \cite[Chapter 9]{Tuc12a}. Computing the cycle index of a given permutation group is an important problem and has been studied for various groups by different authors, see for example \cite{Fri97a,Ful99a,WX93a}.

P{\'o}lya himself described a way to obtain the cycle index of an imprimitive permutational wreath product $G\wr_{\imp}P$ from the cycle indices of $G$ and $P$:

\begin{theorem}\label{polyaTheo}(P{\'o}lya, \cite[table at the bottom of p.~180]{Pol37a})
Let $G\leq\Sym(\Omega)$ and $P\leq\Sym(d)$ be permutation groups of finite degree. For each positive integer $m$, set
\[
\CI^{(m)}(G):=\CI(G)(x_m,x_{2m},\ldots,x_{|\Omega|m}).
\]
Then
\[
\CI(G\wr_{\imp}P)=\CI(P)(\CI^{(1)}(G),\CI^{(2)}(G),\ldots,\CI^{(d)}(G)).
\]
\end{theorem}

\begin{example}\label{polyaEx}
We compute the cycle index of the imprimitive permutational wreath product $\Sym(3)\wr_{\imp}\Sym(3)$. By the above discussion, the cycle index of $\Sym(3)$ is
\[
P(x_1,x_2,x_3):=\frac{1}{6}x_1^3+\frac{1}{2}x_1x_2+\frac{1}{3}x_3\in\IQ[x_1,x_2,x_3].
\]
According to Theorem \ref{polyaTheo},
\begin{align*}
&\CI(\Sym(3)\wr_{\imp}\Sym(3))= \\
&P(P(x_1,x_2,x_3),P(x_2,x_4,x_6),P(x_3,x_6,x_9))= \\
&\frac{1}{6}P(x_1,x_2,x_3)^3+\frac{1}{2}P(x_1,x_2,x_3)P(x_2,x_4,x_6)+\frac{1}{3}P(x_3,x_6,x_9).
\end{align*}
Using any computer algebra software (the authors used GAP \cite{GAP4} for it), one can evaluate this last expression to get that the cycle index of $\Sym(3)\wr_{\imp}\Sym(3)$ is
\begin{align*}
&\frac{1}{1296}x_1^9+\frac{1}{144}x_1^7x_2+\frac{1}{216}x_1^6x_3+\frac{1}{48}x_1^5x_2^2+\frac{1}{36}x_1^4x_2x_3+\frac{5}{144}x_1^3x_2^3+\frac{1}{24}x_1^3x_2x_4+ \\
&\frac{1}{108}x_1^3x_3^2+\frac{1}{24}x_1^2x_2^2x_3+\frac{1}{24}x_1x_2^4+\frac{1}{36}x_1^3x_6+\frac{1}{8}x_1x_2^2x_4+\frac{1}{36}x_1x_2x_3^2+\frac{1}{36}x_2^3x_3+ \\
&\frac{1}{12}x_1x_2x_6+\frac{1}{12}x_2x_3x_4+\frac{5}{81}x_3^3+\frac{2}{9}x_3x_6+\frac{1}{9}x_9.
\end{align*}
\end{example}

We now discuss a proof of P{\'o}lya's result that is based on a more detailed lemma, which will be used in Subsection \ref{subsec5P2}.

\begin{definition}\label{polyaDef}
Let $G\leq\Sym(\Omega)$ be a permutation group on the finite set $\Omega$.
\begin{enumerate}
\item For $g\in G$ and a positive integer $m$, we define
\[
\CT^{(m)}(g):=\CT(g)(x_m,x_{2m},\ldots,x_{|\Omega|m}).
\]
\item Let $d$ be a positive integer, let $\psi\in\Sym(d)$, and let $g_0,g_1,\ldots,g_{d-1}\in G$. Consider the element $g=(\psi,(g_0,g_1,\ldots,g_{d-1}))$ of the imprimitive permutational wreath product $G\wr_{\imp}\Sym(d)$. For each cycle $\zeta=(i_0,i_1,\ldots,i_{\ell-1})$ of $\psi$, say written such that $i_0=\min\{i_0,i_1,\ldots,i_{\ell-1}\}$, we define
\[
\fcp_{\zeta}(g):=g_{i_0}g_{i_1}\cdots g_{i_{\ell-1}},
\]
the so-called \emph{forward cycle product of $g$ with respect to $\zeta$}. Moreover, we denote by $\ell(\zeta):=\ell$ the \emph{length of $\zeta$}.
\end{enumerate}
\end{definition}

\begin{lemma}\label{polyaLem}
Let $G\leq\Sym(\Omega)$ be a permutation group on a finite set $\Omega$, let $d$ be a positive integer, let $\psi\in\Sym(d)$, and let $g_0,g_1,\ldots,g_{d-1}\in G$. Consider the element $g=(\psi,(g_0,g_1,\ldots,g_{d-1}))$ of $G\wr_{\imp}\Sym(d)$. Then
\[
\CT(g)=\prod_{\zeta}{\CT^{(\ell(\zeta))}(\fcp_{\zeta}(g))},
\]
where the running index $\zeta$ ranges over the cycles of $\psi$.
\end{lemma}

\begin{proof}
For a given cycle $\zeta=(i_0,\ldots,i_{\ell(\zeta)-1})$ of $\psi$, written such that
\[
i_0=\min\{i_0,i_1,\ldots,i_{\ell(\zeta)-1}\},
\]
set
\[
U(\zeta):=\bigcup_{j=0}^{\ell(\zeta)-1}{G\times\{i_j\}};
\]
intuitively, $U(\zeta)$ is the (disjoint) union of the copies of $G$ corresponding to the indices on $\zeta$. We note that $g$ restricts to a permutation $g^{(\zeta)}$ on each subset $U(\zeta)$ of $G\times\{0,1,\ldots,d-1\}$. Therefore, for each possible cycle length $\ell\in\{1,2,\ldots,d|\Omega|\}$ of $g$, the number of length $\ell$ cycles of $g$ equals the sum of the numbers of length $\ell$ cycles of the various $g^{(\zeta)}$. In terms of polynomial algebra, this translates to
\[
\CT(g)=\prod_{\zeta}{\CT(g^{(\zeta)})},
\]
so that it suffices to show that for each $\zeta$, we have
\begin{equation}\label{toShowEq}
\CT(g^{(\zeta)})=\CT^{(\ell(\zeta))}(\fcp_{\zeta}(g)).
\end{equation}
Now, note that $g^{(\zeta)}\in\Sym(U(\zeta))$ permutes the $\ell(\zeta)$ copies of $G$ in $U(\zeta)$ cyclically. In particular, every cycle length of $g^{(\zeta)}$ is divisible by $\ell(\zeta)$, and when $g$ is raised to the $\ell(\zeta)$-th power, each cycle of $g$, say of length $k$, splits into $\ell(\zeta)$ cycles of length $\frac{k}{\ell(\zeta)}$, one in each copy of $G$. Therefore, the restrictions of $g^{(\zeta)}$ to the various copies of $G$ in $U(\zeta)$ all have the same cycle type, and without loss of generality, we say
\begin{equation}\label{auxiliaryEq}
\CT(g^{(\zeta)})=\CT^{(\ell(\zeta))}(g^{\ell(\zeta)}_{\mid(G\times\{i_{\ell(\zeta)-1}\})}),
\end{equation}
Next we find which element of $G$ equals the restriction of $g^{\ell(\zeta)}$ (or, equivalently, of $(g^{(\zeta)})^{\ell(\zeta)}$) to $G\times\{i_{\ell(\zeta)-1}\}$. First we note that the restriction $g^{(\zeta)}$ of $g$ to $U(\zeta)$ can be viewed as an element of the imprimitive permutational wreath product
\[
G\wr_{\imp}\Sym(\{i_0,i_1,\ldots,i_{\ell(\zeta)-1}\}).
\]
This allows us to write (for ease of notation)
\[
g^{(\zeta)}=((i_0,i_1,\ldots,i_{\ell(\zeta)-1}),(g_{i_0},g_{i_1},\ldots,g_{i_{\ell(\zeta)-1}})).
\]
Set
\[
\psi':=(i_0,i_1,\ldots,i_{\ell(\zeta)-1})
\]
and
\[
\vec{g}:=(g_{i_0},g_{i_1},\ldots,g_{i_{\ell(\zeta)-1}}),
\]
so that $g^{(\zeta)}=\psi'\vec{g}$. Then
\begin{align*}
(g^{(\zeta)})^{\ell(\zeta)}&=(\psi'\vec{g})^{\ell(\zeta)}=\psi'\vec{g}\psi'\vec{g}\cdots\psi'\vec{g}=\psi'^{\ell(\zeta)}(\vec{g})^{\psi'^{\ell(\zeta)-1}}(\vec{g})^{\psi'^{\ell(\zeta)-2}}\cdots(\vec{g})^{\psi'}\vec{g} \\
&=(\vec{g})^{\psi'^{\ell(\zeta)-1}}(\vec{g})^{\psi'^{\ell(\zeta)-2}}\cdots(\vec{g})^{\psi'}\vec{g} \\
&=(g_{i_1},g_{i_2},\ldots,g_{i_{\ell(\zeta)-1}},g_{i_0})\cdot(g_{i_2},g_{i_3},\ldots,g_{i_0},g_{i_1})\cdots \\
&(g_{i_{\ell(\zeta)-1}},g_{i_0},\ldots,g_{i_{\ell(\zeta)-3}},g_{i_{\ell(\zeta)-2}})(g_{i_0},g_{i_1},\ldots,g_{i_{\ell(\zeta)-2}},g_{i_{\ell(\zeta)-1}}),
\end{align*}
which is an element of $G^{\ell(\zeta)}$, and the last entry of it (which equals the restriction of $g^{\ell(\zeta)}$ to $G\times\{i_{\ell(\zeta)-1}\})$ is
\[
g_{i_0}g_{i_1}\cdots g_{i_{\ell(\zeta)-2}}g_{i_{\ell(\zeta)-1}}=\fcp_{\zeta}(g).
\]
Hence
\[
g^{\ell(\zeta)}_{\mid(G\times\{i_{\ell(\zeta)-1}\})}=\fcp_{\zeta}(g),
\]
so that formula (\ref{toShowEq}) is clear by formula (\ref{auxiliaryEq}), and the proof is complete.
\end{proof}

\begin{proof}[Proof of Theorem \ref{polyaTheo}]
For this proof, it will be convenient to generalize the notion of cycle index as follows: If $X\subseteq\Sym(\Omega)$ is a \emph{nonempty} set of permutations on the finite set $\Omega$, then $\CI(X):=\frac{1}{|X|}\sum_{g\in X}{\CT(g)}$. Moreover, it will be convenient to set $\CI(\emptyset):=0$ (the precise value of $\CI(\emptyset)$ will not be important, just that it is defined).

Let $\lambda=(\lambda_1,\ldots,\lambda_d)\in\Lambda_d$ (see Proposition \ref{ciSymProp}), and let $\rho_{\lambda}$ be the (possibly zero) coefficient of $x_1^{\lambda_1}x_2^{\lambda_2}\cdots x_d^{\lambda_d}$ in $\CI(P)$. Then $\rho_{\lambda}$ is the proportion of elements $\psi\in P$ with cycle type $x_1^{\lambda_1}\cdots x_d^{\lambda_d}$, and it is also the proportion of elements $(\psi,(g_0,\ldots,g_{d-1}))\in G\wr_{\imp}P$ where $\CT(\psi)=x_1^{\lambda_1}\cdots x_d^{\lambda_d}$. Theorem \ref{polyaTheo} asserts that the cycle index of $G\wr_{\imp}P$ is the sum of the polynomials
\[
\rho_{\lambda}\CI^{(1)}(G)^{\lambda_1}\CI^{(2)}(G)^{\lambda_2}\cdots\CI^{(d)}(G)^{\lambda_d}\in\IQ[x_1,\ldots,x_{d|\Omega|}]
\]
for the various $\lambda\in\Lambda_d$. Set
\[
X_{\lambda}:=\{(\psi,(g_0,\ldots,g_{d-1}))\in G\wr_{\imp}P:\CT(\psi)=x_1^{\lambda_1}\cdots x_d^{\lambda_d}\}.
\]
Then $G\wr_{\imp}P$ is a disjoint union of the sets $X_{\lambda}$ for the various $\lambda\in\Lambda_d$, and consequently,
\[
\CI(G\wr_{\imp}P)=\sum_{\lambda\in\Lambda_d}{\frac{|X_{\lambda}|}{|G\wr_{\imp}P|}\CI(X_{\lambda})}=\sum_{\lambda\in\Lambda_d}{\rho_{\lambda}\CI(X_{\lambda})}.
\]
Hence, in order to prove Theorem \ref{polyaTheo}, it suffices to show that for each $\lambda\in\Lambda_d$, we have
\[
\CI(X_{\lambda})=\CI^{(1)}(G)^{\lambda_1}\CI^{(2)}(G)^{\lambda_2}\cdots\CI^{(d)}(G)^{\lambda_d}.
\]
We show this with a probabilistic argument.

For a given $\psi\in P$ with cycle type $x_1^{\lambda_1}\cdots x_d^{\lambda_d}$, we enumerate the cycles of $\psi$ as follows: For $\ell\in\{1,\ldots,d\}$ and $j\in\{1,\ldots,\lambda_{\ell}\}$,
\[
\zeta_{\ell,j}=(i_0^{(\ell,j)},i_1^{(\ell,j)},\ldots,i_{\ell-1}^{(\ell,j)})
\]
denotes the $j$-th cycle of length $\ell$ of $\psi$, written such that
\[
i_0^{(\ell,j)}=\min\{i_0^{(\ell,j)},i_1^{(\ell,j)},\ldots,i_{\ell-1}^{(\ell,j)}\},
\]
and ordered such that $i_0^{(\ell,1)}<i_0^{(\ell,2)}<\cdots<i_0^{(\ell,\lambda_{\ell})}$. Suppose we draw the element
\[
(\psi,(g_0,\ldots,g_{d-1}))\in X_{\lambda}
\]
uniformly at random. This is the same as drawing $\psi\in P$ with $\CT(\psi)=x_1^{\lambda_1}\cdots x_d^{\lambda_d}$ uniformly at random, and then drawing, independently from $\psi$ and each other, the $g_i\in G$ uniformly at random. It follows that the forward cycle products
\[
\fcp_{\zeta_{l,j}}(g)=g_{i_0^{(\ell,j)}}g_{i_1^{(\ell,j)}}\cdots g_{i_{\ell-1}^{(\ell,j)}}
\]
for $\ell\in\{1,\ldots,d\}$ and $j\in\{1,\ldots,\lambda_{\ell}\}$ are independent and uniform random variables over $G$. Fix a tuple
\[
\vec{\mu}=(\mu_{\ell,j})_{\ell,j}\in\Lambda_{|\Omega|}^{\{1,\ldots,d\}\times\{1,\ldots,\lambda_{\ell}\}},
\]
with each $\mu_{\ell,j}$ written as
\[
\mu_{\ell,j}=(\mu_1^{(\ell,j)},\mu_2^{(\ell,j)},\ldots,\mu_{|\Omega|}^{(\ell,j)}).
\]
Define $X_{\lambda,\vec{\mu}}$ as the set of all $(\psi,(g_0,\ldots,g_{d-1}))\in X_{\lambda}$ such that for all $\ell$ and $j$,
\[
\CT(\fcp_{\zeta_{\ell,j}}(g))=x_1^{\mu_1^{(\ell,j)}}\cdots x_{|\Omega|}^{\mu_{|\Omega|^{(\ell,j)}}}.
\]
Then $X_{\lambda}$ is a disjoint union of the sets $X_{\lambda,\vec{\mu}}$ for the various $\vec{\mu}$, and
\begin{equation}\label{ciLambdaEq}
\CI(X_{\lambda})=\sum_{\vec{\mu}}{\frac{|X_{\lambda,\vec{\mu}}|}{|X_{\lambda}|}\CI(X_{\lambda,\vec{\mu}})}.
\end{equation}
For $\mu=(\mu_1,\ldots,\mu_{|\Omega|})\in\Lambda_{|\Omega|}$, denote by $c(\mu)$ the coefficient of $x_1^{\mu_1}\cdots x_{|\Omega|}^{\mu_{|\Omega|}}$ in $\CI(G)$, which is the proportion of elements in $G$ with cycle type $x_1^{\mu_1}\cdots x_{|\Omega|}^{\mu_{|\Omega|}}$. For given $\ell$ and $j$, since $\fcp_{\zeta_{\ell,j}}(g)$ is a uniform random variable over $G$, the relative probability (given that $\CT(\psi)=x_1^{\lambda_1}\cdots x_d^{\lambda_d}$) that the cycle type of $\fcp_{\zeta_{\ell,j}}(g)$ is $x_1^{\mu_1^{(\ell,j)}}x_2^{\mu_2^{(\ell,j)}}\cdots x_{|\Omega|}^{\mu_{|\Omega|}^{(\ell,j)}}$ is just $c(\mu_{\ell,j})$, and since the $\fcp_{\zeta_{\ell,j}}(g)$ are independent, the relative probability that this holds for all $\ell$ and $j$ is just $\prod_{\ell,j}{c(\mu_{\ell,j})}$. In other words,
\[
\frac{|X_{\lambda,\vec{\mu}}|}{|X_{\lambda}|}=\prod_{\ell,j}{c(\mu_{\ell,j})}.
\]
Moreover, by Lemma \ref{polyaLem}, we find that all elements of $X_{\lambda,\vec{\mu}}$ have the cycle type
\begin{equation}\label{stretchedEq}
\prod_{\ell,j}{x_{\ell}^{\mu_1^{(\ell,j)}}x_{2\ell}^{\mu_2^{(\ell,j)}}\cdots x_{|\Omega|\ell}^{\mu_{|\Omega|}^{(\ell,j)}}}.
\end{equation}
Hence as long as the set $X_{\lambda,\vec{\mu}}$ is nonempty, its cycle index is the polynomial from formula (\ref{stretchedEq}), and otherwise the cycle index is $0$. Note that whenever $\vec{\mu}$ is such that $X_{\lambda,\vec{\mu}}=\emptyset$, then the corresponding summand in formula (\ref{ciLambdaEq}) is $0\cdot\CI(X_{\lambda,\vec{\mu}})$, and so it makes no difference if we replace $\CI(X_{\lambda,\vec{\mu}})$ (which is $0$) in such summands by the polynomial in formula (\ref{stretchedEq}). We conclude that
\begin{align*}
&\CI(X_{\lambda})=\sum_{\vec{\mu}}\prod_{\ell,j}{c(\mu_{\ell,j})x_{\ell}^{\mu_1^{(\ell,j)}}x_{2\ell}^{\mu_2^{(\ell,j)}}\cdots x_{|\Omega|\ell}^{\mu_{|\Omega|}^{(\ell,j)}}} \\
&=\prod_{\ell=1}^d\prod_{j=1}^{\lambda_{\ell}}\sum_{\mu_{\ell,j}\in\Lambda_{|\Omega|}}{c(\mu_{\ell,j})x_{\ell}^{\mu_1^{(\ell,j)}}x_{2\ell}^{\mu_2^{(\ell,j)}}\cdots x_{|\Omega|\ell}^{\mu_{|\Omega|}^{(\ell,j)}}}=\prod_{\ell=1}^d\prod_{j=1}^{\lambda_{\ell}}{\CI^{(\ell)}(G)}=\prod_{\ell=1}^d{\CI^{(\ell)}(G)^{\lambda_{\ell}}},
\end{align*}
which is just what we needed to show.
\end{proof}

\section{Cycle indices of holomorphs of finite cyclic groups}\label{sec3}

In our application of Theorem \ref{polyaTheo}, we have $P=\Sym(d)$, and $G$ is the permutation group $\Hol(C)$ on the index $d$ subgroup $C$ of $\IF_q^{\ast}$, as defined just before Theorem \ref{wreathIsoTheo}. Therefore, we will now discuss cycle types in that permutation group.

We start with some general group-theoretic remarks. For every group $G$, the functions $G\rightarrow G$ of the form $\lambda(\alpha,g):x\mapsto x^{\alpha}g$, where $\alpha$ is a fixed automorphism of $G$ and $g$ is a fixed element of $G$, form a permutation group on $G$ that is known as the \emph{holomorph of $G$}, denoted $\Hol(G)$. It is the internal semidirect product of the automorphism group $\Aut(G)$ and the group of right-translations $x\mapsto xg$ on $G$ (also known as the right-regular representation of $G$ on itself), which is isomorphic to $G$. The elements of $\Hol(G)$ are also called \emph{affine permutations of $G$}. In the case of the index $d$ multiplicative subgroup $C$ of $\IF_q^{\ast}$ (where $d\mid q-1$), automorphisms of $C$ are of the form $x\mapsto x^r$ for some fixed integer $r$ that is coprime to $|C|$. Identifying such an automorphism with the integer $r$, we obtain the notation $\lambda(r,c)$ from the Introduction as a special case.

Since $C$ is cyclic and all cyclic groups of the same order are isomorphic, there is no loss in replacing $C$ by the additive group $\IZ/|C|\IZ$. The goal is thus to obtain a characterization of the cycle types of the elements of $\Hol(\IZ/m\IZ)$ for an arbitrary positive integer $m$. Note that affine permutations of $\IZ/m\IZ$ take the form $x+m\IZ\mapsto ax+b+m\IZ$ where $a$ is an integer that is coprime to $m$ and $b$ is an arbitrary integer; we will also denote this permutation by $\lambda(a,b)$.

When rewritten into this form, it becomes clear that this problem is studied in Wei and Xu's paper \cite{WX93a}, and we now provide an overview of their results. They consider the following notion of a direct product of permutation groups (which is not the same as the one defined by P{\'o}lya in \cite[pp.~177f.]{Pol37a}):

\begin{definition}\label{permProductDef}
Let $G_i\leq\Sym(\Omega_i)$ be a permutation group for $i=0,1,\ldots,s-1$. The \emph{direct product} of these permutation groups, denoted $\prod_{i=0}^{s-1}{G_i}$, is the permutation group on the Cartesian set product $\prod_{i=0}^{s-1}{\Omega_i}$ that is the image of the abstract group direct product $\prod_{i=0}^{s-1}{G_i}$ under the abstract group homomorphism $\varphi:\prod_{i=0}^{s-1}{G_i}\rightarrow\Sym(\prod_{i=0}^{s-1}{\Omega_i})$ such that
\[
(g_i)_{i\in I}^{\varphi((\sigma_i)_{i\in I})}=(g_i^{\sigma_i})_{i\in I}.
\]
\end{definition}

Just as P{\'o}lya's Theorem \ref{polyaTheo} provides a way to compute the cycle index of $G\wr_{\imp}P$ from the cycle indices of $G$ and $P$, Wei and Xu give a formula to compute the cycle index of $\prod_{i=1}^s{G_i}$ from the cycle indices of the $G_i$ (assuming that the $G_i$ are permutation groups on finite sets).

\begin{definition}\label{refProdDef}(Wei-Xu, \cite[Definition 2.2]{WX93a})
Let $R$ be a commutative ring. The \emph{$\divideontimes$-product} is the unique function $\divideontimes:R[x_i\mid i\in\IN^+]^2\rightarrow R[x_i\mid i\in\IN^+]$, written in infix-notation, such that for
\[
f=\sum_{i_1,\ldots,i_u\in\IN}{a_{i_1,\ldots,i_u}x_1^{i_1}\cdots x_u^{i_u}}
\]
and
\[
g=\sum_{j_1,\ldots,j_v\in\IN}{b_{j_1,\ldots,j_v}x_1^{j_1}\cdots x_v^{j_v}},
\]
one has
\[
f\divideontimes g=\sum_{i_1,\ldots,i_u,j_1,\ldots,j_v\in\IN}{a_{i_1,\ldots,i_u}b_{j_1,\ldots,j_v}\prod_{1\leq\ell\leq u,1\leq m\leq v}{x_{\ell}^{i_{\ell}}\divideontimes x_m^{j_m}}},
\]
and the $\divideontimes$-product of two powers of variables is defined via the formula
\[
x_i^e\divideontimes x_j^f:=x_{\lcm(i,j)}^{ef\gcd(i,j)}.
\]
\end{definition}

By \cite[Lemma 2.3]{WX93a}, the $\divideontimes$-product is a commutative and associative product of polynomials, and the $\divideontimes$-product of an arbitrary number of powers of variables can be computed according to the following formula (which is useful to know if one wants to compute a $\divideontimes$-product of an arbitrary number of polynomials):
\[
\divideontimes_{j=1}^s{x_{i_j}^{e_j}}=x_{\lcm(i_1,\ldots,i_s)}^{n_1\cdots n_s\cdot\frac{i_1\cdots i_s}{\lcm(i_1,\ldots,i_s)}}.
\]
Moreover, we have the following:

\begin{theorem}\label{weiXuTheo}(Wei-Xu, \cite[Theorem 2.4 and its proof]{WX93a})
Let $G_1,\ldots,G_s$ be permutation groups on finite sets. The following hold:
\begin{enumerate}
\item Let $g_i\in G_i$ for $i=1,2,\ldots,s$. Identify the tuple $(g_1,g_2,\ldots,g_s)$ with its image in the permutation group $\prod_{i=1}^s{G_i}$. Then
\[
\CT((g_1,\ldots,g_s))=\divideontimes_{i=1}^s{\CT(g_i)}.
\]
\item We have
\[
\CI(\prod_{i=1}^s{G_i})=\divideontimes_{i=1}^s{\CI(G_i)}.
\]
\end{enumerate}
\end{theorem}

\begin{example}\label{permProductEx}
We compute the cycle index of the direct product $\Sym(3)\times\Sym(4)$. By Proposition \ref{ciSymProp}, we have
\[
\CI(\Sym(3))=\frac{1}{6}x_1^3+\frac{1}{2}x_1x_2+\frac{1}{3}x_3,
\]
and
\[
\CI(\Sym(4))=\frac{1}{24}x_1^4+\frac{1}{4}x_1^2x_2+\frac{1}{8}x_2^2+\frac{1}{3}x_1x_3+\frac{1}{4}x_4.
\]
Table \ref{mathrefTable1} provides an overview of the $\divideontimes$-products $\pi_1\divideontimes\pi_2$ that can be formed from variable powers $\pi_1$ and $\pi_2$ that occur in one of the monomials of $\CI(\Sym(3))$ and $\CI(\Sym(4))$ respectively, according to the formula $x_i^e\divideontimes x_j^f=x_{\lcm(i,j)}^{ef\gcd(i,j)}$ from above.

\begin{table}[h]\centering
\begin{tabular}{|c|c|c|c|c|}\hline
\diagbox{$\pi_2$}{$\pi_1$} & $x_1$ & $x_1^3$ & $x_2$ & $x_3$ \\ \hline
$x_1$ & $x_1$ & $x_1^3$ & $x_2$ & $x_3$ \\ \hline
$x_1^2$ & $x_1^2$ & $x_1^6$ & $x_2^2$ & $x_3^2$ \\ \hline
$x_1^4$ & $x_1^4$ & $x_1^{12}$ & $x_2^4$ & $x_3^4$ \\ \hline
$x_2$ & $x_2$ & $x_2^3$ & $x_2^2$ & $x_6$ \\ \hline
$x_2^2$ & $x_2^2$ & $x_2^6$ & $x_2^4$ & $x_6^2$ \\ \hline
$x_3$ & $x_3$ & $x_3^3$ & $x_6$ & $x_3^3$ \\ \hline
$x_4$ & $x_4$ & $x_4^3$ & $x_4^2$ & $x_{12}$ \\ \hline
\end{tabular}
\caption{$\divideontimes$-products $\pi_1\divideontimes\pi_2$ of powers of variables.}
\label{mathrefTable1}
\end{table}

Based on this information, we can compute all possible $\divideontimes$-products of a term $\tau_1$ of $\CI(\Sym(3))$ with a term $\tau_2$ of $\CI(\Sym(4))$, as displayed in Table \ref{mathrefTable2}. Note that the coefficient of $\tau_1\divideontimes\tau_2$ is the product of the coefficients of $\tau_1$ and $\tau_2$, and the monomial of $\tau_1\divideontimes\tau_2$ is the usual polynomial product of all possible $\divideontimes$-products that can be formed from a full variable power in $\tau_1$ with a full variable power in $\tau_2$. For example,
\[
(\frac{1}{2}x_1x_2)\divideontimes(\frac{1}{4}x_1^2x_2)=\frac{1}{8}(x_1\divideontimes x_1^2)(x_1\divideontimes x_2)(x_2\divideontimes x_1^2)(x_2\divideontimes x_2)=\frac{1}{8}x_1^2x_2x_2^2x_2^2=\frac{1}{8}x_1^2x_2^5.
\]

\begin{table}[h]\centering
\begin{tabular}{|c|c|c|c|}\hline
\diagbox{$\tau_2$}{$\tau_1$} & $\frac{1}{6}x_1^3$ & $\frac{1}{2}x_1x_2$ & $\frac{1}{3}x_3$ \\ \hline
$\frac{1}{24}x_1^4$ & $\frac{1}{144}x_1^{12}$ & $\frac{1}{48}x_1^4x_2^4$ & $\frac{1}{72}x_3^4$ \\ \hline
$\frac{1}{4}x_1^2x_2$ & $\frac{1}{24}x_1^6x_2^3$ & $\frac{1}{8}x_1^2x_2^5$ & $\frac{1}{12}x_3^2x_6$ \\ \hline
$\frac{1}{8}x_2^2$ & $\frac{1}{48}x_2^6$ & $\frac{1}{16}x_2^6$ & $\frac{1}{24}x_6^2$ \\ \hline
$\frac{1}{3}x_1x_3$ & $\frac{1}{18}x_1^3x_3^3$ & $\frac{1}{6}x_1x_2x_3x_6$ & $\frac{1}{9}x_3^4$ \\ \hline
$\frac{1}{4}x_4$ & $\frac{1}{24}x_4^3$ & $\frac{1}{8}x_4^3$ & $\frac{1}{12}x_{12}$ \\ \hline
\end{tabular}
\caption{$\divideontimes$-products $\tau_1\divideontimes\tau_2$ of terms.}
\label{mathrefTable2}
\end{table}

$\CI(\Sym(3)\times\Sym(4))$, the $\divideontimes$-product of $\CI(\Sym(3))$ and $\CI(\Sym(4))$, is the sum of all the terms displayed in the body of Table \ref{mathrefTable2}. Therefore, it is
\begin{align*}
&\frac{1}{144}x_1^{12}+\frac{1}{24}x_1^6x_2^3+\frac{1}{48}x_1^4x_2^4+\frac{1}{18}x_1^3x_3^3+\frac{1}{8}x_1^2x_2^5+\frac{1}{6}x_1x_2x_3x_6+\frac{1}{12}x_2^6+\frac{1}{8}x_3^4+\frac{1}{12}x_3^2x_6+ \\
&\frac{1}{6}x_4^3+\frac{1}{24}x_6^2+\frac{1}{12}x_{12}.
\end{align*}
\end{example}

Theorem \ref{weiXuTheo} is useful for the study of cycle indices of holomorphs of finite cyclic groups due to the following reduction result:

\begin{proposition}\label{crtProp}(see \cite[Theorem 3.1]{WX93a})
Let $m=p_1^{k_1}p_2^{k_2}\cdots p_s^{k_s}$ be a positive integer, with the factorization of $m$ into pairwise coprime prime powers displayed. By the Chinese Remainder Theorem, the function
\[
\IZ/m\IZ\rightarrow\prod_{i=1}^s{(\IZ/p_i^{k_i}\IZ)},x+m\IZ\mapsto(x+p_1^{k_1}\IZ,\ldots,x+p_s^{k_s}\IZ),
\]
is a group isomorphism, and it induces an isomorphism of permutation groups
\[
\Hol(\IZ/m\IZ)\rightarrow\prod_{i=1}^s{\Hol(\IZ/p_i^{k_i}\IZ)}.
\]
In particular,
\[
\CI(\Hol(\IZ/m\IZ))=\divideontimes_{i=1}^s{\CI(\Hol(\IZ/p_i^{k_i}\IZ))}.
\]
\end{proposition}

Therefore, it suffices to determine the cycle indices of holomorphs $\Hol(\IZ/p^k\IZ)$ of finite primary cyclic groups $\IZ/p^k\IZ$. This is settled in the results \cite[Theorems 1.2 and 1.3]{WX93a}, which are due to Wei, Gao and Yang and appear to have been published in \cite{WGY88a}, a two-page research announcement which the authors were unable to find through an electronic search. In order to make the proof of these results accessible, and because there seem to be a few typographical errors in \cite[Theorems 1.2 and 1.3]{WX93a}, we will formulate those results as Theorem \ref{primaryCyclicTheo} and give a detailed proof of them. Our proof provides more than just the cycle index of $\Hol(\IZ/p^k\IZ)$ itself; we determine the cycle type of a given affine permutation $x\mapsto ax+b$ of $\IZ/p^k\IZ$ in terms of $a$ and $b$ (see Proposition \ref{primaryCyclicProp}) and use this to determine the cycle index of $\Hol(\IZ/p^k\IZ)$. Before stating Proposition \ref{primaryCyclicProp}, we recall the following elementary number-theoretic lemma:

\begin{lemma}\label{primaryCyclicLem}
Let $p$ be a prime, and let $k$ be a positive integer.
\begin{enumerate}
\item If $p>2$, then for every positive integer $e$, we have $\nu_p((1+p)^e-1)=1+\nu_p(e)$. Consequently, the unit $1+p$ of the ring $\IZ/p^k\IZ$ is a generator of the Sylow $p$-subgroup of the (cyclic) unit group $(\IZ/p^k\IZ)^{\ast}$, and every unit $a\in(\IZ/p^k\IZ)^{\ast}$ can be written in a unique way as a product $u\cdot(1+p)^e$ where $u\in(\IZ/p^k\IZ)^{\ast}$ has order $\ord(a)_{p'}$ and $e\in\{0,1,\ldots,p^{k-1}-1\}$. One has $a-1\in(\IZ/p^k\IZ)^{\ast}$ if and only if $u>1$ (i.e., if and only if $\ord(a)_{p'}>1$). If $u=1$, then the order of the kernel of the multiplication by $a-1$ on $\IZ/p^k\IZ$ is $p^{1+\nu_p^{(k-1)}(e)}$.
\item Assume that $p=2$. For every positive integer $e$ and every $\epsilon\in\{0,1\}$, one has
\[
\nu_2((-1)^{\epsilon}5^e-1)=\begin{cases}1, & \text{if }\epsilon=1, \\ 2+\nu_2(e), & \text{if }\epsilon=0.\end{cases}
\]
Consequently, the unit group $(\IZ/2^k\IZ)^{\ast}$ is the direct product of its cyclic subgroups generated by $-1$ and $5$ respectively, which are of orders $2$ and $2^{k-2}$ respectively.
\end{enumerate}
\end{lemma}

\begin{proposition}\label{primaryCyclicProp}
The cycle types of affine permutations $x\mapsto ax+b$ of a primary cyclic group $\IZ/p^k\IZ$, where $a$ and $b$ are integers with $p\nmid a$, are as summarized in Tables \ref{table1} and \ref{table2}. For simplicity, we identify $a$ and $b$ with the corresponding elements of the finite ring $\IZ/p^k\IZ$, so we can avoid using congruences (when stating an equality, it will be clear from context if we mean an equality of integers or of elements of $\IZ/p^k\IZ$).
\end{proposition}

\begin{table}[h]\centering
\begin{tabular}{|c|c|}\hline
Case & Cycle type of $A:x\mapsto ax+b$ on $\IZ/p^k\IZ$ \\ \hline
$a\not\equiv1\Mod{p}$ & \thead{$x_1^1\cdot x_{\ord(a)_{p'}}^{(p^{k-\nu_p(\ord(a))}-1)/\ord(a)_{p'}}\cdot$ \\ $\prod_{s=1}^{\nu_p(\ord(a))}{x_{\ord(a)_{p'}p^s}^{p^{k-1-\nu_p(\ord(a))}(p-1)/\ord(a)_{p'}}}$} \\ \hline
\thead{$a\equiv1\Mod{p}$, \\ $\nu_p^{(k)}(b)\geq\nu_p^{(k)}(a-1)$} & \thead{$x_1^{p^{\nu_p^{(k)}(a-1)}}\cdot\prod_{s=1}^{k-\nu_p^{(k)}(a-1)}{x_{p^s}^{p^{\nu_p^{(k)}(a-1)-1}(p-1)}}$} \\ \hline
\thead{$a\equiv1\Mod{p}$, \\ $\nu_p^{(k)}(b)<\nu_p^{(k)}(a-1)$} & $x_{\aord(b)}^{p^k/\aord(b)}$ \\ \hline
\end{tabular}
\caption{Cycle types of affine permutations of $\IZ/p^k\IZ$ for $p>2$.}
\label{table1}
\end{table}

\begin{table}[h]\centering
\begin{tabular}{|c|c|}\hline
Case & \thead{Cycle type of $A:x\mapsto ax+b$ on $\IZ/2^k\IZ$, \\ $a=(-1)^{\epsilon}5^e$ with $\epsilon\in\{0,1\}$ and $e\in\{0,1,\ldots,2^{k-2}-1\}$} \\ \hline
$k=1$ & $x_{\aord(b)}^{2/\aord(b)}$ \\ \hline
$k=2$, $a=1$ & $x_{\aord(b)}^{4/\aord(b)}$ \\ \hline
$k=2$, $a=3$, $2\mid b$ & $x_1^2x_2^1$ \\ \hline
$k=2$, $a=3$, $2\nmid b$ & $x_2^2$ \\ \hline
$k\geq3$, $\epsilon=1$, $2\nmid b$ & $x_{2^{k-1-\nu_2^{(k-2)}(e)}}^{2^{1+\nu_2^{(k-2)}(e)}}$ \\ \hline
$k\geq3$, $\epsilon=1$, $2\mid b$ & $x_1^2x_2^{2^{2+\nu_2^{(k-3)}(e)}-1}\prod_{s=2}^{k-2-\nu_2^{(k-3)}(e)}{x_{2^s}^{2^{1+\nu_2^{(k-3)}(e)}}}$ \\ \hline
$k\geq3$, $\epsilon=0$, $\nu_2^{(k)}(b)<\nu_2^{(k)}(a-1)$ & $x_{\aord(b)}^{2^k/\aord(b)}$ \\ \hline
$k\geq3$, $\epsilon=0$, $\nu_2^{(k)}(b)\geq\nu_2^{(k)}(a-1)$ & $x_1^{2^{2+\nu_2^{(k-2)}(e)}}\cdot\prod_{s=1}^{k-2-\nu_2^{(k-2)}(e)}{x_{2^s}^{2^{1+\nu_2^{(k-2)}(e)}}}$ \\ \hline
\end{tabular}
\caption{Cycle types of affine permutations of $\IZ/2^k\IZ$.}
\label{table2}
\end{table}

\begin{proof}[Proof of Proposition \ref{primaryCyclicProp}]
Denote by $A$ the affine permutation $x\mapsto ax+b$ of $\IZ/p^k\IZ$, for fixed $a,b\in\IZ$ with $p\nmid a$. We will need to distinguish between the two cases \enquote{$p>2$} and \enquote{$p=2$}, but first, we make a general observation that applies to both cases. For each non-negative integer $\ell$, the $\ell$-th power (iterate) $A^{\ell}$ of $A$ is the affine permutation $x\mapsto a^{\ell}x+b(1+a+a^2+\cdots+a^{\ell-1})$ of $\IZ/p^k\IZ$. Therefore, the elements $x\in\IZ/p^k\IZ$ that lie on a cycle of $A$ of length a divisor of a given positive integer $\ell$ are the solutions $x$ of the equation
\begin{equation}\label{eq0}
a^{\ell}x+b(1+a+a^2+\cdots+a^{\ell-1})=x.
\end{equation}
We observe two things concerning formula (\ref{eq0}). Firstly, because its solution set consists of those points whose cycle length \emph{divides} $\ell$ rather than being equal to $\ell$, the actual number of (points on) cycles of a given length $\ell$ need to be determined using inclusion-exclusion counting. Secondly, formula (\ref{eq0}) implies that
\[
\ord(A)=\ord(a)\cdot\aord(b(1+a+a^2+\cdots+a^{\ord(a)-1})).
\]
This is because $\ord(A)$ is the smallest $\ell\in\IN^+$ such that equation (\ref{eq0}) holds for all $x\in\IZ/p^k\IZ$. Since this implies that $(a^{\ord(A)}-1)x$ must be constant, it follows that $\ord(a)\mid\ord(A)$, and we can write $\ord(A)=k\cdot\ord(a)$ for a suitable $k\in\IN^+$. Substituting $\ell:=k\cdot\ord(a)$ in equation (\ref{eq0}), we obtain
\[
x+b(1+a+a^2+\cdots+a^{k\ord(a)-1})=x.
\]
Since the sequence $(a^n)_{n\in\IN}$ is periodic with period $\ord(a)$, this can be equivalently rewritten into
\[
kb(1+a+a^2+\cdots+a^{\ord(a)-1})=0,
\]
and the smallest value of $k$ that makes this true is $\aord(b(1+a+a^2+\cdots+a^{\ord(a)-1}))$.

Now it is time for the aforementioned case distinction. First, assume that $p>2$. As long as $a-1$ is a unit, equation (\ref{eq0}) can be equivalently rewritten (by multiplying both sides by $a-1$) into the form
\begin{equation}\label{eq1}
(a^{\ell}-1)(a-1)x=-b\cdot(a^{\ell}-1).
\end{equation}
It suffices to consider those $\ell$ that divide $\ord(A)=\ord(a)\cdot\aord(b(1+a+a^2+\cdots+a^{\ord(a)-1}))$. We now make another case distinction.
\begin{enumerate}
\item Case: $a-1\in(\IZ/p^k\IZ)^{\ast}$ (equivalently, $a\not\equiv1\Mod{p}$ if $a$ is viewed as an integer). In particular, $a\not=1$, and $x\mapsto ax$ is an automorphism of $\IZ/p^k\IZ$ that fixes only the neutral element $0$. However, $1+a+a^2+\cdots+a^{\ord(a)-1}$ is clearly a fixed point of $x\mapsto ax$, hence equal to $0$, and thus $\ord(A)=\ord(a)\mid p^{k-1}(p-1)$. Since $a-1$ is a unit, we may consider equation (\ref{eq1}). For $\ell=1$, we can cancel $a^{\ell}-1=a-1$ on both sides, and since $a-1$ is a unit, the resulting equivalent equation $(a-1)x=-b$ has exactly one solution, whence $A$ has exactly one fixed point. For larger values of $\ell$, we make a subcase distinction.
\begin{enumerate}
\item Subcase: $p\nmid\ord(a)$. Then $\ord(a)\mid \phi(p^k)_{p'}=p-1$, and we consider values of $\ell>1$ with $\ell\mid\ord(a)$. If $\ell<\ord(a)$, then $a^{\ell}$ is a nontrivial unit whose order divides $p-1$, so that $a^{\ell}-1$ is a unit by Lemma \ref{primaryCyclicLem}(1), and equation (\ref{eq1}) has exactly one solution (the fixed point already mentioned above). In particular, there are no cycles of length $\ell$ then. If, on the other hand, $\ell=\ord(a)$, then $a^{\ell}-1=0$, and equation (\ref{eq1}) becomes $0=0$, which is universally solvable. This means that
\[
\CT(A)=x_1^1\cdot x_{\ord(a)}^{\frac{p^k-1}{\ord(a)}},
\]
which matches with the (more general) formula for the case \enquote{$a\not\equiv1\Mod{p}$} in Table \ref{table1}.
\item Subcase: $p\mid\ord(a)$. Then by Lemma \ref{primaryCyclicLem}(1), we have $a=u\cdot(1+p)^e$ for some unit $u$ with $\ord(u)=\ord(a)_{p'}$ and some $e\in\{1,2,\ldots,p^{k-1}-1\}$. If $\ord(u)=\ord(a)_{p'}\nmid\ell$, then $u$ is not completely annihilated when raising $a$ to the $\ell$-th power, whence $a^{\ell}-1$ is a unit by Lemma \ref{primaryCyclicLem}(1), and so as in the previous subcase, we conclude that there are no cycles of length $\ell$. Now assume that $\ord(a)_{p'}\mid\ell$. Then $\ell=\ord(a)_{p'}\cdot p^s$ for some $s\in\{0,1,\ldots,\nu_p(\ord(a))\}$. Consequently,
\[
a^{\ell}-1=u^{\ell}\cdot(1+p)^{e\ell}-1=(1+p)^{e\ell}-1,
\]
where the second equality uses that $\ord(u)=\ord(a)_{p'}$ (whence $u$ is annihilated when raised to the $\ell$-th power). Therefore, the kernel of the multiplication $\mu_{a^{\ell}-1}$ by $a^{\ell}-1$ on $\IZ/p^k\IZ$ has order
\[
p^{1+\nu_p^{(k-1)}(e\ell)}=p^{1+\min\{k-1,\nu_p(e)+s\}}=p^{1+\nu_p(e)+s},
\]
where the last equality uses that $s\leq\nu_p(\ord(a))=k-1-\nu_p(e)$. Since equation (\ref{eq1}), which can be equivalently rewritten as
\[
(a^{\ell}-1)x=-(a-1)^{-1}b(a^{\ell}-1),
\]
has at least one solution (such as $x=-(a-1)^{-1}b$), its solution set is a coset of $\ker(\mu_{a^{\ell}-1})$ and thus has size $p^{1+\nu_p(e)+s}=p^{k-\nu_p(\ord(a))+s}$. We conclude that the cycle type of $A$ is
\[
x_1^1\cdot x_{\ord(a)_{p'}}^{\frac{p^{k-\nu_p(\ord(a))}-1}{\ord(a)_{p'}}}\cdot\prod_{s=1}^{\nu_p(\ord(a))}{x_{\ord(a)_{p'}p^s}^{p^{k-1-\nu_p(\ord(a))}\cdot\frac{p-1}{\ord(a)_{p'}}}},
\]
which is just the formula for the case \enquote{$a\not\equiv1\Mod{p}$} in Table \ref{table1}.
\end{enumerate}
\item Case: $a-1\notin(\IZ/p^k\IZ)^{\ast}$ (equivalently, $a\equiv1\Mod{p}$ if $a$ is viewed as an integer). Then the factor $u$ in the unique factorization $a=u(1+p)^e$ as in Lemma \ref{primaryCyclicLem}(1) is trivial, so that $a=(1+p)^e$ for a unique $e\in\{0,1,\ldots,p^{k-1}-1\}$, and $\ord(a)$ is a power of $p$. Hence $\ord(A)\mid p^k$. We consider values of $\ell$ with $\ell\mid\ord(A)$; in particular, we can write $\ell=p^s$ for some $s\in\{0,1,\ldots,k\}$. We make a subcase distinction:
\begin{enumerate}
\item Subcase: $a=1$ (equivalently, $e=0$). Then $A$ is of the form $x\mapsto x+b$ for some $b\in\IZ/p^k\IZ$, and
\[
\CT(A)=x_{\aord(b)}^{\frac{p^k}{\aord(b)}},
\]
which matches with both formulas in the last two rows of Table \ref{table1} (the one in the penultimate row applies when $b=0$, the other when $b\not=0$ as an element of $\IZ/p^k\IZ$).
\item Subcase: $a\not=1$ (equivalently, $e>0$). When viewing $a$ as an integer, we have $\nu_p(a-1)=\nu_p^{(k)}(a-1)$. We reinterpret equation (\ref{eq0}) as the congruence
\begin{equation}\label{eq3}
(a^{\ell}-1)x\equiv-b(1+a+a^2+\cdots+a^{\ell-1})\Mod{p^k},
\end{equation}
and we are interested in counting its solutions modulo $p^k$. Note that $\nu_p(a-1)=1+\nu_p(e)$, that $\nu_p(a^{\ell}-1)=1+\nu_p(e)+s$, and consequently,
\[
\nu_p(1+a+a^2+\cdots+a^{\ell-1})=\nu_p(\frac{a^{\ell}-1}{a-1})=s.
\]
We will make use of the elementary fact that a linear congruence $cx\equiv d\Mod{m}$, where $c,d\in\IZ$, $m\in\IN^+$, and $x$ is a variable, has at least one solution if and only if $\gcd(c,m)\mid\gcd(d,m)$, in which case the number of solutions modulo $m$ is exactly $\gcd(c,m)$. We distinguish further between the following two subsubcases:
\begin{enumerate}
\item Subsubcase: $\nu_p(b)\geq\nu_p(a-1)$. Observe that we have $\nu_p^{(k)}(b)\geq\nu_p^{(k)}(a-1)$, because (as observed above) $\nu_p(a-1)=\nu_p^{(k)}(a-1)$, and $\nu_p^{(k)}(b)$ can only be distinct from $\nu_p(b)$ if $b=0$ (as an element of $\IZ/p^k\IZ$), in which case $\nu_p^{(k)}(b)=k\geq\nu_p^{(k)}(a-1)$. Moreover,
\begin{align*}
\nu_p(a^{\ell}-1)&=\nu_p(a-1)+\nu_p(1+a+\cdots+a^{\ell-1}) \\
&\leq\nu_p(b)+\nu_p(1+a+\cdots+a^{\ell-1}) \\
&=\nu_p(-b(1+a+\cdots+a^{\ell-1})),
\end{align*}
whence $\gcd(a^{\ell-1},p^k)$ divides $\gcd(-b(1+a+\cdots+a^{\ell-1}),p^k)$. Hence congruence (\ref{eq3}) is always solvable, and for $\ell=p^s$ with $s\in\{0,1,\ldots,k-\nu_p(a-1)\}$, it has exactly
\[
\gcd(p^k,a^{\ell}-1)=p^{\min\{k,1+\nu_p(e)+s\}}=p^{1+\nu_p(e)+s}=p^{\nu_p(a-1)+s}
\]
solutions modulo $p^k$. Recalling that $\nu_p(a-1)=\nu_p^{(k)}(a-1)$, we conclude that
\[
\CT(A)=x_1^{p^{\nu_p^{(k)}(a-1)}}\cdot\prod_{s=1}^{k-\nu_p^{(k)}(a-1)}{x_{p^s}^{p^{\nu_p^{(k)}(a-1)-1}(p-1)}},
\]
which is just the formula from the penultimate row in Table \ref{table1}.
\item Subsubcase: $\nu_p(b)<\nu_p(a-1)$. In particular, since $\nu_p(a-1)=\nu_p^{(k)}(a-1)\leq k$, we conclude that $\nu_p(b)<k$, whence $\nu_p(b)=\nu_p^{(k)}(b)$ and thus $\nu_p^{(k)}(b)<\nu_p^{(k)}(a-1)$. For $\ell=p^s$ with $s\in\{0,1,\ldots,k-\nu_p^{(k)}(b)-1\}$, we have
\begin{align*}
\nu_p^{(k)}(-b(1+a+\cdots+a^{\ell-1}))&=\min\{k,\nu_p^{(k)}(b)+s\}=\nu_p^{(k)}(b)+s \\
&<\min\{k,\nu_p^{(k)}(a-1)+s\}=\nu_p^{(k)}(a-1).
\end{align*}
Hence $\gcd(p^k,a^{\ell}-1)=p^{\nu_p^{(k)}(a-1)}$ does \emph{not} divide $\gcd(p^k,-b(1+a+\cdots+a^{\ell-1}))$, whence congruence (\ref{eq3}) has no solutions. However, for $\ell=p^{k-\nu_p^{(k)}(b)}$, we have
\[
\gcd(p^k,a^{\ell}-1)=p^k=\gcd(p^k,-b(1+a+\cdots+a^{\ell-1})),
\]
whence congruence (\ref{eq3}) is universally solvable. This means that
\[
\CT(A)=x_{\aord(b)}^{\frac{p^k}{\aord(b)}},
\]
which is just the formula from the last row of Table \ref{table1}.
\end{enumerate}
\end{enumerate}
\end{enumerate}

We now assume that $p=2$. We will also assume that $k\geq3$ (the cases $k=1,2$ are easily dealt with separately). There are unique $\epsilon\in\{0,1\}$ and $e\in\{0,1,\ldots,2^{k-2}-1\}$ such that $a=(-1)^{\epsilon}5^e$. As for $p>2$, we need to count the number of solutions of the equation
\[
(a^{\ell}-1)x=-b(1+a+a^2+\cdots+a^{\ell-1})
\]
over $\IZ/2^k\IZ$ for the various divisors $\ell$ of $\ord(A)$. Viewing $a$ and $b$ as integers, we can reinterpret this problem by considering the congruence
\begin{equation}\label{eq4}
(a^{\ell}-1)x\equiv-b(1+a+a^2+\cdots+a^{\ell-1})\Mod{2^k},
\end{equation}
and counting its number of solutions modulo $2^k$. Note that $\ord(A)$ is always a divisor of $2^k$; in particular, we only need to consider $\ell$ that are powers of $2$. We make a case distinction.
\begin{enumerate}
\item Case: $e=0$. We make a subcase distinction.
\begin{enumerate}
\item Subcase: $\epsilon=0$. Then $a=1$, so $A$ is of the form $x\mapsto x+b$. It follows that
\[
\CT(A)=x_{\aord(b)}^{\frac{2^k}{\aord(b)}},
\]
which matches with the formulas for $\CT(A)$ given in Table \ref{table2} both when $b\not=0$ (corresponding to the penultimate row in Table \ref{table2}) and when $b=0$ (corresponding to the last row in Table \ref{table2}).
\item Subcase: $\epsilon=1$. Then $a=-1$. We have $\ord(A)=\ord(a)\cdot\aord(b(1+(-1)))=2$. Hence we only need to consider $\ell=1$, for which congruence (\ref{eq4}) becomes $-2x\equiv-b\Mod{2^k}$. We make a subsubcase distinction:
\begin{enumerate}
\item Subsubcase: $2\nmid b$. Then $-2x\equiv-b\Mod{2^k}$ has no solutions, i.e., $A$ has no fixed points. Hence
\[
\CT(A)=x_2^{2^{k-1}},
\]
which matches with the formula for the case \enquote{$\epsilon=1$ and $2\nmid b$} from Table \ref{table2}.
\item Subsubcase: $2\mid b$. Then $-2x\equiv-b\Mod{2^k}$ has exactly $2$ solutions modulo $2^k$. It follows that
\[
\CT(A)=x_1^2x_2^{2^{k-1}-1},
\]
which matches with the formula for the case \enquote{$\epsilon=1$ and $2\mid b$} from Table \ref{table2}.
\end{enumerate}
\end{enumerate}
\item Case: $e\geq1$. Then $\nu_2(e)=\nu_2^{(k-2)}(e)=\nu_2^{(k-3)}(e)$, and since $a\not=1$, we have $\nu_2(a-1)=\nu_2^{(k)}(a-1)$. We make a subcase distinction.
\begin{enumerate}
\item Subcase: $\epsilon=1$ (i.e., $a\equiv3\Mod{4}$). Since $\nu_2(a-1)=1$ by Lemma \ref{primaryCyclicLem}(2), congruence (\ref{eq4}) is equivalent to
\begin{equation}\label{eq5}
(a^{\ell}-1)(a-1)x\equiv-b(a^{\ell}-1)\Mod{2^{k+1}}.
\end{equation}
Write $\ell=2^s$ with $s\geq0$. We make a subsubcase distinction.
\begin{enumerate}
\item Subsubcase: $2\nmid b$. For $s=0$ (i.e., $\ell=1$), since $\nu_2(a-1)=1$ by Lemma \ref{primaryCyclicLem}(2), congruence (\ref{eq5}) is equivalent to $(a-1)x\equiv-b\Mod{2^k}$. As $2\mid a-1$, this congruence has no solutions, so $A$ has no fixed points. Similarly, for $s=1,2,\ldots,k-2-\nu_2(e)$, we find that since $\nu_2(a^{\ell}-1)=2+s+\nu_2(e)$ by Lemma \ref{primaryCyclicLem}(2), congruence (\ref{eq5}) is equivalent to $(a-1)x\equiv-b\Mod{2^{k-1-s-\nu_2(e)}}$, which has no solutions. However, for $s=k-1-\nu_2(e)=k-1-\nu_2^{(k-2)}(e)$, congruence (\ref{eq5}) is universally solvable. Hence
\[
\CT(A)=x_{2^{k-1-\nu_2^{(k-2)}(e)}}^{2^{1+\nu_2^{(k-2)}(e)}}.
\]
\item Subsubcase: $2\mid b$. As in the previous subsubcase, for $s=0$, congruence (\ref{eq5}) is equivalent to $(a-1)x\equiv-b\Mod{2^k}$, which has precisely $2$ solutions modulo $2^k$, so $A$ has exactly $2$ fixed points. For $s=1,2,\ldots,k-2-\nu_2(e)=k-2-\nu_2^{(k-3)}(e)$, since $\nu_2(a^{\ell}-1)=2+s+\nu_2(e)=2+s+\nu_2^{(k-3)}(e)$ by Lemma \ref{primaryCyclicLem}(2), congruence (\ref{eq5}) is equivalent to $(a-1)x\equiv-b\Mod{2^{k-1-s-\nu_2(e)}}$, which has $2^{2+s+\nu_2(e)}=2^{2+s+\nu_2^{(k-3)}(e)}$ solutions modulo $2^k$. It follows that
\[
\CT(A)=x_1^2\cdot x_2^{2^{2+\nu_2^{(k-3)}(e)}-1}\cdot\prod_{s=2}^{k-2-\nu_2^{(k-3)}(e)}{x_{2^s}^{2^{1+\nu_2^{(k-3)}(e)}}}.
\]
\end{enumerate}
\item Subcase: $\epsilon=0$ (i.e., $a\equiv1\Mod{4}$). Then $\nu_2(a-1)=2+\nu_2(e)$ by Lemma \ref{primaryCyclicLem}(2), so congruence (\ref{eq4}) is equivalent to
\begin{equation}\label{eq6}
(a^{\ell}-1)(a-1)x\equiv-b(a^{\ell}-1)\Mod{2^{k+2+\nu_2(e)}}.
\end{equation}
Write $\ell=2^s$ with $s\geq0$. We make a subsubcase distinction.
\begin{enumerate}
\item Subsubcase: $\nu_2(b)<\nu_2(a-1)=2+\nu_2(e)$. In particular, since $\nu_2(a-1)=\nu_2^{(k)}(a-1)\leq k$, we have $\nu_2(b)<k$ and thus $\nu_2(b)=\nu_2^{(k)}(b)$. By Lemma \ref{primaryCyclicLem}(2), we have $\nu_2(a^{\ell}-1)=2+\nu_2(e)+s$. It follows that for $s=0,1,\ldots,k$, congruence (\ref{eq6}) is equivalent to $(a-1)x\equiv-b\Mod{2^{k-s}}$. If $s<k-\nu_2(b)$, then this congruence has no solutions, whereas for $s=k-\nu_2(b)$, it is universally solvable. Hence
\[
\CT(A)=x_{2^{k-\nu_2(b)}}^{2^{\nu_2(b)}}=x_{\aord(b)}^{\frac{2^k}{\aord(b)}},
\]
in accordance with the formula for the case \enquote{$\epsilon=0$ and $\nu_2^{(k)}(b)<\nu_2^{(k)}(a-1)$} from Table \ref{table2}.
\item Subsubcase: $\nu_2(b)\geq\nu_2(a-1)=2+\nu_2(e)$. Since $\nu_2(a-1)=\nu_2^{(k)}(a-1)$, this implies that $\nu_2^{(k)}(b)\geq\nu_2^{(k)}(a-1)$, because $\nu_2^{(k)}(b)$ can only be distinct from $\nu_2(b)$ if $b=0$ (as an element of $\IZ/2^k\IZ$), in which case $\nu_2^{(k)}(b)=k\geq\nu_2^{(k)}(a-1)$. As in the previous subsubcase, for $s=0,1,\ldots,k$, congruence (\ref{eq6}) is equivalent to $(a-1)x\equiv-b\Mod{2^{k-s}}$. This congruence has exactly $2^{s+\nu_2(a-1)}=2^{s+2+\nu_2(e)}=2^{s+2+\nu_2^{(k-2)}(e)}$ solutions modulo $2^k$ for each $s=0,1,\ldots,k-2-\nu_2(e)=k-2-\nu_2^{(k-2)}(e)$. It follows that
\[
\CT(A)=x_1^{2^{2+\nu_2^{(k-2)}(e)}}\cdot\prod_{s=1}^{k-2-\nu_2^{(k-2)}(e)}{x_{2^s}^{2^{1+\nu_2^{(k-2)}(e)}}},
\]
which matches with the formula for the case \enquote{$\epsilon=0$ and $\nu_2^{(k)}(b)\geq\nu_2^{(k)}(a-1)$} in Table \ref{table2}.
\end{enumerate}
\end{enumerate}
\end{enumerate}
\end{proof}

We are now ready to determine the cycle index of $\Hol(\IZ/p^k\IZ)$:

\begin{theorem}\label{primaryCyclicTheo}(Wei-Gao-Yang \cite{WGY88a}; see also \cite[Theorems 1.2 and 1.3]{WX93a})
The following hold:
\begin{enumerate}
\item For each odd prime $p$, the cycle index of $\Hol(\IZ/p^k\IZ)$ is
\begin{align*}
\frac{1}{p^{2k-1}(p-1)}\cdot\biggl(&\sum_{w=1}^k{\left(p^{2w-2}(p-1)\cdot x_{p^w}^{p^{k-w}}\right)}+ \\
&\sum_{w=0}^{k-1}{\left(\phi(p^w)p^w\cdot x_1^{p^{k-w}}\prod_{u=1}^w{x_{p^u}^{p^{k-w-1}(p-1)}}\right)}+ \\
&\sum_{w=0}^{k-1}\sum_{1<l\mid{p-1}}{\left(p^k\phi(p^w)\phi(l)\cdot x_1x_l^{\frac{p^{k-w}-1}{\ell}}\prod_{u=1}^w{x_{\ell p^u}^{\frac{p^{k-1-w}(p-1)}{l}}}\right)}\biggr).
\end{align*}
\item The cycle index of $\Hol(\IZ/2\IZ)$ is
\[
\frac{1}{2}x_1^2+\frac{1}{2}x_2.
\]
\item The cycle index of $\Hol(\IZ/4\IZ)$ is
\[
\frac{1}{8}x_1^4+\frac{1}{4}x_1^2x_2+\frac{3}{8}x_2^2+\frac{1}{4}x_4.
\]
\item For $k\geq3$, the cycle index of $\Hol(\IZ/2^k\IZ)$ is
\begin{align*}
\frac{1}{2^{2k-1}}\cdot\biggl(&2^{2k-3}x_{2^k}+\sum_{w=1}^{k-1}{(2^{2w-2}+\phi(2^{w-1})2^{k-1})x_{2^w}^{2^{k-w}}}+ \\
&\sum_{w=0}^{k-2}{(\phi(2^w)2^wx_1^{2^{k-w}}\prod_{u=1}^w{x_{2^u}^{2^{k-1-w}}})}+2^kx_1^2x_2^{2^{k-1}-1}+ \\
&\sum_{w=2}^{k-2}{2^{k+w-2}x_1^2x_2^{2^{k-w}-1}\prod_{u=2}^w{x_{2^u}^{2^{k-1-w}}}}\biggr).
\end{align*}
\end{enumerate}
\end{theorem}

Before proving Theorem \ref{primaryCyclicTheo}, we comment on the formulas in \cite[Theorems 1.2 and 1.3]{WX93a}. Our formula from Theorem \ref{primaryCyclicTheo}(1), which arises naturally from our argument, is a slight variation of the formula in \cite[Theorem 1.2]{WX93a}. For the reader's convenience (especially because the type-setting in \cite[Theorem 1.2]{WX93a} seems to be a bit off), we display that formula now, replacing $\alpha$ by $k$:
\begin{align*}
\frac{1}{p^{2k-1}(p-1)}\cdot\biggl(&\sum_{w=1}^k{\left(p^{2(w-1)}(p-1)x_{p^w}^{p^{k-w}}\right)}+ \\
&\sum_{w=0}^{k-1}\sum_{l\mid p-1}{p^{w+\delta_{[l>1]}(k-w)}\phi(p^wl)x_1x_l^{\frac{p^{k-w-1}-1}{l}}\cdot\left(\prod_{u=0}^w{x_{p^ul}}\right)^{\frac{p^{k-w-1}(p-1)}{l}}}\biggr)
\end{align*}
The sum in the first line of this formula (with running index $w$ ranging from $1$ to $k$) is the same as the sum in the first line of the formula in Theorem \ref{primaryCyclicTheo}(1). The sum in the second line of Theorem \ref{primaryCyclicTheo}(1) corresponds to the second-line sum of \cite[Theorem 1.2]{WX93a} when $l=1$, and the sum in the third line of Theorem \ref{primaryCyclicTheo}(1) is the part of the second-line sum of \cite[Theorem 1.2]{WX93a} corresponding to $l>1$; note that the exponents of $x_l$ are only seemingly different in the two cases, because the factor of the subsequent product corresponding to $u=0$ (which is absent in our formula) adds to it.

Let us also compare our formula in Theorem \ref{primaryCyclicTheo}(4) with the original one in \cite[Theorem 1.3]{WX93a}. After \enquote{ironing out} a few typesetting issues in the latter, it looks like this:
\begin{align*}
\frac{1}{2^{2k-1}}\cdot\biggl(&2^{2(k-1)}x_{2^k}+\sum_{w=1}^{k-1}{\left(2^{2(w-1)}+\phi(2^{w-1})2^{k-1}\right)x_{2^w}^{2^{k-w}}} \\
&+\sum_{w=0}^{k-2}{\left(\phi(2^w)\left(2^wx_1^{2^{k-w}}+2^{k-1}x_1^2x_2^{2^{k-w-1}-1}\right)\cdot\left(\prod_{u=1}^w{x_{2^u}}\right)^{2^{k-w-1}}\right)}\biggr).
\end{align*}
Attentive readers will notice the difference (by a factor of $2$) between the coefficients of the first terms in this formula and the one from Theorem \ref{primaryCyclicTheo}(4). The correct number of $2^k$-cycles in $\Hol(\IZ/2^k\IZ)$ is the one from our formula, $2^{2k-3}$; a quick way to see this is by applying Knuth's Theorem (see Subsection \ref{subsec4P3}), according to which the affine permutation $x\mapsto ax+b$ of $\IZ/2^k\IZ$ is a $2^k$-cycle if and only if $a\equiv1\Mod{4}$ and $2\nmid b$. This means there are $2^{k-2}$ possibilities for $a$ (the elements of the cyclic subgroup of $(\IZ/2^k\IZ)^{\ast}$ generated by $5$) and $\phi(2^k)=2^{k-1}$ possibilities for $b$, and those can be combined independently with each other. Apart from this small difference, the formulas in \cite[Theorem 1.3]{WX93a} and Theorem \ref{primaryCyclicTheo}(4) match:
\begin{itemize}
\item The second summands in the first lines of both formulas (i.e., the sums with running index $w$ ranging from $1$ to $k-1$) are the same.
\item The first summand in the second line of the formula in Theorem \ref{primaryCyclicTheo}(4) is the \enquote{first half} of the sum in the second line of Wei-Gao-Yang's formula, obtained when \enquote{multiplying out} using distributivity and then splitting the sum in two.
\item When adding the \enquote{lone} term $2^kx_1^2x_2^{2^{k-1}-1}$ in our formula with the sum in the last line of our formula, one obtains the \enquote{second half} of the sum in the second line of Wei-Gao-Yang's formula. Note that compared to their formula, we pulled out the factor for $u=1$ in the subsequent product (because it contributes to the power of $x_2$).
\end{itemize}
We decided to display the formulas the way we do in Theorem \ref{primaryCyclicTheo} because this format makes it easier to read off the coefficient of a given term of the respective cycle index -- note that our formulas always display the monomials in full, and monomials appearing at different places in the same formula are never equal.

The following simple lemma will be useful for the proof of Theorem \ref{primaryCyclicTheo}:

\begin{lemma}\label{simpleLem}
Let $p$ be a prime and $k$ a positive integer. Moreover, let $\ell\in\{0,1,\ldots,k\}$. Then the following hold:
\begin{enumerate}
\item The number of $x\in\{0,1,\ldots,p^k-1\}$ such that $\nu_p^{(k)}(x)=\ell$ is exactly $\phi(p^{k-\ell})$.
\item The number of $x\in\{0,1,\ldots,p^k-1\}$ such that $\nu_p^{(k)}(x)\geq\ell$ is exactly $p^{k-\ell}$.
\end{enumerate}
\end{lemma}

\begin{proof}[Proof of Theorem \ref{primaryCyclicTheo}]
For statement (1): We go through the cases distinguished in Table \ref{table1} and count the number of affine permutations $x\mapsto ax+b$ of a given cycle type in each of them. We follow the order in which the summands appear in our formula, which differs from the order of rows in Table \ref{table1}.
\begin{enumerate}
\item Case: $a\equiv1\Mod{p}$ and $\nu_p^{(k)}(b)<\nu_p^{(k)}(a-1)$. Set $w:=\nu_p(\aord(b))\in\{1,2,\ldots,k\}$. According to the third row of Table \ref{table1}, all affine permutations $x\mapsto ax+b$ corresponding to a fixed value of $w$ have the same cycle type, namely $x_{p^w}^{p^{k-w}}$. To count these permutations, note that $x\mapsto ax+b$ corresponds to $w$ if and only if $\nu_p^{(k)}(b)=k-w$ and $\nu_p^{(k)}(a-1)\geq\nu_p^{(k)}(b)+1=k-w+1$. By Lemma \ref{simpleLem}(1), this means that there are
\[
\phi(p^{k-(k-w)})=\phi(p^w)=p^{w-1}(p-1)
\]
possibilities for $b$, and
\[
p^{k-(k-w+1)}=p^{w-1}
\]
possibilities for $a-1$ (and thus for $a$). The number of affine permutations for fixed $w$ is therefore
\[
p^{w-1}\cdot p^{w-1}(p-1)=p^{2w-2}(p-1),
\]
and noting that $|\Hol(\IZ/p^k\IZ)|=p^{2k-1}(p-1)$, this case accounts for the following contribution to the cycle index:
\[
\frac{1}{p^{2k-1}(p-1)}\sum_{w=1}^k{p^{2w-2}(p-1)x_{p^w}^{p^{k-w}}}.
\]
\item Case: $a\equiv1\Mod{p}$ and $\nu_p^{(k)}(b)\geq\nu_p^{(k)}(a-1)$. Set $w:=\nu_p(\ord(a))\in\{0,1,\ldots,k-1\}$. Writing $a=(1+p)^e$ with $e\in\{0,1,\ldots,p^{k-1}-1\}$, we find by Lemma \ref{primaryCyclicLem}(1) that
\[
w=\nu_p(\ord(a))=k-1-\nu_p^{(k-1)}(e)=k-(1+\nu_p^{(k-1)}(e))=k-\nu_p^{(k)}(a-1),
\]
i.e., that $\nu_p^{(k)}(a-1)=k-w$. By the second row of Table \ref{table1}, this means that for fixed $w$, the corresponding affine permutations all have the cycle type
\[
x_1^{p^{k-w}}\prod_{u=1}^w{x_{p^u}^{p^{k-w-1}(p-1)}}.
\]
Moreover, $x\mapsto ax+b$ corresponds to a fixed $w$ if and only if
\[
\nu_p^{(k)}(a-1)=k-w\text{ and }\nu_p^{(k)}(b)\geq\nu_p^{(k)}(a-1)=k-w.
\]
Hence, by Lemma \ref{simpleLem}, the number of such affine permutations is
\[
\phi(p^{k-(k-w)})\cdot p^{k-(k-w)}=\phi(p^w)p^w,
\]
and this case contributes
\[
\frac{1}{p^{2k-1}(p-1)}\sum_{w=0}^{k-1}{\phi(p^w)p^wx_1^{p^{k-w}}\prod_{u=1}^w{x_{p^u}^{p^{k-w-1}(p-1)}}}
\]
to the cycle index.
\item Case: $a\not\equiv1\Mod{p}$. Write $\ord(a)=l\cdot p^w$ where $l=\ord(a)_{p'}$ is a divisor of $p-1$ with $l>1$, and $w=\nu_p(\ord(a))\in\{0,1,\ldots,k-1\}$. According to the first row of Table \ref{table1}, for fixed $l$ and $w$, all associated affine permutations $x\mapsto ax+b$ have the cycle type
\[
x_1x_l^{\frac{p^{k-w}-1}{l}}\prod_{u=1}^w{x_{lp^u}^{\frac{p^{k-1-w}(p-1)}{l}}}.
\]
For counting those permutations $x\mapsto ax+b$, note that since the unit group $(\IZ/p^k\IZ)^{\ast}$ is cyclic, it has exactly $\phi(lp^w)$ elements of order $lp^w$, so the number of possibilities for $a$ is $\phi(lp^w)$. There are no restrictions on $b$ in this case, and so the number of associated affine permutations is
\[
\phi(lp^w)\cdot p^k=p^k\phi(p^w)\phi(l),
\]
and the overall contribution to the cycle index is
\[
\sum_{w=0}^{k-1}\sum_{1<l\mid p-1}{p^k\phi(p^w)\phi(l)x_1x_l^{\frac{p^{k-w}-1}{l}}\prod_{u=1}^w{x_{lp^u}^{\frac{p^{k-1-w}(p-1)}{l}}}}
\]
\end{enumerate}

Statements (2) and (3) can be verified by some simple computations, which we omit. This leaves us with the proof of statement (4), which we will carry out now. Similarly to the proof of statement (1), we go through the cases distinguished in Table \ref{table2} and count the number of affine permutations $x\mapsto ax+b$ of $\IZ/2^k\IZ$ of a given cycle type, with $a=(-1)^{\epsilon}5^e$ where $\epsilon\in\{0,1\}$ and $e\in\{0,1,\ldots,2^{k-2}-1\}$.
\begin{enumerate}
\item Case: $\epsilon=0$ and $\nu_2^{(k)}(b)<\nu_2^{(k)}(a-1)$. Set $w:=\nu_2^{(k)}(\aord(b))=k-\nu_2^{(k)}(b)\in\{1,2,\ldots,k\}$. By the penultimate row of Table \ref{table2}, all affine permutations $x\mapsto ax+b$ for a fixed $w$ have the cycle type $x_{2^w}^{2^{k-w}}$. To count the number of such permutations, note that since $\nu_2^{(k)}(b)=k-w$, the number of matching values for $b$ is $\phi(2^{k-(k-w)})=\phi(2^w)=2^{w-1}$ by Lemma \ref{simpleLem}(1). Moreover, $\nu_2^{(k)}(a-1)\geq\nu_2^{(k)}(b)+1=k-w+1$. By Lemma \ref{simpleLem}(2), the total number of $x\in\{0,1,\ldots,2^k-1\}$ with $\nu_2^{(k)}(x-1)\geq k-w+1$ is $2^{k-(k-w+1)}=2^{w-1}$. However, we only want those $x$ that are also of the form $5^{e'}$ for some $e'\in\{0,1,\ldots,2^{k-2}-1\}$. If $w<k$, then all $x$ with $\nu_2^{(k)}(x-1)\geq k-w+1>1$ are of that form, because by Lemma \ref{primaryCyclicLem}(2), otherwise $\nu_2^{(k)}(x-1)=1$. Hence, as long as $w<k$, the number of affine permutations with $w$ fixed is $2^{2w-2}$. However, if $w=k$, then the condition $\nu_2^{(k)}(a-1)\geq k-w+1=1$ is satisfied by all units $a$. Since we only want to count those that are of the form $5^{e'}$, we conclude that the number of matching $a$ is $2^{k-2}$, and the number of matching affine permutations is $2^{k+w-3}=2^{2k-3}$. Noting that $|\Hol(\IZ/2^k\IZ)|=2^{2k-1}$, we obtain that the total contribution of this case to the cycle index is
\[
\frac{1}{2^{2k-1}}\left(2^{2k-3}x_{2^k}+\sum_{w=1}^{k-1}{2^{2w-2}x_{2^w}^{2^{k-w}}}\right).
\]
\item Case: $\epsilon=1$ and $2\nmid b$. Set $v:=\nu_2^{(k)}(b)\geq\nu_2^{(k)}(a-1)$. By the last row of Table \ref{table2}, for fixed $v$, all affine permutations $x\mapsto ax+b$ have the cycle type $x_{2^{k-1-v}}^{2^{1+v}}$. Moreover, the number of matching $b$ is $2^{k-1}$, and the number of matching $a=-5^e$ is $\phi(2^{k-2-v})$. The overall contribution of this case to the cycle index is
\[
\frac{1}{2^{2k-1}}\sum_{v=0}^{k-2}{\phi(2^{k-2-v})2^{k-1}x_{2^{k-1-v}}^{2^{1+v}}}=\frac{1}{2^{2k-1}}\sum_{w=1}^{k-1}{\phi(2^{w-1})2^{k-1}x_{2^w}^{2^{k-w}}},
\]
where the equality is achieved through the substitution $w:=k-1-v$.
\item Case: $\epsilon=0$ and $\nu_2^{(k)}(b)\geq\nu_2^{(k)}(a-1)$. Set $v:=\nu_2^{(k-2)}(e)\in\{0,1,\ldots,k-2\}$. By the last row of Table \ref{table2}, for fixed $v$, all affine permutations $x\mapsto ax+b$ have the cycle type $x_1^{2^{2+v}}\prod_{u=1}^{k-2-v}{x_{2^u}^{2^{1+v}}}$. Moreover, by Lemma \ref{primaryCyclicLem}(2), $\nu_2^{(k)}(a-1)=2+v>1$, so the number of matching $a$ is $\phi(2^{k-(v+2)})=\phi(2^{k-v-2})$ by Lemma \ref{simpleLem}(1), whereas the number of matching $b$ is $2^{k-v-2}$ by Lemma \ref{simpleLem}(2). The overall contribution to the cycle index is
\begin{align*}
&\frac{1}{2^{2k-1}}\sum_{v=0}^{k-2}{\phi(2^{k-2-v})2^{k-2-v}x_1^{2^{2+v}}\prod_{u=1}^{k-2-v}{x_{2^u}^{2^{1+v}}}} \\
&=\frac{1}{2^{2k-1}}\sum_{w=0}^{k-2}{\phi(2^w)2^wx_1^{2^{k-w}}\prod_{u=1}^w{x_{2^u}^{2^{k-1-w}}}},
\end{align*}
where the equality is achieved through the substitution $w:=k-2-v$.
\item Case: $\epsilon=1$ and $2\mid b$. Set $v:=\nu_2^{(k-3)}(e)\in\{0,1,\ldots,k-3\}$. By the third-to-last row of Table \ref{table2}, all affine permutations $x\mapsto ax+b$ for a fixed $v$ have the cycle type $x_1^2x_2^{2^{2+v}-1}\prod_{u=2}^{k-2-v}{x_{2^u}^{2^{1+v}}}$. In order to count these permutations, we distinguish two subcases:
\begin{itemize}
\item If $v<k-3$, then $v=\nu_2^{(k-2)}(e)$, and by Lemma \ref{simpleLem}(1), the number of matching $e$ (and hence the number of matching $a$) is $\phi(2^{k-2-v})=2^{k-3-v}$, whereas the number of matching $b$ is $2^{k-1}$.
\item If $v=k-3$, then $\nu_2^{(k-2)}(e)\in\{k-3,k-2\}$, and the associated values of $e$ are $2^{k-3}$ and $0$. Therefore, there are $2$ matching values for $a$, and still $2^{k-1}$ matching values for $b$.
\end{itemize}
In total, we obtain the following contribution of this case to the cycle index:
\begin{align*}
&\frac{1}{2^{2k-1}}\left(2^kx_1^2x_2^{2^{k-1}-1}+\sum_{v=0}^{k-4}{x_1^2x_2^{2^{2+v}-1}\prod_{u=2}^{k-2-v}{x_{2^u}^{2^{1+v}}}}\right)= \\
&\frac{1}{2^{2k-1}}\left(2^kx_1^2x_2^{2^{k-1}-1}+\sum_{w=2}^{k-2}{x_1^2x_2^{2^{k-w}-1}\prod_{u=2}^w{x_{2^u}^{2^{k-1-w}}}}\right),
\end{align*}
where the equality is achieved through the substitution $w:=k-2-v$.
\end{enumerate}
\end{proof}

\begin{example}\label{hol12Ex}
As an application of the theory presented in this section, we compute the cycle index of $\Hol(\IZ/12\IZ)$. By Theorem \ref{primaryCyclicTheo}(1),
\[
\CI(\Hol(\IZ/3\IZ))=\frac{1}{6}(2x_3+x_1^3+3x_1x_2)=\frac{1}{6}x_1^3+\frac{1}{2}x_1x_2+\frac{1}{3}x_3=\CI(\Sym(3)),
\]
and by Theorem \ref{primaryCyclicTheo}(3),
\[
\CI(\Hol(\IZ/4\IZ))=\frac{1}{8}x_1^4+\frac{1}{4}x_1^2x_2+\frac{3}{8}x_2^2+\frac{1}{4}x_4.
\]
By Proposition \ref{crtProp},
\[
\CI(\Hol(\IZ/12\IZ))=\CI(\Hol(\IZ/3\IZ))\divideontimes\CI(\Hol(\IZ/4\IZ)).
\]
Table \ref{hol12Table} lists all possible $\divideontimes$-products of a term $\tau_1$ of $\CI(\Hol(\IZ/3\IZ))$ with a term $\tau_2$ of $\CI(\Hol(\IZ/4\IZ))$ (cf.~Table \ref{mathrefTable2} in Example \ref{permProductEx}).

\begin{table}[h]\centering
\begin{tabular}{|c|c|c|c|}\hline
\diagbox{$\tau_2$}{$\tau_1$} & $\frac{1}{6}x_1^3$ & $\frac{1}{2}x_1x_2$ & $\frac{1}{3}x_3$ \\ \hline
$\frac{1}{8}x_1^4$ & $\frac{1}{48}x_1^{12}$ & $\frac{1}{16}x_1^4x_2^4$ & $\frac{1}{24}x_3^4$ \\ \hline
$\frac{1}{4}x_1^2x_2$ & $\frac{1}{24}x_1^6x_2^3$ & $\frac{1}{8}x_1^2x_2^5$ & $\frac{1}{12}x_3^2x_6$ \\ \hline
$\frac{3}{8}x_2^2$ & $\frac{1}{16}x_2^6$ & $\frac{3}{16}x_2^6$ & $\frac{1}{8}x_6^2$ \\ \hline
$\frac{1}{4}x_4$ & $\frac{1}{24}x_4^3$ & $\frac{1}{8}x_4^3$ & $\frac{1}{12}x_{12}$ \\ \hline
\end{tabular}
\caption{$\divideontimes$-products of terms $\tau_1$ of $\CI(\Hol(\IZ/3\IZ))$ with terms $\tau_2$ of $\CI(\Hol(\IZ/4\IZ))$.}
\label{hol12Table}
\end{table}

Based on this, we conclude that
\begin{align*}
\CI(\Hol(\IZ/12\IZ))=&\frac{1}{48}x_1^{12}+\frac{1}{24}x_1^6x_2^3+\frac{1}{16}x_1^4x_2^4+\frac{1}{8}x_1^2x_2^5+\frac{1}{4}x_2^6+\frac{1}{24}x_3^4+\frac{1}{12}x_3^2x_6+ \\
&\frac{1}{6}x_4^3+\frac{1}{8}x_6^2+\frac{1}{12}x_{12}.
\end{align*}
\end{example}

\section{Some other helpful results}\label{sec4}

\subsection{Conjugacy in wreath products}\label{subsec4P1}

Conjugacy in wreath products was characterized by James and Kerber in \cite[Theorem 4.2.8, p.~141]{JK81a}. Here, we recall the formulation of their result given in \cite[Subsection 2.2]{Bor19a}.

\begin{deffinition}\label{bcpcDef}
Let $G$ be a group, $d\in\IN^+$, and let $\vec{g}=\sigma(g_1,\ldots,g_d)\in G\wr\Sym(d)$.

\begin{enumerate}
\item For a length $\ell$ cycle $\zeta=(i_1,\ldots,i_{\ell})$ of $\sigma$, we define $\fcpc_{\zeta}(\vec{g}):=(g_{i_1}g_{i_2}\cdots g_{i_{\ell}})^G$, the \emph{forward cycle product class of $\vec{g}$ with respect to $\zeta$}, a $G$-conjugacy class.
\item For $\ell\in\{1,\ldots,n\}$, we denote by $M_{\ell}(\vec{g})$ the (possibly empty) multiset of the $\fcpc_{\zeta}(\vec{g})$, where $\zeta$ runs through the length $\ell$ cycles of $\sigma$.
\end{enumerate}
\end{deffinition}

\begin{lemmma}\label{wpcLem}(cf.~\cite[Lemma 2.2.5]{Bor19a})
Let $G$ be a group, let $d\in\IN^+$, let
\[
g_1,\ldots,g_d,h_1,\ldots,h_d\in G
\]
and $\sigma,\upsilon\in\Sym(d)$. The following are equivalent:

\begin{enumerate}
\item The two elements $\vec{g}:=\sigma(g_1,\ldots,g_d)$ and $\vec{h}:=\upsilon(h_1,\ldots,h_d)$ of $G\wr\Sym(d)$ are conjugate.
\item $\sigma$ and $\upsilon$ have the same cycle type (i.e., are conjugate in $\Sym(d)$), and for all $\ell\in\{1,\ldots,d\}$, the equality of multisets $M_{\ell}(\vec{g})=M_{\ell}(\vec{h})$ holds.
\end{enumerate}
\end{lemmma}

We note that the differences between the results as stated here and their versions in \cite[Subsection 2.2]{Bor19a} (where \enquote{backward cycle products} were used instead) is due to us working with right actions here.

\subsection{Conjugacy in holomorphs of finite cyclic groups}\label{subsec4P2}

Consider the finite cyclic group $\IZ/m\IZ$. Recall that by Proposition \ref{crtProp}, if $m=p_1^{k_1}\cdots p_r^{k_r}$ is the factorization of $m$ into pairwise coprime prime powers, then
\[
\Hol(\IZ/m\IZ)\cong\prod_{t=1}^r{\Hol(\IZ/p_t^{k_t}\IZ)}
\]
as permutation groups. Since conjugacy in direct products (of abstract groups) is characterized by component-wise conjugacy, it suffices to consider the case $m=p^k$ for some prime $p$ and some $k\in\IN^+$.

Write $\Hol(\IZ/p^k\IZ)=(\IZ/p^k\IZ)^{\ast}\ltimes\IZ/p^k\IZ$. Let $(a,x),(b,y)\in\Hol(\IZ/p^k\IZ)$. A necessary condition for these two pairs to be conjugate in $\Hol(\IZ/p^k\IZ)$ is that $a$ and $b$ are conjugate in the abelian group $(\IZ/p^k\IZ)^{\ast}$, i.e., that $a=b$. The problem thus is to determine for which $x,y\in\IZ/p^k\IZ$ it is the case that the pairs $(a,x),(a,y)\in\Hol(\IZ/p^k\IZ)$ are conjugate.

Let $(c,z)\in\Hol(\IZ/p^k\IZ)$. Then
\begin{align*}
(a,x)^{(c,z)} &=(c,z)^{-1}(a,x)(c,z)=(c^{-1},-c^{-1}z)(a,x)(c,z)=(c^{-1}a,-c^{-1}az+x)(c,z) \\
&=(c^{-1}ac,-c^{-1}acz+cx+z)=(a,(1-a)z+cx).
\end{align*}
Hence $(a,x)$ and $(a,y)$ are conjugate in $\Hol(\IZ/p^k\IZ)$ if and only if $y\in\{(1-a)z+cx\mid z\in\IZ/p^k\IZ,c\in(\IZ/p^k\IZ)^{\ast}\}$. We distinguish two cases:
\begin{enumerate}
\item Case: $x\in(1-a)\IZ/p^k\IZ$. Then $\{(1-a)z+cx\mid z\in\IZ/p^k\IZ,c\in(\IZ/p^k\IZ)^{\ast}\}=(1-a)\IZ/p^k\IZ$.
\item Case: $x\notin(1-a)\IZ/p^k\IZ$. Then since the subgroup lattice of $\IZ/p^k\IZ$ is a line, one has that $(1-a)\IZ/p^k\IZ$ is a proper subgroup of $\langle x\rangle$. In particular, since $cx$ is a generator of $\langle x\rangle$ but $(1-a)z$ is a non-generator of $\langle x\rangle$, each sum $(1-a)z+cx$ is a generator of $\langle x\rangle$, so $\{(1-a)z+cx\mid z\in\IZ/p^k\IZ,c\in(\IZ/p^k\IZ)^{\ast}\}=\langle x\rangle\setminus p\langle x\rangle$.
\end{enumerate}
In particular, the number of conjugacy classes of elements of $\Hol(\IZ/p^k\IZ)$ of the form $(a,x)$ for a fixed $a\in(\IZ/p^k\IZ)^{\ast}$ equals $\log_p{|(\IZ/p^k\IZ):(1-a)(\IZ/p^k\IZ)|}$ (for example, if the multiplication by $a$ is fixed-point-free, then there is precisely one such conjugacy class).

\subsection{Knuth's Theorem}\label{subsec4P3}

The following important theorem characterizes when a linear congruential generator with modulus $m$ achieves the maximum period length of $m$:

\begin{theoremm}\label{knuthTheo}(Knuth, \cite[Section 3.2.1.2, Theorem A]{Knu98a})
Let $m\in\IN^+$, and let $a,b\in\IZ$ with $\gcd(a,m)=1$. Denote by $\rad(m)$ the squarefree radical of $m$ (the product of the primes dividing $m$), and set
\[
\rad'(m):=\begin{cases}\rad(m), & \text{if }4\nmid m, \\ 2\cdot\rad(m), & \text{if }4\mid m.\end{cases}
\]
The affine permutation $\lambda(a,b):\IZ/m\IZ\rightarrow\IZ/m\IZ$, $x\mapsto ax+b$, is an $m$-cycle if and only if $a\equiv1\Mod{\rad'(m)}$ and $\gcd(b,m)=1$.
\end{theoremm}

Observe that for every generator $b$ of $\IZ/m\IZ$, one has
\[
\{(1-a)z+cb: z\in\IZ/m\IZ,c\in(\IZ/m\IZ)^{\ast}\}=(\IZ/m\IZ)^{\ast},
\]
the set of all generators of $\IZ/m\IZ$. Indeed, if $u\in(\IZ/m\IZ)^{\ast}$, then $u=(1-a)0+(ub^{-1})b$, and conversely, viewing $a$, $z$, $c$ and $b$ as integers, one has $\gcd((1-a)z+cb,m)=1$ because $\gcd(cb,m)=1$ and $\rad(m)\mid(1-a)z$. Therefore, by what was said in the previous subsection, the affine permutations $\lambda(a,1):\IZ/m\IZ\rightarrow\IZ/m\IZ$, $x\mapsto ax+1$, where $a$ ranges over the set $\{1+k\cdot\rad'(m)\mid k=0,1,\ldots,\frac{m}{\rad'(m)}-1\}$, form a complete system of representatives for the conjugacy classes of $m$-cycles in $\Hol(\IZ/m\IZ)$.

\subsection{Conjugacy classes of involutions in holomorphs of finite cyclic groups}\label{subsec4P4}

Similarly to how the conjugacy classes of $m$-cyles in $\Hol(\IZ/m\IZ)$ were described at the end of the previous subsection, in this subsection, we will give a description of a complete set of representatives for the conjugacy classes of involutions in $\Hol(\IZ/m\IZ)$. It will be convenient to define an involution in a group $G$ as an element $g\in G$ such that $g^2=1_G$, so that the neutral element $1_G$ is also an involution.

As in Subsection \ref{subsec4P2}, it suffices to describe the conjugacy classes of involutions in $\Hol(\IZ/p^k\IZ)$; the effective Chinese Remainder Theorem can then be used to obtain a complete list of involution conjugacy class representatives of involutions in $\Hol(\IZ/m\IZ)$ for composite $m$ (see the end of this subsection for more details and an example).

As for involutions in $\Hol(\IZ/p^k\IZ)$, we first consider the case $p>2$. Note that in order for $\lambda(a,b)$ to be an involution, $a$ must be an involution in the unit group $(\IZ/p^k\IZ)^{\ast}$. Since the unit group is cyclic of even order, it has exactly two involutions, namely $1$ and $-1$. For $a=1$, where $\lambda(a,b)$ belongs to the regular representation of $\IZ/p^k\IZ$ on itself, we find that $\ord(\lambda(a,b))=\aord(b)$, so the only way for this to be an involution is when $b=0$, thus yielding the trivial involution $\lambda(1,0)$. For $a=-1$, we find that $\lambda(a,b)^2=\lambda(a^2,b(1+a))=\lambda(1,0)$ for all $b\in\IZ/p^k\IZ$, so every choice of $b$ corresponds to an involution. However, since $a-1=-2$ is coprime to $p$ (or, equivalently, the multiplication by $a$ is fixed-point-free), all choices of $b$ lead to the same conjugacy class (see the very end of Subsection \ref{subsec4P2}), so we may choose $\lambda(-1,0)$ as a conjugacy class representative.

We now consider the case $p=2$. For $k=1$, we have $\Hol(\IZ/2^k\IZ)=\Hol(\IZ/2\IZ)\cong\IZ/2\IZ$, so there are two conjugacy classes of involutions, each consisting of only one element, namely $\lambda(1,0)$ and $\lambda(1,1)$ respectively. For $k=2$, there are two involutions in $(\IZ/2^k\IZ)^{\ast}=(\IZ/4\IZ)^{\ast}\cong\IZ/2\IZ$, namely $a=\pm1$. For $a=1$, we find two (singleton) conjugacy classes of involutions, consisting of $\lambda(1,0)$ and $\lambda(1,2)$ respectively. And for $a=-1$, there are also two conjugacy classes of involutions, spanned by $\lambda(-1,0)$ and $\lambda(-1,1)$ respectively.

We may thus now assume that $p=2$ and $k\geq3$. Then since $(\IZ/2^k\IZ)^{\ast}\cong\IZ/2^{k-2}\IZ\times\IZ/2\IZ$, there are four involutions in the unit group, and they are represented by $1$, $-1$, and $2^{k-1}\pm1$. We go through these four values for $a$ and describe the associated conjugacy classes.
\begin{itemize}
\item For $a=1$, there are two conjugacy classes, with representatives $\lambda(1,0)$ and $\lambda(1,2^{k-1})$.
\item For $a=-1$, since $a-1\equiv2\Mod{2^k}$ and $b$ may be chosen arbitrarily to get an involution, we find that there are also two conjugacy classes, with representatives $\lambda(-1,0)$ and $\lambda(-1,1)$.
\item For $a=2^{k-1}-1$, we have $\lambda(a,b)^2=\lambda(1,b(1+a))=\lambda(1,2^{k-1}b)$, so we see that $\lambda(a,b)$ is an involution if and only if $b$ is even. Since $a-1\equiv2\Mod{2^k}$, this implies that $b$ lies in the (additive) subgroup of $\IZ/2^k\IZ$ generated by $a-1$, so there is exactly one conjugacy class, represented by $\lambda(2^{k-1}-1,0)$.
\item For $a=2^{k-1}+1$, we have $\lambda(a,b)^2=\lambda(1,b(1+a))=\lambda(1,2(2^{k-2}+1)b)$. Hence $\lambda(a,b)$ is an involution if and only if $b\equiv2^{k-1}\Mod{2^k}$. Moreover, $a-1=2^{k-1}$, so again, $b$ necessarily lies in the subgroup generated by $a-1$, whence there is exactly one conjugacy class, represented by $\lambda(2^{k-1}+1,0)$.
\end{itemize}

In summary, we obtain the following lemma:

\begin{lemmma}\label{involConjLem}
Let $p$ be a prime, and let $k$ be a positive integer.
\begin{enumerate}
\item If $p>2$, then there are exactly two conjugacy classes of involutions in the holomorph $\Hol(\IZ/p^k\IZ)$, with representatives $\lambda(1,0)$ and $\lambda(-1,0)$.
\item There are exactly two conjugacy classes of involutions in $\Hol(\IZ/2\IZ)$, with representatives $\lambda(1,0)$ and $\lambda(1,1)$.
\item There are exactly four conjugacy classes of involutions in $\Hol(\IZ/4\IZ)$, with representatives $\lambda(1,0)$, $\lambda(1,2)$, $\lambda(-1,0)$ and $\lambda(-1,1)$.
\item For $k\geq3$, there are exactly six conjugacy classes of involutions in $\Hol(\IZ/2^k\IZ)$, with representatives $\lambda(1,0)$, $\lambda(1,2^{k-1})$, $\lambda(-1,0)$, $\lambda(-1,1)$, $\lambda(2^{k-1}-1,0)$ and $\lambda(2^{k-1}+1,0)$.
\end{enumerate}
In particular, for an arbitrary modulus $m\in\IN^+$, the number of conjugacy classes of involutions in $\Hol(\IZ/m\IZ)$ is exactly
\[
\min(6,2\nu_2(m))\cdot 2^{|\{p>2\text{ prime}: p\mid m\}|}.
\]
\end{lemmma}

The \enquote{In particular} in Lemma \ref{involConjLem} holds because any isomorphism
\[
\IZ/m\IZ\rightarrow\prod_{p\in\IP}{\IZ/p^{\nu_p(m)}\IZ}
\]
induces an isomorphism of (permutation) groups
\[
\Hol(\IZ/m\IZ)\rightarrow\prod_{p\in\IP}{\Hol(\IZ/p^{\nu_p(m)}\IZ)},
\]
and conjugacy classes in a direct product of groups are Cartesian products of conjugacy classes in the factor groups. In particular, an involution conjugacy class in $\Hol(\IZ/m\IZ)$ is essentially a tuple of involution conjugacy classes in the groups $\Hol(\IZ/p^{\nu_p(m)}\IZ)$ where $p$ ranges over the set of primes $\IP$, and thus the number of involution conjugacy classes in $\Hol(\IZ/m\IZ)$ is the product of the numbers of involution conjugacy classes in the groups $\Hol(\IZ/p^{\nu_p(m)}\IZ)$.

In the same vein, if one wants to obtain an explicit list of representatives for the conjugacy classes of involutions in $\Hol(\IZ/m\IZ)$, one can do so by combining Lemma \ref{involConjLem} with the effective Chinese Remainder Theorem. Recall that if $m=p_1^{k_1}\cdots p_r^{k_r}$ is the prime power factorization of $m$, then the following function $\iota:\prod_{i=1}^r{\IZ/p_i^{k_i}\IZ}\rightarrow\IZ/m\IZ$ is a group isomorphism:
\[
\iota((a_i+p_i^{k_i}\IZ)_{i=1,\ldots,r})=\sum_{i=1}^r{a_iM_i}
\]
with
\[
M_i:=\prod_{j\not=i}{p_j^{k_j}}\cdot(\prod_{j\not=i}{p_j^{k_j}})^{-1},
\]
the inverse being taken modulo $p_i^{k_i}$. Now, conjugacy classes of involutions in the direct product $\prod_{i=1}^r{\Hol(\IZ/p_i^{k_i}\IZ)}$ correspond to tuples of involution conjugacy classes in the factor groups, and by mapping (representatives of) those class tuples under the group isomorphism
\begin{equation}\label{holIotaEq}
\prod_{i=1}^r{\Hol(\IZ/p_i^{k_i}\IZ)}\rightarrow\Hol(\IZ/m\IZ),(\lambda(a_i,b_i))_{i=1,\ldots,r}\mapsto\lambda(\iota(a_1,\ldots,a_r),\iota(b_1,\ldots,b_r)),
\end{equation}
one obtains a complete list of (representatives of) the conjugacy classes of involutions in $\Hol(\IZ/m\IZ)$.

\begin{exxample}\label{involConjEx1}
We determine a complete set of representatives for the conjugacy classes of involutions in $\Hol(\IZ/12\IZ)$. Set $p_1:=2$ and $p_2:=3$. Then $12=p_1^2\cdot p_2^1$, and with the above notation, we have $M_1=3\cdot 3^{-1}=3\cdot 3=9$ and $M_2=4\cdot 4^{-1}=4\cdot 1^{-1}=4\cdot 1=4$. We thus have the group isomorphism
\[
\iota:\IZ/4\IZ\times\IZ/3\IZ\rightarrow\IZ/12\IZ,(a_1+4\IZ,a_2+3\IZ)\mapsto(9a_1+4a_2+12\IZ).
\]
Note that by Lemma \ref{involConjLem}, $\{\lambda(1,0),\lambda(1,2),\lambda(-1,0),\lambda(-1,1)\}$ resp.~$\{\lambda(1,0),\lambda(-1,0)\}$ is a complete set of involution conjugacy class representatives in $\Hol(\IZ/4\IZ)$ respectively $\Hol(\IZ/3\IZ)$. We now form all $4\cdot2=8$ possible pairs of these representatives and apply the isomorphism from Formula (\ref{holIotaEq}) to get the corresponding conjugacy class representatives in $\Hol(\IZ/12\IZ)$. For example, for $(\lambda(1,2),\lambda(-1,0))$, the corresponding representative is $\lambda(9\cdot1+4\cdot(-1),9\cdot2+4\cdot0)=\lambda(5,18)=\lambda(5,6)$. The results of this are summarized in the following table:

\begin{table}[h]\centering
\begin{tabular}{|c|c|c|}\hline
\diagbox{$\text{ mod }4$}{$\text{mod }3$} & $\lambda(1,0)$ & $\lambda(-1,0)$ \\ \hline
$\lambda(1,0)$ & $\lambda(1,0)$ & $\lambda(5,0)$ \\ \hline
$\lambda(1,2)$ & $\lambda(1,6)$ & $\lambda(5,6)$ \\ \hline
$\lambda(-1,0)$ & $\lambda(7,0)$ & $\lambda(-1,0)$ \\ \hline
$\lambda(-1,1)$ & $\lambda(7,9)$ & $\lambda(-1,9)$ \\ \hline
\end{tabular}
\end{table}
\end{exxample}

\section{Some applications of Theorems \ref{wreathIsoTheo} and \ref{cycloPolyTheo}}\label{sec5}

\subsection{Cycle indices of (generalized) cyclotomic permutations}\label{subsec5P1}

In this subsection, we discuss how to compute, for a given positive integer $d$ and prime power $q$ with $d\mid q-1$, the cycle indices of the permutation groups $\GCP(d,q)$, $\FOCP(d,q)$ and $\CP(d,q)$ introduced in Subsection \ref{subsec1P1}. For the two classes of groups $\GCP(d,q)$ and $\FOCP(d,q)$, this is a simple application of P{\'o}lya's theorem, Theorem \ref{polyaTheo}, due to their wreath product structure. However, for $\CP(d,q)$, we will need to put in more work of our own. We will actually deal with a slightly more general problem that we will describe in what follows.

First, consider the following group-theoretic concept: For every group $G$, the (injective) group homomorphism $G\rightarrow\Sym(G),g\mapsto(x\mapsto xg)$, is called the \emph{(right-)regular representation of $G$}, and its image is denoted by $G_{\reg}$ (cf.~the notation $C_{\reg}$ from Theorem \ref{wreathIsoTheo}).

Now, equipped with this concept, consider the following permutation groups:

\begin{nottation}\label{wGroupNot}
Let $d$ and $m$ be positive integers.
\begin{enumerate}
\item We define $W(d,m)$ as the permutation group
\[
\Hol(\IZ/m\IZ)\wr_{\imp}\Sym(d)\leq\Sym(\IZ/m\IZ\times\{0,1,\ldots,d-1\}).
\]
\item We define $W_=(d,m)$ as the subgroup of $W(d,m)$ that consists of the elements of the form
\[
(\psi,(\lambda(a,b_i))_{i=0,1,\ldots,d-1})
\]
where $\psi\in\Sym(d)$, $a\in(\IZ/m\IZ)^{\ast}$, and $b_0,\ldots,b_{d-1}\in\IZ/m\IZ$.
\item We define $W_1(d,m)$ as the permutation group
\[
(\IZ/m\IZ)_{\reg}\wr_{\imp}\Sym(d)\leq\Sym(\IZ/m\IZ\times\{0,1,\ldots,d-1\}).
\]
\end{enumerate}
\end{nottation}

Observe that $W_1(d,m)\leq W_=(d,m)\leq W(d,m)$ and that by Theorem \ref{wreathIsoTheo}, for $m=\frac{q-1}{d}$, these permutation groups are isomorphic to $\FOCP(d,q)$, $\CP(d,q)$ and $\GCP(d,q)$ respectively. Our goal is to determine the cycle indices of the groups from Notation \ref{wGroupNot}.

We start with the two easy cases, $W(d,m)$ and $W_1(d,m)$. The general approach for computing their cycle indices is discussed in Propositions \ref{GCPProp} and \ref{FOCPProp} below, and each is followed up by an example.

\begin{propposition}\label{GCPProp}
Let $d$ and $m$ be positive integers. Using the notation $\CI^{(m)}$ from Theorem \ref{polyaTheo}, we have
\[
\CI(W(d,m))=\CI(\Sym(d))(\CI^{(1)}(\Hol(\IZ/m\IZ)),\ldots,\CI^{(d)}(\Hol(\IZ/m\IZ))).
\]
In particular, if $q$ is a prime power with $d\mid q-1$, then
\[
\CI(\GCP(d,q))=\CI(\Sym(d))(\CI^{(1)}(\Hol(\IZ/\frac{q-1}{d}\IZ)),\ldots,\CI^{(d)}(\Hol(\IZ/\frac{q-1}{d}\IZ)))
\]
\end{propposition}

\begin{proof}
The main statement is a direct application of Theorem \ref{polyaTheo}, and the \enquote{In particular} holds by Theorem \ref{wreathIsoTheo} (and the fact that isomorphic permutation groups have the same cycle index).
\end{proof}

In Section \ref{sec3}, we discussed how to compute the cycle index of $\Hol(\IZ/m\IZ)$, which involves computing the prime factorization of $m$. See Definition \ref{refProdDef}, Proposition \ref{crtProp}, Theorem \ref{primaryCyclicTheo}, and Example \ref{hol12Ex}, where we computed $\CI(\Hol(\IZ/12\IZ))$, for more details. In the following example, we build on Example \ref{hol12Ex} and compute the cycle index $\CI(\GCP(2,25))$:

\begin{exxample}\label{GCPEx}
We compute
\[
\CI(\GCP(2,25))=\CI(W(2,12)).
\]
By Proposition \ref{ciSymProp},
\[
\CI(\Sym(2))=\frac{1}{2}x_1^2+\frac{1}{2}x_2,
\]
and by Example \ref{hol12Ex},
\begin{align*}
\CI(\Hol(\IZ/12\IZ))=&\frac{1}{48}x_1^{12}+\frac{1}{24}x_1^6x_2^3+\frac{1}{16}x_1^4x_2^4+\frac{1}{8}x_1^2x_2^5+\frac{1}{4}x_2^6+\frac{1}{24}x_3^4+\frac{1}{12}x_3^2x_6+ \\
&\frac{1}{6}x_4^3+\frac{1}{8}x_6^2+\frac{1}{12}x_{12}.
\end{align*}
Set
\[
P_1:=\CI(\Hol(\IZ/12\IZ))=\CI^{(1)}(\Hol(\IZ/12\IZ)),
\]
view $P_1$ as a polynomial in the variables $x_1,\ldots,x_{12}$, and set
\[
P_2:=\CI^{(2)}(\Hol(\IZ/12\IZ))=P_1(x_2,x_4,\ldots,x_{24})=\frac{1}{48}x_2^{12}+\frac{1}{24}x_2^6x_4^3+\cdots+\frac{1}{12}x_{24}.
\]
By Proposition \ref{GCPProp},
\[
\CI(\GCP(2,25))=\CI(\Sym(2))(P_1,P_2)=\frac{1}{2}P_1^2+\frac{1}{2}P_2.
\]
Using GAP \cite{GAP4}, the authors evaluated the expression $\frac{1}{2}P_1^2+\frac{1}{2}P_2$. The result is a homogeneous polynomial $Q_1$ of degree $24$ with $54$ distinct terms (in other words, there are $54$ distinct cycle types of generalized cyclotomic permutations of degree $2$ of $\IF_{25}$), which we will not display here. Recall that $\GCP(2,25)$ is the permutation group on $\IF_{25}^{\ast}$ of all restrictions to $\IF_{25}^{\ast}$ of index $2$ generalized cyclotomic permutations of $\IF_{25}$. If one wants the cycle index of the permutation group on $\IF_{25}$ of actual index $2$ generalized cyclotomic permutations of $\IF_{25}$, one needs to multiply $Q_1$ by $x_1$, to account for the additional fixed point $0$.
\end{exxample}

For first-order cyclotomic permutations, we have the following analogous result:

\begin{propposition}\label{FOCPProp}
Let $d$ and $m$ be positive integers. Using the notation $\CI^{(m)}$ from Theorem \ref{polyaTheo}, we have
\[
\CI(W_1(d,m))=\CI(\Sym(d))(\CI^{(1)}((\IZ/m\IZ)_{\reg}),\ldots,\CI^{(d)}((\IZ/m\IZ)_{\reg})).
\]
In particular, if $q$ is a prime power with $d\mid q-1$, then
\[
\CI(\FOCP(d,q))=\CI(\Sym(d))(\CI^{(1)}((\IZ/\frac{q-1}{d}\IZ)_{\reg}),\ldots,\CI^{(d)}((\IZ/\frac{q-1}{d}\IZ)_{\reg})).
\]
\end{propposition}

\begin{proof}
This is analogous to the proof of Proposition \ref{GCPProp}, using Theorem \ref{wreathIsoTheo}(2) for the required permutation group isomorphism.
\end{proof}

Concerning the cycle index of $(\IZ/m\IZ)_{\reg}$, we note the following simple and well-known result:

\begin{lemmma}\label{regularRepLem}
Let $G$ be a finite group. Then
\[
\CI(G_{\reg})=\frac{1}{|G|}\sum_{o\mid|G|}{|\{g\in G:\ord(g)=o\}|x_o^{|G|/o}}.
\]
In particular, if $G$ is cyclic, then
\[
\CI(G_{\reg})=\frac{1}{|G|}\sum_{o\mid|G|}{\phi(o)x_o^{|G|/o}}.
\]
\end{lemmma}

\begin{exxample}\label{FOCPEx}
We compute
\[
\CI(\FOCP(2,25))=\CI(W_1(2,12)).
\]
By Proposition \ref{ciSymProp},
\[
\CI(\Sym(2))=\frac{1}{2}x_1^2+\frac{1}{2}x_2,
\]
and by Lemma \ref{regularRepLem},
\[
\CI((\IZ/12\IZ)_{\reg})=\frac{1}{12}x_1^{12}+\frac{1}{12}x_2^6+\frac{1}{6}x_3^4+\frac{1}{6}x_4^3+\frac{1}{6}x_6^2+\frac{1}{3}x_{12}.
\]
Set
\[
P_1:=\CI((\IZ/12\IZ)_{\reg})=\CI^{(1)}((\IZ/12\IZ)_{\reg}),
\]
view $P_1$ as a polynomial in the variables $x_1,\ldots,x_{12}$, and set
\[
P_2:=\CI^{(2)}((\IZ/12\IZ)_{\reg})=P_1(x_2,x_4,\ldots,x_{24})=\frac{1}{12}x_2^{12}+\frac{1}{12}x_4^6+\cdots+\frac{1}{3}x_{24}.
\]
By Proposition \ref{FOCPProp},
\[
\CI(\FOCP(2,25))=\CI(\Sym(2))(P_1,P_2)=\frac{1}{2}P_1^2+\frac{1}{2}P_2.
\]
As in Example \ref{GCPProp}, we used GAP \cite{GAP4} to evaluate this last expression. It is a homogeneous polynomial $Q_2$ of degree $24$ with $23$ distinct terms, which we will not display here. In order to get the cycle index of the permutation group of index $2$ first-order cyclotomic permutations on $\IF_{25}$ (as opposed to their restrictions to $\IF_{25}^{\ast}$), multiply $Q_2$ by $x_1$.
\end{exxample}

There is also a systematic way to compute $\CI(W_=(d,m))$, but it is more complicated to describe. We need some more preparations first.

\begin{nottation}\label{omicronNot}
Let $m$ be a positive integer, and let $a\in\IZ$ with $\gcd(m,a)=1$.
\begin{enumerate}
\item Assume that $m=p^k$ is a power of a prime $p$ (possibly with $k=0$). We define $\omicron_m(a)=\omicron_{p^k}(a)$ as follows:
\begin{enumerate}
\item If $p>2$, then $\omicron_{p^k}(a):=\ord_{p^k}(a)$, the multiplicative order of $a$ modulo $p^k$.
\item If $p=2$, then
\begin{enumerate}
\item if $k\leq1$, then $\omicron_{p^k}(a):=(0,1)$.
\item if $k>1$, then $a$, viewed as an element of the ring $\IZ/2^k\IZ$, has a unique representation as $(-1)^{\epsilon}a'$ where $\epsilon\in\{0,1\}$ and $a'$ is a power of $5$. We set $\omicron_{2^k}(a):=(\epsilon,\ord_{2^k}(a'))$.
\end{enumerate}
\end{enumerate}
\item Assume that $m=p^k$ is a power of a prime $p$. We set
\[
\Omega(m)=\Omega(p^k):=\{\omicron_{p^k}(a)\mid a\in\IZ,\gcd(m,a)=1\}.
\]
\item We set $\vec{\omicron}_m(a):=(\omicron_{p^{\nu_p(m)}}(a))_{p\in\IP}$, where $\IP$ denotes the set of primes.
\end{enumerate}
\end{nottation}

Note that
\[
\Omega(p^k)=\begin{cases}\{d\in\IN^+:d\mid\phi(p^k)\}, & \text{if }p>2, \\ \{(0,1)\}, & \text{if }p=2\text{ and }k\leq1, \\ \{0,1\}\times\{d\in\IN^+:d\mid 2^{k-2}\}, & \text{if }p=2\text{ and }k\geq2.\end{cases}
\]

\begin{exxample}\label{omicronEx}
Let $m=12$ and $a=-1$. Then
\[
\vec{\omicron}_m(a)=\vec{\omicron}_{12}(-1)=(\omicron_{2^2}(-1),\omicron_{3^1}(-1),\omicron_{5^0}(-1),\omicron_{7^0}(-1),\ldots)=((1,1),2,1,1,\ldots).
\]
\end{exxample}

In order to explain why Notation \ref{omicronNot} is of interest to us, we introduce the following concept:

\begin{deffinition}\label{cycleCounterDef}
Let $X\subseteq\Sym(\Omega)$ be a set of permutations on the finite set $\Omega$. We call the polynomial
\[
\CC(X):=\sum_{g\in X}{\CT(g)}\in\IQ[x_1,x_2,\ldots,x_{|\Omega|}]
\]
the \emph{cycle counter of $X$}.
\end{deffinition}

Note that as long as $X$ is nonempty, we have $\CC(X)=|X|\cdot\CI(X)$, where the \emph{cycle index} $\CI(X)$ is as defined at the beginning of the proof of Theorem \ref{polyaTheo}.

Our motivation behind Notation \ref{omicronNot} is this: By Tables \ref{table1} and \ref{table2} in Section \ref{sec3}, for a fixed integer $a$ with $\gcd(a,m)=1$, the cycle counter of the set of all affine permutations of $\IZ/m\IZ$ of the form $x\mapsto ax+b$ for some $b\in\IZ/m\IZ$ only depends on $\vec{\omicron}_m(a)$. More precisely, set
\[
\Hol_a(\IZ/m\IZ):=\{(x\mapsto ax+b)\in\Hol(\IZ/m\IZ): b\in \IZ/m\IZ\}.
\]
We will give a formula for $\CC(\Hol_a(\IZ/m\IZ))$ in terms of $\vec{\omicron}_m(a)$, see Proposition \ref{holaProp} below. This proposition is easier to formulate if we first introduce some notation for the case that $m$ is a prime power:

\begin{nottation}\label{gammaNot}
Let $p$ be a prime, let $k$ be a non-negative integer, and let $o\in\Omega(p^k)$. We introduce the notation $\Gamma_{p^k}(o)$, denoting a polynomial in $\IQ[x_1,\ldots,x_{p^k}]$, in the following case distinction:
\begin{enumerate}
\item If $p>2$, in which case $o$ is a divisor of $\phi(p^k)$,
\[
\Gamma_{p^k}(o):=\begin{cases}p^kx_1x_{o_{p'}}^{\frac{p^{k-\nu_p(o)}-1}{o_{p'}}}\prod_{u=1}^{\nu_p(o)}{x_{p^uo_{p'}}^{\frac{p^{k-1-\nu_p(o)}(p-1)}{o_{p'}}}}, & \text{if }o\nmid\phi(p^k)_p, \\ ox_1^{\frac{p^k}{o}}\prod_{u=1}^{\nu_p(o)}{x_{p^u}^{\frac{p^{k-1}}{o}(p-1)}}+\sum_{v=0}^{k-\nu_p(o)-1}{\phi(p^{k-v})x_{p^{k-v}}^{p^v}}, & \text{if }o\mid\phi(p^k)_p.\end{cases}
\]
\item If $p=2$ and $k\leq1$, then $o=(0,1)$, and we set
\[
\Gamma_{p^k}(o):=\begin{cases}x_1, & \text{if }k=0, \\ x_1^2+x_2, & \text{if }k=1.\end{cases}
\]
\item If $p=2$ and $k\geq2$, then $o$ is a pair $(\epsilon,o')$ where $\epsilon\in\{0,1\}$ and $o'\mid 2^{k-2}$. We define $\Gamma_{p^k}(o)$ as follows:
\begin{enumerate}
\item If $k=2$:
\[
\Gamma_4((\epsilon,1)):=\begin{cases}x_1^4+x_2^2+2x_4, & \text{if }\epsilon=0, \\ 2x_1^2x_2+2x_2^2, & \text{if }\epsilon=1.\end{cases}
\]
\item If $k\geq3$, set $\nu:=k-2-\nu_2(o')$ and $\nu':=\min(k-3,\nu)$. Then
\[
\Gamma_{2^k}(\epsilon,o'):=\begin{cases}2^{k-1}x_{2^{k-1-\nu}}^{2^{1+\nu}}+2^{k-1}x_1^2x_2^{2^{2+\nu'}-1}\prod_{u=2}^{k-2-\nu'}{x_{2^u}^{2^{1+\nu'}}}, & \text{if }\epsilon=1, \\ 2^{k-2-\nu}x_1^{2^{2+\nu}}\prod_{u=1}^{k-2-\nu}{x_{2^u}^{2^{1+\nu}}}+\sum_{v=0}^{1+\nu}{\phi(2^{k-v})x_{2^{k-v}}^{2^v}}, & \text{if }\epsilon=0.\end{cases}
\]
\end{enumerate}
\end{enumerate}
\end{nottation}

\begin{propposition}\label{gammaProp}
Let $p$ be a prime, let $k$ be a non-negative integer, and let $o\in\Omega(p^k)$. Then $\Gamma_{p^k}(o)$ is the cycle counter of the set $\Hol_a(\IZ/p^k\IZ)$ that consists of all affine permutations $x\mapsto ax+b$ of $\IZ/m\IZ$ where $a\in(\IZ/m\IZ)^{\ast}$ is fixed and satisfies $\omicron_{p^k}(a)=o$ and $b\in\IZ/m\IZ$.
\end{propposition}

\begin{proof}
For statement (1): Noting that $o=\ord_{p^k}(a)$, the case \enquote{$o\nmid\phi(p^k)_p$}, where $o$ is not a power of $p$, corresponds to $a\not\equiv1\Mod{p}$. In that case, the cycle type does not depend on $b$, so we obtain $p^k$ permutations $x\mapsto ax+b$ that are all of the same cycle type, which can be read off from the first row in Table \ref{table1}. On the other hand, if $o\mid\phi(p^k)_p$, then $o$ is a power of $p$, say $o=p^s$ with $s\in\{0,1,\ldots,k-1\}$. Since $\ord(a)=o$, it follows that $a$ is an element of the Sylow $p$-subgroup of $(\IZ/p^k\IZ)^{\ast}$, whence $a=(1+p)^e$ for some $e\in\{0,1,\ldots,p^{k-1}-1\}$ by Lemma \ref{primaryCyclicLem}(1). Also by Lemma \ref{primaryCyclicLem}(1), we find that $\nu_p(o)=\nu_p(\ord(a))=k-1-\nu_p^{(k-1)}(e)$ and $\nu_p^{(k)}(a-1)=1+\nu_p^{(k-1)}(e)$, whence $\nu_p^{(k)}(a-1)=k-\nu_p(\ord(a))=k-\nu_p(o)$. We can substitute this into the cycle types specified in rows $2$ and $3$ of Table \ref{table1} to get formulas for those cycle types in terms of $o$. In order to conclude the proof of statement (1), we need to count the number of values of $b$ that correspond to each of the two cases ($\nu_p^{(k)}(b)\geq\nu_p^{(k)}(a-1)$ and its complementary case), but this is routine using Lemma \ref{simpleLem}.

For statement (2): This can be verified with some simple computations, which we omit.

For statement (3): The case \enquote{$k=2$} can be easily verified separately, and we omit the argument. For \enquote{$k\geq3$}, we carry the proof out. First, assume that $\epsilon=1$. Then there are only two distinct cycle types, and which of them occurs depends on whether or not $b$ is even. In both cases, $2^{k-1}$ values of $b$ \enquote{match}, which explains the coefficients. For the monomials, observe that by Lemma \ref{primaryCyclicLem}(2), we have $\nu_2(o')=\nu_2(5^e)=k-2-\nu_2^{(k-2)}(e)$, so that
\[
\nu_2^{(k-2)}(e)=k-2-\nu_2(o')=\nu
\]
and
\[
\nu_2^{(k-3)}(e)=\min\{k-3,\nu_2^{(k-2)}(e)\}=\min\{k-3,\nu\}=\nu'.
\]
Now assume that $\epsilon=0$. By Lemma \ref{primaryCyclicLem}(2), we have $\nu_2^{(k)}(a-1)=\nu_2^{(k)}(5^e-1)=2+\nu_2^{(k-2)}(e)=2+\nu$. This allows one to count the number of $b$ that lead to each possible cycle type using Lemma \ref{simpleLem}.
\end{proof}

\begin{propposition}\label{holaProp}
Let $m$ be a positive integer, and let
\[
\vec{\omicron}=(o_p)_{p\in\IP}\in\prod_{p\in\IP}{\Omega(p^{\nu_p(m)})}.
\]
Then for each $a\in\IZ$ with $\gcd(m,a)=1$ and $\vec{\omicron}_m(a)=\vec{\omicron}$, we have
\[
\CC(\Hol_a(\IZ/m\IZ))=\divideontimes_{p\in\IP}{\Gamma_{p^{\nu_p(m)}}(o_p)}.
\]
\end{propposition}

\begin{proof}
The permutation group isomorphism $\Hol(\IZ/m\IZ)\rightarrow\prod_{p\in\IP}{\Hol(\IZ/p^{\nu_p(m)}\IZ)}$ from Proposition \ref{crtProp} maps the set $\Hol_a(\IZ/m\IZ)$ to the Cartesian product
\[
\prod_{p\in\IP}{\Hol_a(\IZ/p^{\nu_p(m)}\IZ)}.
\]
By Proposition \ref{gammaProp}, we know that for each prime $p$,
\[
\CC(\Hol_a(\IZ/p^{\nu_p(m)}\IZ))=\Gamma_{p^{\nu_p(m)}}(o_p),
\]
and the proof that
\[
\CC(\prod_{p\in\IP}{\Hol_a(\IZ/p^{\nu_p(m)}\IZ)})=\divideontimes_{p\in\IP}{\CC(\Hol_a(\IZ/p^{\nu_p(m)}\IZ))}
\]
is analogous to \cite[proof of Theorem 2.4]{WX93a}.
\end{proof}

Another crucial observation is that $\vec{\omicron}_m(a^{\ell})$ can be expressed in terms of $\vec{\omicron}_m(a)$ and $\ell$:

\begin{nottation}\label{powNot}
We define the notation $\pow$ as follows:
\begin{enumerate}
\item Let $p$ be a prime, let $k\in\IN$, let $\ell\in\IZ$, and let $o\in\Omega(p^k)$. Then
\begin{enumerate}
\item if $p>2$, we set $\pow_{p^k}(o,\ell):=\frac{o}{\gcd(o,\ell)}$.
\item if $p=2$, write $o=(\epsilon,o')$ where $\epsilon\in\{0,1\}$ and $o'\mid\max\{1,2^{k-2}\}$. We set
\[
\pow_{2^k}(o,\ell):=\begin{cases}(0,\frac{o}{\gcd(o,\ell)}), & \text{if }\epsilon=0,\text{ or }\epsilon=1\text{ and }2\mid\ell, \\ (1,\frac{o}{\gcd(o,\ell)}), & \text{if }\epsilon=1\text{ and }2\nmid\ell.\end{cases}
\]
\end{enumerate}
\item Let $m\in\IN^+$, let $\ell\in\IZ$, and let $\vec{\omicron}=(o_p)_{p\in\IP}\in\prod_{p\in\IP}{\Omega(p^{\nu_p(m)})}$. We set
\[
\pow_m(\vec{\omicron},\ell):=(\pow_{p^{\nu_p(m)}}(o_p,\ell))_{p\in\IP}.
\]
\end{enumerate}
\end{nottation}

The following is easy to check:

\begin{propposition}\label{powProp}
The following hold:
\begin{enumerate}
\item Let $p$ be a prime, let $k\in\IN$, let $\ell\in\IN^+$, and let $o\in\Omega(p^k)$. Then for each $a\in\IZ$ with $p\nmid a$ and $\omicron_{p^k}(a)=o$, we have $\omicron_{p^k}(a^{\ell})=\pow_{p^k}(o,\ell)$.
\item Let $m\in\IN^+$, let $\ell\in\IZ$, and let $\vec{\omicron}=(o_p)_{p\in\IP}\in\prod_{p\in\IP}{\Omega(p^{\nu_p(m)})}$. Then for each $a\in\IZ$ with $\gcd(m,a)=1$ and $\vec{\omicron}_m(a)=\vec{\omicron}$, we have $\vec{\omicron}_m(a^{\ell})=\pow_m(\vec{\omicron},\ell)$.
\end{enumerate}
\end{propposition}

We now discuss how all of this relates to the problem of computing the cycle index of $W(d,m)$, for positive integers $d$ and $m$. Note that
\[
|W_=(d,m)|=d!\cdot\phi(m)\cdot m^d,
\]
whence it suffices to compute the cycle counter $\CC(W(d,m))$. Assume that
\[
\vec{\omicron}_m(a)=\vec{\omicron}=(o_p)_{p\in\IP}\in\prod_{p\in\IP}{\Omega(p^{\nu_p(m)})}
\]
is fixed. We denote the corresponding subset of $W_=(d,m)$ by $W_{\vec{\omicron}}(d,m)$ and note that
\[
W_=(d,m)=\bigsqcup_{\vec{\omicron}}{W_{\vec{\omicron}}(d,m)}\text{ (disjoint union) and }\CC(W_=(d,m))=\sum_{\vec{\omicron}}{\CC(W_{\vec{\omicron}}(d,m))}.
\]
In order to determine $\CC(W_{\vec{\omicron}}(d,m))$, let us fix $\psi$ up to $\Sym(d)$-conjugacy. That is, we fix a $d$-tuple $\lambda=(\lambda_1,\ldots,\lambda_d)\in\IN^d$ such that $\sum_{i=1}^d{i\lambda_i}=d$, and assume that $\psi$ has $i$ cycles of length $\lambda_i$ for each $i$. Let $\ell\in\{1,\ldots,d\}$. It is not hard to check that the possible forward cycle product classes of $g$ with regard to the length $\ell$ cycles of $\psi$ are of the form $\lambda(a^{\ell},b)^{\Hol(\IZ/m\IZ)}$, where the elements $b\in\IZ/m\IZ$ may be chosen independently of each other for different cycles (of arbitrary lengths, not just $\ell$). Since $\vec{\omicron}_m(a)=\vec{\omicron}$, Proposition \ref{powProp}(2) lets us conclude that for each $\ell\in\{1,\ldots,d\}$, we have $\vec{\omicron}_m(a^{\ell})=\pow_m(\vec{\omicron},\ell)$, and by Proposition \ref{holaProp},
\[
\CC(\{\lambda(a^{\ell},b):b\in\IZ/m\IZ\})=\divideontimes_{p\in\IP}{\Gamma_{p^{\nu_p(m)}}(\pow_m(o_p,\ell))}.
\]
View this as a polynomial in the variables $x_1,x_2,\ldots,x_m$, and set
\[
\Delta_m(\vec{o},\ell):=(\divideontimes_{p\in\IP}{\Gamma_{p^{\nu_p(m)}}(\pow_m(o_p,\ell))})(x_{\ell},x_{2\ell},\ldots,x_{m\ell}).
\]
Using Lemma \ref{polyaLem} and the definition of the cycle counter, we conclude that
\[
\CC(W_{\vec{\omicron}}(d,m))=|\{a\in(\IZ/m\IZ)^{\ast}:\vec{\omicron}_m(a)=\vec{\omicron}\}|\cdot\CC(\Sym(d))(\Delta_m(\vec{o},1),\ldots,\Delta_m(\vec{o},d)).
\]
For the sake of notational simplicity, we introduce the following notation:

\begin{nottation}\label{nNot}
Let $m$ be a positive integer.
\begin{enumerate}
\item Assume that $m=p^k$ is a power of a prime $p$. Let $o\in\Omega(m)=\Omega(p^k)$. We set
\[
n_m(o)=n_{p^k}(o):=\begin{cases}\phi(o), & \text{if }p>2, \\ \phi(o'), & \text{if }p=2,o=(\epsilon,o').\end{cases}
\]
\item Let $\vec{\omicron}=(o_p)_{p\in\IP}\in\prod_{p\in\IP}{\Omega(p^{\nu_p(m)})}$. We set
\[
n_m(\vec{\omicron}):=\prod_{p\in\IP}{n_{p^{\nu_p(m)}}(o_p)}.
\]
\end{enumerate}
\end{nottation}

The following is not hard to check:

\begin{propposition}\label{nProp}
Let $m$ be a positive integer.
\begin{enumerate}
\item If $m=p^k$ is a power of a prime $p$ and $o\in\Omega(m)$, then
\[
n_m(o)=|\{a\in(\IZ/m\IZ)^{\ast}: \omicron_m(a)=o\}|.
\]
\item If $\vec{\omicron}\in\prod_{p\in\IP}{\Omega(p^{\nu_p(m)})}$, then
\[
n_m(\vec{\omicron})=|\{a\in(\IZ/m\IZ)^{\ast}: \vec{\omicron}_m(a)=\vec{\omicron}\}|.
\]
\end{enumerate}
\end{propposition}

Putting everything together, we obtain the following:

\begin{propposition}\label{CPProp}
Let $d$ and $m$ be positive integers. Then
\[
\CI(W_=(d,m))=\frac{1}{d!\phi(m)m^d}\sum_{\vec{\omicron}}{n_m(\vec{o})\CC(\Sym(d))(\Delta_m(\vec{o},1),\ldots,\Delta_m(\vec{o},d))},
\]
where $\omicron$ ranges over the set $\prod_{p\in\IP}{\Omega(p^{\nu_p(m)})}$. Moreover, if $q$ is a prime power with $d\mid q-1$, then
\[
\CI(\CP(d,q))=\CI(W(d,\frac{q-1}{d})).
\]
\end{propposition}

\begin{exxample}\label{CPEx}
We compute
\[
\CI(\CP(2,25))=\CI(W_=(2,12)).
\]
Note that by the formula after Notation \ref{omicronNot},
\[
\Omega(4)=\{(0,1),(1,1)\}\text{ and }\Omega(3)=\{1,2\}.
\]
We need to go through the $2\cdot2=4$ possible combinations $\vec{\omicron}=(o_2,o_3)\in\Omega(4)\times\Omega(3)$ and compute the respective summand in the formula from Proposition \ref{CPProp}. We discuss one case in detail, namely $o_2=(0,1)$ and $o_3=1$. Then
\[
n_{12}(\vec{\omicron})=\phi(1)\cdot\phi(1)=1,
\]
and
\begin{align*}
\Delta_{12}(\vec{\omicron},1)&=\Gamma_4(\pow_4(o_2,1))\divideontimes\Gamma_3(\pow_3(o_3,1)) \\
&=\Gamma_4((0,1))\divideontimes\Gamma_3(1) \\
&=(x_1^4+x_2^4+x_4)\divideontimes(x_1^3+2x_3) \\
&=x_1^{12}+2x_3^4+x_2^6+2x_6^2+x_4^3+2x_{12}
\end{align*}
as well as
\begin{align*}
\Delta_{12}(\vec{\omicron},2)&=(\Gamma_4(\pow_4(o_2,2))\divideontimes\Gamma_3(\pow_3(o_3,2)))(x_i\mapsto x_{2i}) \\
&=(\Gamma_4((0,1))\divideontimes\Gamma_3(1))(x_i\mapsto x_{2i}) \\
&=x_2^{12}+2x_6^4+x_4^6+2x_{12}^2+x_8^3+2x_{24};
\end{align*}
see Definition \ref{refProdDef} and Example \ref{permProductEx} on how to compute $\divideontimes$-products.

Table \ref{deltaTable} summarizes the values of $n_{12}(\vec{\omicron})$, $\Delta_{12}(\vec{\omicron},1)$ and $\Delta_{12}(\vec{\omicron},2)$ for each of the four possible values of $\vec{\omicron}$.

\begin{table}[h]\centering
\begin{tabular}{|c|c|c|c|c|}\hline
$o_2$ & $o_3$ & $n_{12}(\vec{\omicron})$ & $\Delta_{12}(\vec{\omicron},1)$ & $\Delta_{12}(\vec{\omicron},2)$ \\ \hline
$(0,1)$ & $1$ & $1$ & $x_1^{12}+2x_3^4+x_2^6+2x_6^2+x_4^3+2x_{12}$ & $x_2^{12}+2x_6^4+x_4^6+2x_{12}^2+x_8^3+2x_{24}$ \\ \hline
$(0,1)$ & $2$ & $1$ & $3x_1^4x_2^4+3x_2^6+3x_4^3$ & $x_2^{12}+2x_6^4+x_4^6+2x_{12}^2+x_8^3+2x_{24}$ \\ \hline
$(1,1)$ & $1$ & $1$ & $2x_1^6x_2^3+4x_3^2x_6+2x_2^6+4x_6^2$ & $x_2^{12}+2x_6^4+x_4^6+2x_{12}^2+x_8^3+2x_{24}$ \\ \hline
$(1,1)$ & $2$ & $1$ & $6x_1^2x_2^5+6x_2^6$ & $x_2^{12}+2x_6^4+x_4^6+2x_{12}^2+x_8^3+2x_{24}$ \\ \hline
\end{tabular}
\caption{Relevant data for each of the four values of $\vec{\omicron}=(o_2,o_3)$.}
\label{deltaTable}
\end{table}

According to Proposition \ref{CPProp},
\[
\CI(\CP(2,25))=\CI(W_=(2,12))=\frac{1}{2\phi(12)12^2}\sum_{\vec{\omicron}=(o_2,o_3)}{(\Delta_{12}(\vec{\omicron},1)^2+\Delta_{12}(\vec{\omicron},2))}.
\]
Evaluating this expression with GAP \cite{GAP4}, we obtain a homogeneous polynomial $Q_3\in\IQ[x_1,\ldots,x_{24}]$ with $32$ terms, which we do not display.
\end{exxample}

\subsection{Cycle types of individual generalized cyclotomic permutations}\label{subsec5PX}

Let $d$ be a positive integer, let $q$ be a prime power with $d\mid q-1$, let $\omega$ be a primitive root of $\IF_q$, let $C$ be the index $d$ subgroup of $\IF_q^{\ast}$, and set $C_i:=\omega^iC$ for $i=0,1,\ldots,d-1$. Consider Algorithm \ref{algo2}, and observe that it is correct:

\begin{algorithm}
\SetKwInOut{Input}{input}\SetKwInOut{Output}{output}

\Input{A polynomial $P\in\IF_q[T]$ with $\deg{P}\leq q-1$.}
\Output{The information whether $P$ is the polynomial form of a generalized index $d$ cyclotomic \emph{permutation} $f$ of $\IF_q$ and, if so, an $\omega$-cyclotomic form of $f$ and the unique permutation $\psi\in\Sym(d)$ with $f(C_i)=C_{\psi(i)}$ for $i=0,1,\ldots,d-1$.}
\BlankLine
\nl Apply Algorithm \ref{algo1} to $P$ to see if the function $f:\IF_q\rightarrow\IF_q$ with polynomial form $P$ is an index $d$ generalized cyclotomic mapping of $\IF_q$ and, if so, find an $\omega$-cyclotomic form $f=f_{\omega}(\vec{a},\vec{r})$.

\nl If $f$ is not an index $d$ generalized cyclotomic mapping of $\IF_q$, output \enquote{$P$ is not the polynomial form of any index $d$ generalized cyclotomic permutation of $\IF_q$.}, and halt.

\nl Write $\vec{a}=(a_0,\ldots,a_{d-1})$ and $\vec{r}=(r_0,\ldots,r_{d-1})$. If there is an $i\in\{0,\ldots,d-1\}$ such that $a_i=0$ or $\gcd(r_i,\frac{q-1}{d})>1$, output \enquote{$P$ is not the polynomial form of any index $d$ generalized cyclotomic permutation of $\IF_q$.}, and halt.

\nl Set $L:=\emptyset$ (the empty list).

\nl \For{$i=0,1,\ldots,d-1$}{

\nl Set $y_i:=a_i\omega^{r_ii}$.

\nl \For{$j=0,1,\ldots,d-1$}{

\nl \If{$(\omega^{-j}y_i)^{\frac{q-1}{d}}=1$}{

\nl Add $j$ to $L$ (as a new last entry), and exit the (inner) loop.
}
}
}

\nl Let $\psi:\{0,1,\ldots,d-1\}\rightarrow\{0,1,\ldots,d-1\}$ be the function such that $\psi(i)$ is the $i$-th entry of $L$.

\nl If $\psi$ is not a permutation of $\{0,1,\ldots,d-1\}$, output \enquote{$P$ is not the polynomial form of any index $d$ generalized cyclotomic permutation of $\IF_q$.}, and halt.

\nl Output \enquote{$P$ is the polynomial form of the index $d$ cyclotomic permutation of $\IF_q$ with $\omega$-cyclotomic form $f_{\omega}(\vec{a},\vec{r})$, permuting the cosets of $C$ according to $\psi$.}, and halt.

\caption{Extended conversion from polynomial form to cyclotomic form for permutations.}\label{algo2}
\end{algorithm}

\begin{itemize}
\item In step 3, if $a_i=0$, then $f$ maps $C_i\subseteq\IF_q^{\ast}$ to $\{0\}$. Since $f(0)=0$, this implies that $f$ is not injective, hence not a permutation of $\IF_q$. Moreover, $\gcd(r_i,\frac{q-1}{d})=1$ for each $i$ is necessary for $f$ to be a permutation by \cite[Theorem 1.2]{Wan13a}.
\item Beyond step 3, we know that $f$ maps cosets of $C$ onto full cosets of $C$, and so the only potentially remaining obstacle for $f$ to be a permutation is when two distinct cosets get mapped onto the same image coset. That is, we need to check whether the induced function $\tilde{f}$ on cosets is a permutation.
\item To check this, we map, for each $i=0,1,\ldots,d-1$, the coset representative $\omega^i$ of $C_i=\omega^iC$ under $f$, resulting in $y_i$, and then we check for which $j\in\{0,1,\ldots,d-1\}$ it is the case that $y_i\in C_j=\omega^jC$, or equivalently (using that $C=\{x\in\IF_q^{\ast}: x^{\frac{q-1}{d}}=1\}$), $(\omega^{-j}y_i)^{\frac{q-1}{d}}=1$. This $j$ is unique and equals $\tilde{f}(i)$. Therefore, the function $\psi$ defined in step 10 is equal to $\tilde{f}$.
\end{itemize}

We note that our approach for computing $\psi$ in Algorithm \ref{algo2} is a variant of the bijectivity criterion for index $d$ generalized cyclotomic mappings from \cite[Theorem 1.2]{Wan13a}. The advantage of our approach is that it avoids using discrete logarithms, which makes Algorithm \ref{algo2} efficient.

Once we know that $P$ represents an index $d$ generalized cyclotomic permutation $f$ of $\IF_q$, and we know the cyclotomic form of $f$ and $\psi$, we can apply Theorem \ref{wreathIsoTheo} to compute (also efficiently) the $\omega$-wreath product form of $f$. Form there, we can obtain the cycle type of $f$ using Lemma \ref{polyaLem} and the theory from Section \ref{sec3}, but this is inefficient (at least on non-quantum computers), as it requires us to find the prime factorization of $\frac{q-1}{d}$. We conclude this subsection with an illustrative example:

\begin{exxample}\label{unnumberedEx}
Let $d=2$ and $q=25$. Moreover, let $\omega\in\IF_{25}$ be a root of the Conway polynomial $T^2-T+2\in\IF_5[T]$. Consider the polynomial
\[
P:=\omega^{15}T^5+\omega^{23}T^7+\omega^3T^{17}+\omega^{23}T^{19}\in\IF_{25}[T].
\]
We check whether $P$ is the polynomial form of an index $2$ generalized cyclotomic permutation of $\IF_{25}$ and, if so, determine its cycle type.

We start by applying Algorithm \ref{algo1}. Observe that in step 4, we have $k=2$, $\rho_0=5$ and $\rho_1=7$. In step 5,
\[
\vec{v_0}=\begin{pmatrix}\omega^{15} \\ \omega^3\end{pmatrix}\text{ and }\vec{v_1}=\begin{pmatrix}\omega^{23} \\ \omega^{23}\end{pmatrix},
\]
and so in step 6 (noting that $\zeta=\omega^{12}=-1$ here),
\[
\vec{b_0}=2\cdot\begin{pmatrix}1 & 1 \\ 1 & -1\end{pmatrix}^{-1}\cdot\vec{v_0}=\begin{pmatrix}0 \\ \omega^{21}\end{pmatrix}
\]
and
\[
\vec{b_1}=2\cdot\begin{pmatrix}1 & 1 \\ 1 & -1\end{pmatrix}^{-1}\cdot\vec{v_1}=\begin{pmatrix}\omega^5 \\ 0\end{pmatrix}.
\]
It follows from steps 7--9 that $P$ indeed represents an index $2$ generalized cyclotomic permutation $f_{\omega}(\vec{a},\vec{r})$ of $\IF_{25}$, with
\[
\vec{a}=(a_0,a_1)=(\omega^5,\omega^{21})
\]
and
\[
\vec{r}=(r_0,r_1)=(7,5).
\]
Observe that each $a_i$ is nonzero, and that $\gcd(12,r_0)=\gcd(12,r_1)=1$. Next, we compute (following Algorithm \ref{algo2}) the function $\psi:\{0,1\}\rightarrow\{0,1\}$ that encodes the mapping behavior of $f_{\omega}(\vec{a},\vec{r})$ on cosets $C_i=\omega^iC$ (for $i=0,1$) of the index $2$ subgroup $C$ of $\IF_{25}^{\ast}$:
\begin{itemize}
\item $y_0=a_0\omega^{r_00}=\omega^5$. Hence $y_0^{12}=\omega^{60}\not=1$ (because $24\nmid60$), but $(\omega^{-1}y_0)^{12}=\omega^{48}=1$. Hence $y_0\in C_1$, and so $\psi(0)=1$.
\item $y_1=a_1\omega^{r_11}=\omega^{21}\cdot\omega^5=\omega^{26}=\omega^2$. We have $y_1^{12}=\omega^{24}=1$, whence $y_1\in C_0$ and $\psi(1)=0$.
\end{itemize}
We see that $\psi$ is the nontrivial permutation $(0,1)$ on $\{0,1\}$. In particular, we conclude that $f$ is a permutation on $\IF_q$. Moreover, we can compute the $\omega$-wreath product form of $f$ (an element of $\Hol(C)\wr_{\imp}\Sym(2)$) according to Theorem \ref{wreathIsoTheo} as follows:
\begin{align*}
&\iota_{\omega}^{-1}(f_{\omega}((\omega^5,\omega^{21}),(7,5)))=((0,1),(\lambda(5,\omega^{5\cdot1-0}\cdot\omega^{21}),\lambda(7,\omega^{7\cdot0-1}\cdot\omega^5)))= \\
&((0,1),(\lambda(5,\omega^2),\lambda(7,\omega^4))).
\end{align*}
When rewritten into an element of $\Hol(\IZ/12\IZ)\wr_{\imp}\Sym(2)$ (additive notation, with $1+12\IZ$ corresponding to the generator $\omega^2$ of $C$), this becomes
\[
((0,1),(\lambda(5,\frac{2}{2}),\lambda(7,\frac{4}{2})))=((0,1),(\lambda(5,1),\lambda(7,2))).
\]
We compute the cycle type of this according to Lemma \ref{polyaLem}. For this, we need to determine the cycle type (on $\IZ/12\IZ$) of the forward cycle product
\[
\lambda(5,1)\cdot\lambda(7,2)=\lambda(5\cdot7,7\cdot1+2)=\lambda(35,9)=\lambda(-1,9).
\]
For greater clarity in what follows, let us write $\lambda_m(a,b)$ for the affine function $x\mapsto ax+b$ on $\IZ/m\IZ$. By Proposition \ref{crtProp}, we have a permutation group isomorphism
\[
\Hol(\IZ/12\IZ)\rightarrow\Hol(\IZ/4\IZ)\times\Hol(\IZ/3\IZ),
\]
under which $\lambda_{12}(-1,9)$ corresponds to the pair
\[
(\lambda_4(-1,9),\lambda_3(-1,9))=(\lambda_4(3,1),\lambda_3(2,0)).
\]
According to Tables \ref{table1} and \ref{table2} in Section \ref{sec3},
\[
\CT(\lambda_4(3,1))=x_2^2\text{ and }\CT(\lambda_3(2,0))=x_1x_2.
\]
By Theorem \ref{weiXuTheo}(1),
\begin{align*}
\CT(\lambda_{12}(-1,9))&=\CT((\lambda_4(3,1),\lambda_3(2,0)))=\CT(\lambda_4(3,1))\divideontimes\CT(\lambda_3(2,0)) \\
&=x_2^2\divideontimes x_1x_2=x_2^6.
\end{align*}
Finally, by Lemma \ref{polyaLem},
\begin{align*}
&\CT(f_{\omega}((\omega^5,\omega^{21}),(7,5)))=\CT(((0,1),(\lambda_{12}(5,1),\lambda_{12}(7,2)))) \\
&=\CT^{(2)}(\lambda_{12}(5,1)\cdot\lambda_{12}(7,2))=\CT^{(2)}(\lambda_{12}(-1,9))=x_4^6.
\end{align*}
That is, the index $2$ generalized cyclotomic permutation of $\IF_{25}$ represented by $P$ consists of $6$ cycles of length $4$ (on $\IF_{25}^{\ast}$) as well as the trivial fixed point $0$.
\end{exxample}

\subsection{Classification of long-cycle elements and of involutions up to conjugacy}\label{subsec5P2}

Let $d$ be a positive integer and let $q$ be a prime power with $d\mid q-1$. Moreover, let $G$ be one of the permutation groups $\GCP(d,q)$, $\CP(d,q)$ or $\FOCP(d,q)$ on $\IF_q^{\ast}$ (defined in Subsection \ref{subsec1P1}). The goal of this subsection is to classify the conjugacy classes in $G$ whose elements are $(q-1)$-cycles (so-called \emph{long-cycle elements of $G$}) as well as the conjugacy classes in $G$ whose elements are involutions (i.e., elements $g\in G$ with $g^2=1_G$). We will do so by specifying complete systems of conjugacy class representatives in cyclotomic form. Our interest in this classification problem stems from the significance of long-cycle elements and involutions in practical applications such as pseudorandom number generation and cryptography, see \cite{CMS16a,NLQW20a,ZYLHZ19a} for involutions and \cite[beginning of Section 7.2 on p.~164]{Nie92a} for long cycles.

As in Subsection \ref{subsec5P1}, we consider the slightly more general problem of classifying conjugacy classes of long-cycle elements and of involutions in the groups $W(d,m)$, $W_=(d,m)$ and $W_1(d,m)$ from Notation \ref{wGroupNot}. The following result, which characterizes long-cycle elements and involutions in $W(d,m)$, is useful in all three cases:

\begin{lemmma}\label{characLem}
Let $d$ and $m$ be positive integers, and let
\[
g=(\psi,(g_0,g_1,\ldots,g_{d-1}))\in W(d,m)
\]
with $g_i=\lambda(a_i,b_i)\in\Hol(\IZ/m\IZ)$ for $i=0,1,\ldots,d-1$.
\begin{enumerate}
\item $g$ is a $(dm)$-cycle if and only if $\psi=(i_0,i_1,\ldots,i_{d-1})$ is a $d$-cycle, written such that $i_0=0$, and the forward cycle product
\[
g_{i_0}g_{i_1}\cdots g_{i_{d-1}}=\prod_{j=0}^{d-1}{\lambda(a_{i_j},b_{i_j})}=\lambda(\prod_{j=0}^{d-1}{a_j},\sum_{j=0}^{d-1}{(\prod_{k=j+1}^{d-1}{a_{i_k}})b_{i_j}})
\]
is a long-cycle element of $\Hol(\IZ/m\IZ)$, i.e.,
\begin{enumerate}
\item $\prod_{j=0}^{d-1}{a_j}\equiv1\Mod{\rad'(m)}$ and
\item $\gcd(m,\sum_{j=0}^{d-1}{(\prod_{k=j+1}^{d-1}{a_{i_k}})b_{i_j}})=1$.
\end{enumerate}
\item $g$ is an involution if and only if
\begin{equation}\label{psiEq}
\psi=(i_0,i_1)(i_2,i_3)\cdots(i_{2k-2},i_{2k-1})(i_{2k})(i_{2k+1})\cdots(i_{d-1})
\end{equation}
is an involution in $\Sym(d)$ and the following hold:
\begin{enumerate}
\item For $j=0,1,\ldots,k$, we have that $g_{i_{2j+1}}=g_{i_{2j}}^{-1}$ (equivalently, $\fcp_{(i_{2j},i_{2j+1})}(g)=\id_{\IZ/m\IZ}$).
\item For $j=2k,2k+1,\ldots,d-1$, we have that $g_{i_j}=\fcp_{(i_j)}(g)$ is an involution in $\Hol(\IZ/m\IZ)$.
\end{enumerate}
\end{enumerate}
\end{lemmma}

\begin{proof}
For statement (1): If $g$ is a $(dm)$-cycle, then it must in particular permute the $d$ copies $\IZ/m\IZ\times\{i\}$ of $\IZ/m\IZ$ cyclically, i.e., $\psi$ must be a $d$-cycle. Moreover, by Lemma \ref{polyaLem}, the only forward cycle product $\fcp_{\psi}(g)$ of $g$ must be an $m$-cycle on $\IZ/m\IZ$, which implies the number-theoretic conditions (a) and (b) by Knuth's Theorem \ref{knuthTheo}. Conversely, Theorem \ref{knuthTheo} and Lemma \ref{polyaLem} also imply that $g$ is a $(dm)$-cycle if it is of the specified form.

For statement (2): We use the semidirect product structure of $W(d,m)=\Sym(d)\ltimes\Hol(\IZ/m\IZ)^d$ (with $\Sym(d)$ acting on $\Hol(\IZ/m\IZ)^d$ by coordinate permutations). First, in view of the projection $W(d,m)\rightarrow\Sym(d)$ with $(\psi,(g_0,\ldots,g_{d-1}))\mapsto\psi$, we see that a necessary condition for $g$ to be an involution in $W(d,m)$ is that $\psi$ is an involution in $\Sym(d)$. Let $\psi$ be written in the form (\ref{psiEq}), and for ease of notation, assume that $i_j=j$ for $j=0,1,\ldots,d-1$. Then
\begin{align}\label{squareEq}
\notag &g^2=\psi(g_0,\ldots,g_{d-1})\psi(g_0,\ldots,g_{d-1})=\psi^2(g_0,\ldots,g_{d-1})^{\psi}(g_0,\ldots,g_{d-1})= \\
\notag &(g_{\psi^{-1}(0)},g_{\psi^{-1}(1)},\ldots,g_{\psi^{-1}(d-1)})(g_0,\ldots,g_{d-1})= \\
\notag &(g_1,g_0,g_3,g_2,\ldots,g_{2k-1},g_{2k-2},g_{2k},g_{2k+1},\ldots,g_{d-1})(g_0,\ldots,g_{d-1})= \\
&(g_1g_0,g_0g_1,g_3g_2,g_2g_3,\ldots,g_{2k-1}g_{2k},g_{2k}g_{2k-1},g_{2k+1}^2,\ldots,g_{d-1}^2).
\end{align}
In summary, $g^2\in\Hol(\IZ/m\IZ)^d$, and the $i$-th entry of $g^2$ is either $g_i^2$ (in case $i^{\psi}=i$) or $g_{\psi(i)}g_i$ (in case $i$ lies on a $2$-cycle of $\psi$). Since all entries of $g^2$ must be equal to the trivial element $\lambda(1,0)$, statements (a) and (b) follow. Conversely, if $g$ is of the specified form, it is easy to check (using formula (\ref{squareEq})) that $g$ is an involution in $W(d,m)$.
\end{proof}

The following notion will be useful when formulating our classification results:

\begin{deffinition}\label{lexicographicDef}
Let $m$ be a positive integer. We can write each element of $\Hol(\IZ/m\IZ)$ in a unique way as $\lambda(a,b)$ where $a\in\{0,1,\ldots,m-1\}$ with $\gcd(m,a)=1$ and $b\in\{0,1,\ldots,m-1\}$. Assume that $\vec{\iota}=(\lambda(a_0,b_0),\ldots,\lambda(a_{k-1},b_{k-1}))$ is a $k$-tuple of elements of $\Hol(\IZ/m\IZ)$. We say that $\vec{\iota}$ is \emph{lexicographically ordered} if
\[
(a_0,b_0)\leq_{\lex}(a_1,b_1)\leq_{\lex}\cdots\leq_{\lex}(a_{k-1},b_{k-1}),
\]
where $\leq_{\lex}$ is the lexicographical ordering on $\IZ^2$: $(a,b)\leq_{\lex}(a',b')$ if and only if $a<a'$ or $a=a'$ and $b\leq b'$.
\end{deffinition}

As for the classification results, we will start with the groups $W(d,m)$ and $W_1(d,m)$, as they are easy to deal with modulo known results due to their wreath product structure (see Subsection \ref{subsec4P1}). For $W(d,m)$, we have the following result:

\begin{propposition}\label{wdmProp}
Let $d$ and $m$ be positive integers.
\begin{enumerate}
\item The following is a complete system of representatives of the conjugacy classes of long-cycle elements in $W(d,m)$: For each
\[
a\in\{1+k\rad'(m): k=0,1,\ldots,\frac{m}{\rad'(m)}-1\},
\]
the element
\[
L_a:=((0,1,\ldots,d-1),(\lambda(a,1),\lambda(1,0),\lambda(1,0),\ldots,\lambda(1,0))).
\]
\item The following is a complete system of representatives of the conjugacy classes of involutions in $W(d,m)$: For each $k=0,1,\ldots,\lfloor\frac{d}{2}\rfloor$ and each lexicographically ordered $(d-2k)$-tuple $\vec{\iota}=(\iota_{2k},\iota_{2k+1},\ldots,\iota_{d-1})$ of involution conjugacy class representatives in $\Hol(\IZ/m\IZ)$, the element
\begin{align*}
&I_{k,\vec{\iota}}:= \\
&((0,1)(2,3)\cdots(2k-2,2k-1)(2k)(2k+1)\cdots(d-1), \\
&(\lambda(1,0),\ldots,\lambda(1,0),\iota_{2k},\ldots,\iota_{d-1})).
\end{align*}
\end{enumerate}
\end{propposition}

\begin{proof}
For statement (1): Let $g=(\psi,(g_0,\ldots,g_{d-1}))\in W(d,m)$ be a long cycle element. We show that $g$ is conjugate to precisely one of the elements $L_a$. By Lemma \ref{characLem}(1), $\fcp_{\psi}(g)$, which we write as $\lambda(a,b)$, is a long-cycle element of $\Hol(\IZ/m\IZ)$, so $a\equiv1\Mod{\rad'(m)}$ by Theorem \ref{knuthTheo}, whence $a=1+k\rad'(m)$ for a unique $k\in\{0,1,\ldots,\frac{m}{\rad'(m)}-1\}$. Moreover, $b$ must be a generator of $\IZ/m\IZ$, and by the discussion in Subsection \ref{subsec4P2}, $\lambda(a,b)$ and $\lambda(a,1)$ are conjugate in $\Hol(\IZ/m\IZ)$. Since $\fcp_{(0,\ldots,d-1)}(L_a)=\lambda(a,1)$, it follows by Lemma \ref{wpcLem} that $g$ is conjugate to $L_a$ in $W(d,m)$. Moreover, $g$ cannot be conjugate to $L_{a'}$ for any $a'\not=a$, since $L_{a'}$ has the unique forward cycle product $\lambda(a',1)$, which is non-conjugate to $\lambda(a,b)$ by the discussion in Subsection \ref{subsec4P2}.

For statement (2): Let $g=(\psi,(g_0,\ldots,g_{d-1}))\in W(d,m)$ be an involution. We show that $g$ is conjugate to precisely one of the elements $I_{k,\vec{\iota}}$. By Lemma \ref{characLem}(2), $\psi$ is an involution in $\Sym(d)$. Let $k$ be the number of $2$-cycles of $\psi$. Then, using the notation of Definition \ref{bcpcDef}(2), $M_1(g)$ is a multiset of $d-2k$ involution conjugacy classes in $\Hol(\IZ/m\IZ)$, and $M_2(g)$ is a multiset of cardinality $k$ whose only element is the trivial conjugacy class in $\Hol(\IZ/m\IZ)$. We have that $\psi$ is $\Sym(d)$-conjugate to $(0,1)(2,3)\cdots(2k-2,2k-1)(2k)(2k+1)\cdots(d-1)$, and there is a unique lexicographically ordered $(d-2k)$-tuple $\vec{\iota}=(\iota_{2k},\ldots,\iota_{d-1})$ of involution conjugacy class representatives in $\Hol(\IZ/m\IZ)$ such that $M_1(g)=\{\iota_{2k}^{\Hol(\IZ/m\IZ)},\ldots,\iota_{d-1}^{\Hol(\IZ/m\IZ)}\}$ (as multisets). It follows that $M_1(g)=M_1(I_{k,\vec{\iota}})$ and $M_2(g)=M_2(I_{k,\vec{\iota}})$, whence $g$ is $W(d,m)$-conjugate to $I_{k,\vec{\iota}}$ by Lemma \ref{wpcLem}. Moreover, any different choice of $(k',\vec{\iota}')$ leads to an element $I_{k',\vec{\iota}'}=(\psi',(g_0',\ldots,g_{d-1}'))$ to which $g$ is non-conjugate, either because $\psi$ is non-conjugate to $\psi'$ (in case $k'\not=k$), or because $M_1(g)\not=M_1(I_{k',\vec{\iota}'})$ (in case $k=k'$ but $\vec{\iota}\not=\vec{\iota}'$).
\end{proof}

\begin{corrollary}\label{wdmCor}
Let $d$ be a positive integer, and let $q$ be a prime power with $d\mid q-1$. Fix a primitive root $\omega$ of $\IF_q$, denote by $C$ the index $d$ subgroup of $\IF_q^{\ast}$, and set $C_i:=\omega^iC$ for $i=0,1,\ldots,d-1$.
\begin{enumerate}
\item The following is a complete system of representatives of the conjugacy classes of long-cycle elements in $\GCP(d,q)$: For each
\[
a\in\{1+k\rad'(\frac{q-1}{d}): k=0,1,\ldots,\frac{(q-1)/d}{\rad'((q-1)/d)}-1\},
\]
the permutation $L_a:\IF_q^{\ast}\rightarrow\IF_q^{\ast}$ with
\[
L_a(x)=\begin{cases}\omega x, & \text{if }x\notin C_{d-1}, \\ \omega^{d-(d-1)a}x^a, & \text{if }x\in C_{d-1}.\end{cases}
\]
\item The following is a complete system of representatives of the conjugacy classes of involutions in $\GCP(d,q)$: For each $k=0,1,\ldots,\lfloor\frac{d}{2}\rfloor$ and each lexicographically ordered $(d-2k)$-tuple $(\lambda(a_j,b_j))_{j=2k,2k+1,\ldots,d-1}$ of involution conjugacy class representatives in $\Hol(\IZ/\frac{q-1}{d}\IZ)$, the permutation $I_{k,\vec{\iota}}:\IF_q^{\ast}\rightarrow\IF_q^{\ast}$ with
\[
I_{k,\vec{\iota}}=\begin{cases}\omega x, & \text{if }x\in\bigcup_{i=0}^{k-1}{C_{2i}}, \\ \omega^{-1}x, & \text{if }x\in\bigcup_{i=0}^{k-1}{C_{2i+1}}, \\ \omega^{db_j+(1-a_j)d}x^{a_j}, & \text{if }x\in C_j\text{ for some }j=2k,2k+1,\ldots,d-1.\end{cases}
\]
\end{enumerate}
\end{corrollary}

\begin{proof}
Both statements follow by setting $m:=\frac{q-1}{d}$ in Proposition \ref{wdmProp} and mapping the elements
\[
L_a,I_{k,\vec{\iota}}\in W(d,\frac{q-1}{d})=\Hol(\IZ/\frac{q-1}{d}\IZ)\wr_{\imp}\Sym(d)
\]
under the permutation group isomorphism $\iota_{\omega}$ from Theorem \ref{wreathIsoTheo}. More precisely, before applying $\iota_{\omega}$, we apply the permutation group isomorphism
\[
\Hol(\IZ/\frac{q-1}{d}\IZ)\wr_{\imp}\Sym(d) \rightarrow \Hol(C)\wr_{\imp}\Sym(d)
\]
that is induced by the abstract group isomorphism
\[
\IZ/\frac{q-1}{d}\IZ\rightarrow C, (x+\frac{q-1}{d}\IZ)\mapsto\omega^{dx}.
\]
In particular, when switching from an affine permutation $\lambda(a,b)$ of $\Hol(\IZ/\frac{q-1}{d}\IZ)$ to the corresponding affine permutation of $C$, it becomes $\lambda(a,\omega^{db}):x\mapsto \omega^{bd}x^a$. We leave the computational details of applying this rewriting process to the $L_a$ and $I_{k,\vec{\iota}}$ as an exercise to the reader.
\end{proof}

For $W_1(d,m)$, we have the following analogue of Proposition \ref{wdmProp}, whose (analogous) proof we omit:

\begin{propposition}\label{w1dmProp}
Let $d$ and $m$ be positive integers.
\begin{enumerate}
\item There is precisely one conjugacy class of long-cycle elements in $W_1(d,m)$, with representative
\[
L^{(1)}:=((0,1,\ldots,d-1),(\lambda(1,1),\lambda(1,0),\ldots,\lambda(1,0))).
\]
\item The following is a complete system of representatives of the conjugacy classes of involutions in $W_1(d,m)$: For each $k=0,1,\ldots,\lfloor\frac{d}{2}\rfloor$ and each $(d-2k)$-tuple
\[
\vec{b}=(b_{2k},\ldots,b_{d-1})\in\{b\in\IZ/m\IZ: 2b=0\}^{d-2k},
\]
ordered such that the zero entries of $\vec{b}$ come before the nonzero entries, the element
\begin{align*}
&I_{k,\vec{b}}^{(1)}= \\
&((0,1)\cdots(2k-2,2k-1)(2k)\cdots(d-1), \\
&(\lambda(1,0),\ldots,\lambda(1,0),\lambda(1,b_{2k}),\ldots,\lambda(1,b_{d-1}))).
\end{align*}
\end{enumerate}
\end{propposition}

Likewise, we have the following analogue of Corollary \ref{wdmCor}:

\begin{corrollary}\label{w1dmCor}
Let $d$ be a positive integer, and let $q$ be a prime power with $d\mid q-1$. Fix a primitive root $\omega$ of $\IF_q$, denote by $C$ the index $d$ subgroup of $\IF_q^{\ast}$, and set $C_i:=\omega^iC$ for $i=0,1,\ldots,d-1$.
\begin{enumerate}
\item There is precisely one conjugacy class of long-cycle elements in $\FOCP(d,q)$, with representative $L^{(1)}:\IF_q^{\ast}\rightarrow\IF_q^{\ast}$, $x\mapsto\omega x$.
\item The following is a complete system of representatives of the conjugacy classes of involutions in $\FOCP(d,q)$: For each $k=0,1,\ldots,\lfloor\frac{d}{2}\rfloor$ and each $(d-2k)$-tuple
\[
\vec{c}=(c_{2k},\ldots,c_{d-1})\in\{c\in C: c^2=1\},
\]
ordered such that the trivial entries $1$ come before the nontrivial entries, the permutation $I_{k,\vec{c}}^{(1)}:\IF_q^{\ast}\rightarrow\IF_q^{\ast}$ with
\[
I_{k,\vec{\iota}}^{(1)}(x)=\begin{cases}\omega x, & \text{if }x\in\bigcup_{i=0}^{k-1}{C_{2i}}, \\ \omega^{-1}x, & \text{if }x\in\bigcup_{i=0}^{k-1}{C_{2i+1}}, \\ c_jx, & \text{if }x\in C_j\text{ for some }j=2k,2k+1,\ldots,d-1.\end{cases}
\]
\end{enumerate}
\end{corrollary}

Now we will discuss the groups $W_=(d,m)$ and $CP(d,q)$. We need the following analogue of Lemma \ref{wpcLem}:

\begin{lemmma}\label{wpcLem2}
Let $d$ and $m$ be positive integers. Moreover, let $\sigma,\sigma'\in\Sym(d)$, let $a,a'\in(\IZ/m\IZ)^{\ast}$, and let $b_i,b'_i\in\IZ/m\IZ$ for $i=0,1,\ldots,d-1$. The following are equivalent:
\begin{enumerate}
\item The two elements
\[
g=(\sigma,(\lambda(a,b_0),\ldots,\lambda(a,b_{d-1})))
\]
and
\[
g'=(\sigma',(\lambda(a,b'_0),\ldots,\lambda(a,b'_{d-1})))
\]
of $W_=(d,m)$ are conjugate in $W(d,m)$.
\item $\sigma$ and $\sigma'$ are $\Sym(d)$-conjugate, $a=a'$, and for all $\ell\in\{1,\ldots,d\}$, the equality of multisets $M_{\ell}(g)=M_{\ell}(g')$ holds.
\end{enumerate}
\end{lemmma}

\begin{proof}
For \enquote{(1)$\Rightarrow$(2)}: Since $g$ and $g'$ are in particular $W(d,m)$-conjugate, it follows by Lemma \ref{wpcLem} that $\sigma$ and $\sigma'$ must be $\Sym(d)$-conjugate and that $M_{\ell}(g)=M_{\ell}(g')$ for each $\ell$. In order to see that $a=a'$, let
\[
h=(\psi,(\lambda(c,e_0),\lambda(c,e_1),\ldots,\lambda(c,e_{d-1})))\in W_=(d,m).
\]
Then
\begin{align}\label{conjugateEq}
\notag g^h &=(g^{\psi})^{(\lambda(c,e_i))_{i=0,\ldots,d-1}}=(\sigma^{\psi},(\lambda(a,b_{\psi^{-1}(i)}))_{i=0,\ldots,d-1})^{(\lambda(c,e_i))_{i=0,\ldots,d-1}} \\
&=(\sigma^{\psi},(\lambda(c,e_{(\sigma^{\psi})^{-1}(i)})^{-1}\cdot\lambda(a,b_{\psi^{-1}(i)})\cdot\lambda(c,e_i))_{i=0,\ldots,d-1}).
\end{align}
Since $\Hol(\IZ/m\IZ)=(\IZ/m\IZ)^{\ast}\ltimes{\IZ/m\IZ}$ and $(\IZ/m\IZ)^{\ast}$ is abelian, we conclude that for each $i=0,1,\ldots,d-1$,
\[
\lambda(c,e_{(\sigma^{\psi})^{-1}(i)})^{-1}\cdot\lambda(a,b_{\psi^{-1}(i)})\cdot\lambda(c,e_i)=\lambda(c^{-1}ac,f_i)=\lambda(a,f_i)
\]
for a suitable $f_i\in\IZ/m\IZ$. Hence $a'=a$ if $g'$ is to be conjugate to $g$ in $W_=(d,m)$.

For \enquote{(2)$\Rightarrow$(1)}: Since $\sigma$ and $\sigma'$ are conjugate in $\Sym(d)$, we may, through replacing $g'$ by $(g')^{(\psi,(\lambda(1,0),\ldots,\lambda(1,0)))}$ for a suitable $\psi\in\Sym(d)$, assume that $\sigma'=\sigma$. Moreover, through replacing $g'$ by $(g')^{(\psi,(\lambda(1,0),\ldots,\lambda(1,0)))}$ for a suitable $\psi$ in the centralizer of $\sigma$ in $\Sym(d)$ (through which we can shuffle cycles of $\sigma$ of the same length arbitrarily), we may assume that $\fcp_{\zeta}(g)=\fcp_{\zeta}(g')$ for each cycle $\zeta$ of $\sigma$. It remains to show that under these additional assumptions, we can choose $\vec{\lambda}\in\Hol(\IZ/m\IZ)^d$ suitably such that $g^{(\id,\vec{\lambda})}=g'$. Note that for each cycle $\zeta$ of $\sigma$, whether the $i$-th components of $g^{(\id,\vec{\lambda})}$ and $g'$ are equal for all indices $i$ on $\zeta$ only depends on those components of $\vec{\lambda}$ that correspond to indices on $\zeta$. We thus aim to construct $\vec{\lambda}$ \enquote{cycle-wise}, and for notational simplicity, we may assume that $\sigma=(0,1,\ldots,d-1)$. Write $\vec{\lambda}=(\lambda(c,e_i))_{i=0,\ldots,d-1}$. Then by formula (\ref{conjugateEq}),
\[
g^{(\id,\vec{\lambda})}=(\sigma,(\lambda(c,e_{\sigma^{-1}(i)})^{-1}\lambda(a,b_i)\lambda(c,e_i))_{i=0,\ldots,d-1}),
\]
and
\begin{align*}
\lambda(c,e_{\sigma^{-1}(i)})^{-1}\lambda(a,b_i)\lambda(c,e_i)&=\lambda(c^{-1},-c^{-1}e_{\sigma^{-1}(i)})\lambda(ac,cb_i+e_i) \\
&=\lambda(a,cb_i+e_i-ae_{\sigma^{-1}(i)}).
\end{align*}
Hence
\[
g^{(\id,\vec{\lambda})}=(\sigma,(\lambda(a,cb_0+e_0-ae_{d-1}),\lambda(a,cb_1+e_1-ae_0),\ldots,\lambda(a,cb_{d-1}+e_{d-1}-ae_{d-2}))).
\]
For simplicity, set $c:=1$ (as we will see, this is still strong enough to achieve $g^{(\id,\vec{\lambda})}=g'$). Then
\[
g^{(\id,\vec{\lambda})}=(\sigma,(\lambda(a,b_0+(e_0-ae_{d-1})),\lambda(a,b_1+(e_1-ae_0)),\ldots,\lambda(a,b_{d-1}+(e_{d-1}-ae_{d-2})))).
\]
Note that we may choose $e_0,\ldots,e_{d-1}\in\IZ/m\IZ$ arbitrarily. This allows us to make all but one of the second entries $b_i+(e_i-ae_{\sigma^{-1}(i)})$ assume arbitrary values. Say we choose $\vec{\lambda}$ such that $g^{(\id,\vec{\lambda})}$ agrees with $g'$ in all components of their respective second entries except the last. Then it follows that $g^{(\id,\vec{\lambda})}=g'$ -- the last components must be equal as well, because of $\fcp_{\sigma}(g^{(\id,\vec{\lambda})})=\fcp_{\sigma}(g')$.
\end{proof}

Using Lemma \ref{wpcLem2}, one can prove the following analogue of Proposition \ref{wdmProp}:

\begin{propposition}\label{wdmeProp}
Let $d$ and $m$ be positive integers.
\begin{enumerate}
\item The following is a complete system of representatives of the conjugacy classes of long-cycle elements in $W_=(d,m)$: For each $a\in(\IZ/m\IZ)^{\ast}$ such that $a^d\equiv1\Mod{\rad'(m)}$, the element
\[
L_a^{(=)}:=((0,\ldots,d-1),(\lambda(a,1),\lambda(a,0),\ldots,\lambda(a,0))).
\]
\item The following is a complete system of representatives of the conjugacy classes of involutions in $W_=(d,m)$: For each $k=0,1,\ldots,\lfloor\frac{d}{2}\rfloor$, each involution $a\in(\IZ/m\IZ)^{\ast}$ and each lexicographically ordered $(d-2k)$-tuple
\[
\vec{\iota}=(\lambda(a,b_{2k}),\ldots,\lambda(a,b_{d-1}))
\]
of involution conjugacy class representatives in $\Hol(\IZ/m\IZ)$, the element
\begin{align*}
&I_{k,a,\vec{\iota}}^{(=)}:= \\
&((0,1)\cdots(2k-1,2k)(2k+1)\cdots(d-1), \\
&(\lambda(a,0),\ldots,\lambda(a,0),\lambda(a,b_{2k}),\ldots,\lambda(a,b_{d-1}))).
\end{align*}
\end{enumerate}
\end{propposition}

This leads to the following analogue of Corollary \ref{wdmCor}:

\begin{corrollary}\label{wdmeCor}
Let $d$ be a positive integer, and let $q$ be a prime power with $d\mid q-1$. Fix a primitive root $\omega$ of $\IF_q$, let $C$ be the index $d$ subgroup of $\IF_q^{\ast}$, and set $C_i:=\omega^iC$ for $i=0,1,\ldots,d-1$.
\begin{enumerate}
\item The following is a complete system of representatives of the conjugacy classes of long-cycle elements of $\CP(d,q)$: For each $a\in(\IZ/\frac{q-1}{d}\IZ)^{\ast}$ such that $a^d\equiv1\Mod{\rad'(\frac{q-1}{d})}$, the permutation $L_a^{(=)}:\IF_q^{\ast}\rightarrow\IF_q^{\ast}$ with
\[
L_a^{(=)}(x)=\begin{cases}\omega^{i+1-ia}x^a, & \text{if }x\in C_i\text{ for some }i=0,1,\ldots,d-2, \\ \omega^ax^a, & \text{if }x\in C_{d-1}.\end{cases}
\]
\item The following is a complete system of representatives of the conjugacy classes of involutions in $\CP(d,q)$: For each $k=0,1,\ldots,\lfloor\frac{d}{2}\rfloor$, each involution $a\in(\IZ/m\IZ)^{\ast}$ and each lexicographically ordered $(d-2k)$-tuple
\[
\vec{\iota}=(\lambda(a,b_{2k}),\ldots,\lambda(a,b_{d-1}))
\]
of involution conjugacy class representatives in $\Hol(\IZ/\frac{q-1}{d}\IZ)$, the permutation $I_{k,a,\vec{\iota}}^{(=)}:\IF_q^{\ast}\rightarrow\IF_q^{\ast}$ with
\[
I_{k,a,\vec{\iota}}^{(=)}(x)=\begin{cases}\omega^{i+1-ia}x^a, & \text{if }x\in C_i\text{ for some }i=0,2,\ldots,2k-2, \\ \omega^{i-(i+1)a}x^a, & \text{if }x\in C_i\text{ for some }i=1,3,\ldots,2k-1, \\ \omega^{i(1-a)+db_i}x^a, & \text{if }x\in C_i\text{ for some }i=2k,2k+1,\ldots,d-1.\end{cases}
\]
\end{enumerate}
\end{corrollary}

\subsection{Inversion of generalized cyclotomic permutations in polynomial form}\label{subsec5P3}

Let $d$ be a positive integer, let $q$ be a prime power with $d\mid q-1$, let $\omega$ be a primitive root of $\IF_q$, set $\zeta:=\omega^{\frac{q-1}{d}}$, let $C$ be the the index $d$ subgroup of $\IF_q^{\ast}$, set $C_i:=\omega^iC$ for $i=0,1,\ldots,d-1$, and let $f:\IF_q\rightarrow\IF_q$ be an index $d$ generalized cyclotomic permutation of $\IF_q$. Then the $\omega$-cyclotomic form $f_{\omega}(\vec{a},\vec{r})$ of $f$ is unique, because none of the entries of $\vec{a}$ can be $0$ (due to the injectivity of $f$). In \cite{ZYZP16a}, Zheng, Yu, Zhang and Pei gave the following formula for the polynomial form of the inverse function $f^{-1}$:

\begin{theoremm}\label{zyzpTheo}(Zheng-Yu-Zhang-Pei, see \cite[Theorem 3.3]{ZYZP16a})
With notation as above, let $\vec{a}=(a_0,\ldots,a_{d-1})$, let $\vec{r}=(r_0,\ldots,r_{d-1})$ and, using that $\gcd(r_i,\frac{q-1}{d})=1$ for each $i$, let $\tilde{r_i}\in\{1,2,\ldots,\frac{q-1}{d}\}$ and $t_i\in\IZ$ be such that $r_i\tilde{r_i}+\frac{q-1}{d}t_i=1$ (note that $\tilde{r_i}$ and $t_i$ are uniquely determined by this, with $\tilde{r_i}$ being the multiplicative inverse of $r_i$ modulo $\frac{q-1}{d}$). Then the polynomial form of the inverse function of $f$ is
\[
\frac{1}{d}\sum_{i,j=0}^{d-1}{\zeta^{i(t_i-jr_i)}a_i^{-\tilde{r_i}-j\frac{q-1}{d}}T^{\tilde{r_i}+j\frac{q-1}{d}}}.
\]
\end{theoremm}

We note that the second author in \cite{Wan17a} gave a different, shorter proof of Theorem \ref{zyzpTheo}, and using our Theorems \ref{wreathIsoTheo} and \ref{cycloPolyTheo}, one can give yet another proof of this formula, by explicitly converting from cyclotomic form to wreath product form, then applying the inversion of the wreath product group, and finally converting the result of that inversion back to cyclotomic form and then polynomial form. However, this proof would still be longer than the one in \cite{Wan17a}, and we refrain from providing the technial details of it here.

We conclude this paper by observing that through combining our Algorithm \ref{algo1} with Theorem \ref{zyzpTheo}, one can pass from the polynomial form of $f$ to that of $f^{-1}$, without needing to have explicit knowledge of the $\omega$-cyclotomic form of $f$. This is demonstrated in the following example:

\begin{exxample}\label{zyzpEx}
Let $d=2$, $q=25$, let $\omega$ be a root of the Conway polynomial $T^2-T+2\in\IF_5[T]$, and consider once more the polynomial
\[
P=\omega^{15}T^5+\omega^{23}T^7+\omega^3T^{17}+\omega^{23}T^{19}
\]
from Example \ref{unnumberedEx}. Assume we are given the task of deciding whether $P$ represents an index $2$ generalized cyclotomic permutation $f$ of $\IF_{25}$ and, if so, finding the polynomial form of $f^{-1}$. Following Example \ref{unnumberedEx}, we find that $P$ indeed represents an index $2$ generalized cyclotomic permutation of $\IF_{25}$, namely
\[
f=f_{\omega}((\omega^5,\omega^{21}),(7,5)).
\]
We can use Theorem \ref{zyzpTheo} to compute the polynomial form of $f^{-1}$. First, we determine $(\tilde{r_0},t_0)$ and $(\tilde{r_1},t_1)$ with the extended Euclidean algorithm:
\begin{align*}
12&=1\cdot7+5, \\
7&=1\cdot5+2, \\
5&=2\cdot2+1,
\end{align*}
whence
\begin{align*}
1&=5-2\cdot2 \\
&=5-2\cdot(7-1\cdot5)=3\cdot5-2\cdot7 \\
&=3\cdot(12-1\cdot7)-2\cdot7=3\cdot12-5\cdot7 \\
&=(3-7)\cdot 12+(-5+12)\cdot7=(-4)\cdot12+7\cdot7.
\end{align*}
We conclude that $\tilde{r_0}=7$ and $t_0=-4$. Analogously, one computes $\tilde{r_1}=5$ and $t_1=-2$. So, in summary, we have
\[
a_0=\omega^5,r_0=\tilde{r_0}=7,t_0=-4
\]
and
\[
a_1=\omega^{21},r_1=\tilde{r_1}=5,t_1=-2.
\]
Observing that $\zeta=\omega^{12}=-1$ and using Theorem \ref{zyzpTheo}, we conclude that the polynomial form of $f^{-1}$ is
\begin{align*}
&\frac{1}{2}\sum_{i,j=0}^1{(-1)^{i(t_i-jr_i)}a_i^{-\tilde{r_i}-12j}T^{\tilde{r_i}+12j}} \\
&=\frac{1}{2}(\omega^{-5\cdot7}T^7+\omega^{-5(7+12)}T^{19}+(-1)^{-2}\omega^{-21\cdot5}T^5+(-1)^7\omega^{-21(5+12)}T^{17}) \\
&=\omega^9T^5+\omega^7T^7+\omega^9T^{17}+\omega^{19}T^{19}.
\end{align*}
Note that it is just a coincidence (caused by $r_i=\tilde{r_i}$ for $i=0,1$) that the term degrees of $P$ and this polynomial are the same.
\end{exxample}

\end{document}